\documentclass[a4paper, 11pt, reqno]{amsart}

\usepackage[latin1]{inputenc}
\usepackage[T1]{fontenc}
\usepackage[english]{babel}
\usepackage{latexsym,amsmath,amssymb,amsfonts}
\usepackage{dsfont}
\usepackage{color}
\usepackage{graphicx}
\usepackage{float}
\usepackage{esint}
\usepackage{bm}
\usepackage{mathtools}
\usepackage[top=1in, bottom=1.25in, left=1in, right=1in]{geometry}
\usepackage{enumitem}
\usepackage{hyperref}
\usepackage{cleveref}
\usepackage{bbm}
\usepackage{tikz}                        
\usepackage{pgfplots}    
\usepackage{soul}         

\hypersetup{
  colorlinks = true,
  linkcolor = blue,
  citecolor = red
}

\theoremstyle{plain}
\begingroup
\newtheorem{theorem}{Theorem}[section]
\newtheorem{lemma}[theorem]{Lemma}
\newtheorem{proposition}[theorem]{Proposition}
\newtheorem{corollary}[theorem]{Corollary}
\endgroup
\theoremstyle{definition}
\begingroup
\newtheorem{definition}[theorem]{Definition}
\newtheorem{remark}[theorem]{Remark}

\endgroup
\theoremstyle{remark}

\newcommand{\N}{\mathbb N}
\newcommand{\Z}{\mathbb Z}

\newcommand{\R}{\mathbb R}

\renewcommand{\S}{\mathbb S}

\newcommand{\dist}{{\rm dist}}

\newcommand{\diam}{{\rm diam}}
\newcommand{\Lip}{{\rm Lip}}

\newcommand{\hno}{{\mathcal H}^{N-1}}

\newcommand{\dd}{d}

\newcommand{\dhno}{\,\dd{\mathcal H}^{N-1}}

\newcommand{\e}{\varepsilon}

\renewcommand{\o}{\Omega}

\def\ca{\mathbbmss{1}}

\newcommand{\f}{\mathcal{F}}

\newcommand{\restr}{%
  \,\raisebox{-.127ex}{\reflectbox{\rotatebox[origin=br]{-90}{$\lnot$}}}\,%
}

\newcommand{\average}{{\mathchoice {\kern1ex\vcenter{\hrule height.4pt
width 6pt
depth0pt} \kern-9.7pt} {\kern1ex\vcenter{\hrule height.4pt width 4.3pt
depth0pt}
\kern-7pt} {} {} }}

\title[Nonisothermal multi-phase transitions]{Sharp interface limit of a multi-phase transitions model under nonisothermal conditions}
\author{Riccardo Cristoferi, Giovanni Gravina}
\address{Heriot-Watt University \\ Department of Mathematical Sciences \\ Edinburgh EH14 4AS \\ United Kingdom}
\email{r.cristoferi@hw.ac.uk}
\address{Department of Mathematical Analysis\\Faculty of Mathematics and Physics\\ 
Charles University\\Prague\\ Czech Republic}
\email {gravina@karlin.mff.cuni.cz} 
\date{\today}

\begin{document}

\begin{abstract}
A vectorial Modica--Mortola functional is considered and the convergence to a sharp interface model is studied.
The novelty of the paper is that the wells of the potential are not constant, but depend on the spatial position in the domain $\o$.
The mass constrained minimization problem and the case of Dirichlet boundary conditions are also treated. The proofs rely on the precise understanding of minimizing geodesics for the degenerate metric induced by the potential.
\end{abstract}

\maketitle
\tableofcontents

\section{Introduction}

Phase transitions phenomena are ubiquitous in nature. Examples are the spinodal decomposition in metallic alloys, the change in the crystallographic structure in metals, the order-disorder transitions, and the alterations of the molecular structures. In view of the wide range of physical and industrial applications where phase transitions are observed, it is of primary interest to understand the different mechanisms  that govern these complex processes.
Many physical models have been proposed over the years to capture the behavior of these phenomena and an enormous amount of insight has been gained  by performing analytical studies.
For this reason, the theoretical investigation of phase transitions is still currently an active field of research in the mathematical community.
In the particular case of liquid-liquid phase transitions, the preferred model was proposed by van der Waals (see \cite{vanDW}) and was later independently rediscovered by Cahn and Hilliard (see \cite{CH}).
This theory revolves around the study of the so called Modica--Mortola energy functional (often referred to as the Ginzburg--Landau free energy in the physics literature), which is the foundation of the model we consider in this paper.

The primary focus of this work is the study of the $\Gamma$-convergence of the family of functionals
\[
\f_\e(u) \coloneqq \int_\o \left[\frac{1}{\e} W(x, u(x)) + \e|\nabla u(x)|^2 \right]\,dx,
\]
where $u \in H^1(\o; \R^M)$, with $M \ge 1$, and $W \colon \o \times \R^M \to [0,\infty)$ is a locally Lipschitz potential such that, for all $x \in \o$, $W(x, p) = 0$ if and only if $p \in \{z_1(x),\dots,z_k(x)\}$. Here $\Omega$ denotes an open bounded subset of $\R^N$ with Lipschitz  continuous boundary and, for $i \in \{1, \dots, k\}$, the $z_i \colon \o \to \R^M$ are given Lipschitz functions.

Our main contribution is the treatment of the case $M \geq 2$ for $x$-dependent wells, thus providing a first vectorial counterpart to some of the results in \cite{Bou, Ste_Sing}, where moving wells were considered in the scalar case. For the precise statement of our results we refer the reader to \Cref{statementMR}.

\begin{center}
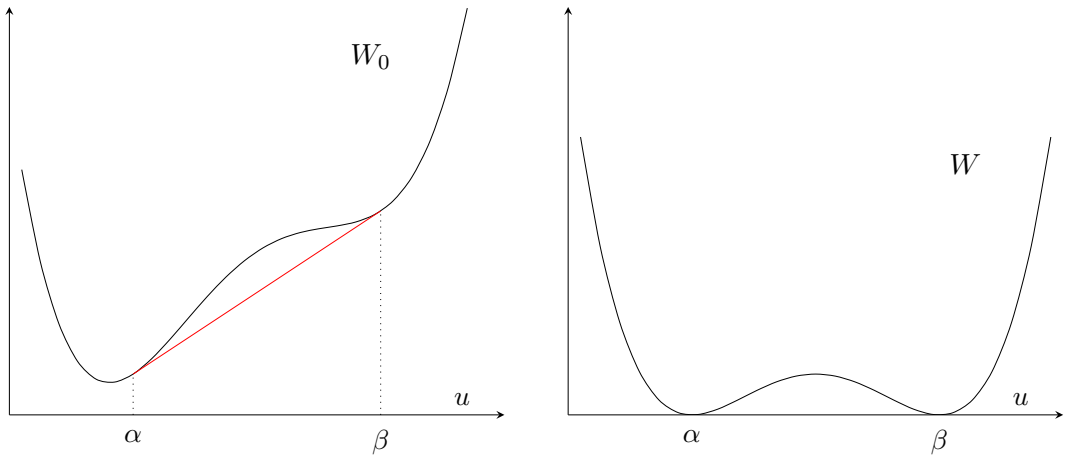
\begin{figure}[t]\centering
\begin{tikzpicture}[blend group=normal, scale=0.95]
\begin{axis}[axis x line = middle, xmin = 0, xmax = 4, axis y line = middle, ymax = 10, ymin = 0, xlabel=$u$, ylabel=$$, xtick={1, 3}, xticklabels={$\alpha$, $\beta$}, xtick style={draw=none}, ytick style={draw=none}, yticklabels={},]           
             \addplot[smooth, domain=0.1:3.7] {(x - 1)^2*(x - 3)^2 + 2*x - 1};
             \addplot[smooth, domain=1:3, red, thin] {2*x - 1};
             \addplot[smooth, domain=0:1, dotted] (1,x);
             \addplot[smooth, domain=0:5, dotted] (3,x);
\end{axis}
\node at (5,5) {$W_0$};
\end{tikzpicture}
\quad
\begin{tikzpicture}[blend group=normal, scale=0.95]
\begin{axis}[axis x line = middle, xmin = 0, xmax = 4, axis y line = middle, ymax = 10, ymin = 0, xlabel=$u$, ylabel=$$, xtick={1, 3}, xticklabels={$\alpha$, $\beta$}, xtick style={draw=none}, ytick style={draw=none}, yticklabels={},]        
             \addplot[smooth, domain=0.1:3.9] {(x - 1)^2*(x - 3)^2};
\end{axis}
\node at (5.5,3.5) {$W$};
\end{tikzpicture}
         
\caption{On the left: the typical profile of the potential $W_0$ together with its convex envelope; the region where the two do not coincide is highlighted in red. On the right: the potential $W$, obtained by subtracting a linear term with slope $W_0'(\alpha)$ from $W_0$.}
\label{fig:W}
\end{figure}
\end{center}

\subsection{Previous works} Denote by $\o \subset \R^N$ the container of the material, and assume that 
the system is described by a scalar valued phase (or order) parameter $u \colon \o \to \R$, which for instance, in the case of a mixture of two or more fluids, represents the density.
Stable equilibrium configurations are local minimizers of the Gibbs free energy.
Under \emph{isothermal conditions}, consider
\begin{equation}\label{eq:enW}
\int_\o W_0(u(x))\,dx,
\end{equation}
where the free energy density $W_0\colon \R \to [0,\infty)$ is taken to be non-convex in order to support a phase transitions.
If the material has two stable phases, the typical form of $W_0$ is depicted in Figure \ref{fig:W}.
In many situations, the physical interpretation of the phase parameter naturally imposes a constraint on the class of admissible functions for the minimization problem for \eqref{eq:enW}.
If $u$ represents a density, this often takes the form of a volume constraint, \emph{i.e.},
\begin{equation}\label{eq:mcW}
\int_\o u(x)\,dx = m,
\end{equation}
for some $m\in\R$.
For $W_0$ as in Figure \ref{fig:W}, let $(\alpha,\beta)$ be the interval where $W_0$ does not coincide with its convex envelope.
To be precise, $\alpha$ and $\beta$ are chosen to satisfy
\[
W_0(\beta) - W_0(\alpha) = W'_0(\alpha)(\beta - \alpha),
\quad\quad\quad
W'_0(\alpha) = W'_0(\beta).
\]
The numbers $\alpha, \beta,\mu$, where $\mu \coloneqq W'_0(\alpha)$, are called \emph{Maxwells parameters} (see \cite{Gur_Interfacial}). Notice that if 
\[
m \in \left(-\infty,\alpha \mathcal{L}^N(\o) \right) \cup \left(\beta\mathcal{L}^N(\o),\infty \right)
\]
then there is no phase transitions, in that solutions to the minimization problem for \eqref{eq:enW} subject to the constraint in \eqref{eq:mcW} are constant. Here $\mathcal{L}^N(\o)$ denotes the volume of the set $\o$. Therefore, we assume that
\[
m \in \left(\alpha \mathcal{L}^N(\o),\beta \mathcal{L}^N(\o) \right)
\]
and restrict our attention to admissible functions $u$ taking values in the interval $[\alpha,\beta]$.
Define
\[
W(u) \coloneqq W_0(u) - W'_0(\alpha)(u-\beta) - W_0(\beta) . 
\]
Notice that $W(\alpha) = W(\beta) = 0$, that $W > 0$ otherwise, and that, in view of the mass constraint \eqref{eq:mcW}, replacing $W_0$ with $W$ in (\ref{eq:enW}) changes the free energy only by a constant.

As previously remarked by several authors, an energy of the form \eqref{eq:enW} cannot properly describe the physics of phase transitions. Indeed, given \emph{any} region $A \subset \o$ with
\[
\mathcal{L}^N(A) = \frac{\beta - m}{\beta - \alpha}\mathcal{L}^N(\o),
\]
the phase variable which takes the value $\alpha$ in $A$ and $\beta$ in $\o\setminus A$  satisfies the mass constraint \eqref{eq:mcW} and is a minimizer of the energy \eqref{eq:enW}.
This is not what it is observed in experiments, where for stable configurations the two phases are separated by an interface with minimal surface area.
Therefore, in order to capture this behavior, a term that penalizes the creation of interfaces has to be added to the energy.
Indeed, in the van der Waals--Cahn--Hilliard theory of phase transitions the following functional is considered
\begin{equation}\label{eq:Wgrad}
\int_\o \left[W(u(x)) + \e^2|\nabla u(x)|^2\right]\,dx.
\end{equation}
One can justify heuristically the choice of the \emph{singular perturbation} in \eqref{eq:Wgrad} by considering the idealized situation where the transition between the two phases takes place in a layer  $\Sigma_{\e}$, representing the diffuse separating interface. This is assumed to be an $\e$-tubular neighborhood of an $(N-1)$-dimensional surface $\Sigma$. In this case
\begin{align}\label{eq:heuristic}
\int_\o \left[ W(u(x)) + \e^2|\nabla u(x)|^2 \right]\,dx & = \int_{\Sigma_\e} \left[ W(u(x)) + \e^2|\nabla u(x)|^2 \right]\,dx \nonumber \\
& \sim \e \hno(\Sigma) \left[ 1 + \e^2 \frac{|\beta-\alpha|^2}{\e^2} \right],
\end{align}
where the last estimate is obtained by assuming that both $u$ and $|\nabla u|$ are bounded.
Here with $\hno(\Sigma)$ we denote the surface area of $\Sigma$. Therefore, in order to have an energy of order $1$ we need to rescale the functional by a factor of $1/\e$. Hence we consider
\begin{equation}\label{eq:rescaled}
\mathcal{E}_\e(u)\coloneqq\int_\o \left[ \frac{1}{\e} W(u(x)) + \e|\nabla u(x)|^2 \right]\,dx.
\end{equation}

It was conjectured by Gurtin (see \cite{Gur}) that in the limit as $\e\to0$, minimizers $\{u_\e\}_{\e > 0}$ of \eqref{eq:rescaled} subject to the constraint \eqref{eq:mcW} converge to a piecewise constant function $u$ that partitions $\o$ into two regions separated by an interface with minimal surface area.
This conjecture was proved rigorously by Carr, Gurtin, and Slemrod for $N = 1$ (see \cite{CarGurSle}), and by Modica (see \cite{Mod_Gra}) and Sternberg (see \cite{Ste_Sing}) for $N \ge 2$ (see also \cite{ModMor_Limite, ModMor_Esempio}), thus showing that the sharp interface limit of the phase field model \eqref{eq:rescaled} provides the minimal interface criterion observed in experiments.
The mathematical framework used was that of $\Gamma$-convergence, a notion of convergence introduced by De Giorgi and Franzoni in \cite{DGFra}.
To be precise, it was proved that the energy of the sequence $\{u_\e\}_{\e > 0}$ converges, as $\e \to 0$, to
\begin{equation}\label{eq:limenergy}
\sigma \hno(\Sigma),
\end{equation}
where $\Sigma$ is an interface having minimal area, separating two regions whose volume is determined by the mass constraint \eqref{eq:mcW} (which is preserved in the limit), and
\[
\sigma \coloneqq 2 \int_{\alpha}^\beta \sqrt{W(t)}\,dt.
\]
Observe that the factor $\sigma$ represents the energy needed in order to have a phase transition, and it is independent of the position and of the orientation of the interface. The value of $\sigma$ can be also characterized as
\begin{equation}\label{eq:optprofscalar}
\inf\left\{
\int_{-\infty}^{\infty} \left[W(\gamma(t)) + |\gamma'(t)|^2 \right]\,dt : \gamma \in W^{1,\infty}((-\infty,\infty)), \lim_{t \to-\infty} \gamma(t) = \alpha, \lim_{t\to\infty}\gamma(t)= \beta \right\}.
\end{equation}
Functions achieving the minimum in \eqref{eq:optprofscalar} are called \emph{heteroclinic connections} between $\alpha$ and $\beta$.

Let us remark that solutions to the minimal area problem enjoy some regularity properties:
the interior regularity was studied by Gonzales, Massari, and Tamanini in \cite{GonMasTam}, the behavior at the boundary of $\o$ was investigated by Gr\"uter in \cite{Gru}, while the connectivity of the interface was the focus of the paper \cite{SteZum_Conn} by Sternberg and Zumbrun.\\

After the early works mentioned above, the mathematical study of phase transitions has flourished.
Since the literature on this problem and its variants is vast, here we limit ourselves at recalling only the main contributions that are close to the problem we consider in this paper: the static problem for first order phase transitions.

The case where the material has more than two stable phases requires vector-valued phase variables $u:\o\to\R^M$.
Indeed, even if one considers a potential $W \colon \R \to [0,\infty)$ having more than two wells, minimizers of \eqref{eq:Wgrad} will converge to piecewise constant functions taking only two values, which are selected by the mass constraint.
A key ingredient in the treatment of the vectorial case is the study of the relation between the value in \eqref{eq:optprofscalar} where the minimization problem is suitably adapted to the vectorial case, and the geodesic distance between two of the wells $\alpha_i,\alpha_j\in\R^M$ with respect to the metric induced by the degenerate conformal factor $2 \sqrt{W}$, namely
\begin{multline}\label{eq:distW}
d_W(\alpha_i,\alpha_j) \coloneqq \inf\Biggl\{\int_{-1}^1 2\sqrt{W(\gamma(t))}|\gamma'(t)|\,dt : \gamma \in W^{1,\infty}((-1,1);\R^M), \\ \gamma(-1)=\alpha_i, \gamma(1)=\alpha_j \Biggr\}.
\end{multline}
The importance of this relation was first observed by Sternberg in \cite{Ste_Sing} for the case of a potential $W$ vanishing on two disjoint closed simple curves in $\R^2$.
The case of a phase variable $u \colon \o \to \R^M$ and a potential supporting $k = 2$ stable phases was treated by Sternberg in \cite{Ste_Vect} when $M=2$ and by Fonseca and Tartar in \cite{FonTar} when $M\geq 2$, while the general case $k \ge 2$ was investigated by Baldo in \cite{Bal}.
I these works the limiting energy is shown to be of the form
\begin{equation}\label{eq:limenmmulti}
\sum_{i=1}^{k-1} \sum_{j=i+1}^k\hno(\partial^* \o_i \cap \partial^* \o_j) d_W(\alpha_i, \alpha_j),
\end{equation}
where $\alpha_1,\dots,\alpha_k\in\R^M$ denote the wells of the potential $W$, and $\o_i\subset\o$ is the region where the phase variable $u\in BV(\o;\R^M)$ takes the value $\alpha_i$.
Here $\partial^*\o^i$ denotes the essential boundary of the set of finite perimeter $\o^i$ (see Definition \ref{def:partialstar}).

A further generalization was studied by Barroso and Fonseca in \cite{BarFon}, where the authors considered singular perturbations of the form $h(x, \e \nabla u(x))$ and vector-valued phase variables. Moreover, the fully coupled singular perturbations, \emph{i.e.}, Gibbs free energies of the form
\begin{equation}\label{eq:coupled}
\frac{1}{\e} \int_\o f(x, u(x), \e \nabla u(x))\,dx,
\end{equation}
were the main focus of the work \cite{OweSte} by Owen and Sternberg in the scalar case and of \cite{FonPop} by Fonseca and Popovici for vector-valued phase variables. For functionals defined as in \eqref{eq:coupled}, the sharp interface limiting energy was shown to be of the form
\begin{equation}\label{eq:limenergyanis}
\sum_{i=1}^{k-1} \sum_{j=i+1}^k \int_{\partial^*\o^i\cap\partial^*\o^j} \sigma(x,\nu(x))\dhno(x).
\end{equation}
In this case,  if $x$ is in the interface separating $\o_i$ and $\o_j$ and $\nu\in\S^{N-1}$, the energy density $\sigma$ is given by the so called \emph{cell problem}
\[
\sigma(x,\nu) \coloneqq \inf_{s>0} \inf_v \left\{\frac{1}{s}\int_{Q_\nu} f(x,v(y),s \nabla v(y))\,dy \right\},
\]
where $Q_\nu\subset\R^N$ is a unit cube centered at the origin having two faces orthogonal to $\nu$, and the function $v$ ranges among all Lipschitz functions taking the value $\alpha_i$ on one of these two faces, $\alpha_j$ on the other one, and it is $1$-periodic in the directions orthogonal to every other face of the cube $Q_\nu$.

The boundary of the container $\o$ could enter into play either via an interaction energy, or by forcing the phase variable to assume a specific value (not necessarily corresponding to a stable phase).
The first case was studied by Modica in \cite{Mod_Bou} where he considered the energy
\[
\mathcal{E}_\e(u) +  \int_{\partial\o} \tau(\mathrm{Tr}\, u(x)) \,\dhno(x).
\]
Here $\mathcal{E}_\e$ is defined as in \eqref{eq:rescaled},  $\tau \colon \R \to [0,\infty)$ is a $1$-Lipschitz function, and  $\mathrm{Tr}\,u$ denotes the trace of $u$ on $\partial \o$.
The author showed that the sharp interface limit is given by (see \eqref{eq:limenergy})
\[
\sigma\hno(\Sigma) + \int_{\partial\o} \tau(\mathrm{Tr}\, u(x)) \,\dhno(x).
\]
Owen, Rubinstein, and Sternberg treated the case where admissible functions for the minimization problem for \eqref{eq:rescaled} are constrained to satisfy a Dirichlet boundary condition $u = g_\e$ on $\partial \o$ (see \cite{OweRubSte}). The limiting problem was shown to be
\[
\sigma\hno(\Sigma) + \int_{\partial \o} d_W(\mathrm{Tr}\,u(x), g(x)) \dhno(x),
\]
where $g_\e \to g$ in a suitable sense and the distance $d_W$ is the one induced by the degenerate metric as in \eqref{eq:distW}.

The case where the zero level set of $W$ has a more complicated topology was considered by several authors.
The particular situation in which the potential vanishes on two disjoint $C^1$ curves in $\R^2$ was considered by Sternberg in \cite{Ste_Sing}.
The case where the set $Z$ of zeros of $W$ is a generic compact set in $\R^M$ was studied by Ambrosio in \cite{Amb}, where by considering the canonical quotient space $F$ of $(\R^M, d_W)$, together with the canonical projection map $\pi \colon \R^M \to F$, the author was able to prove that the family of functionals in \eqref{eq:rescaled} $\Gamma$-converges to
\begin{equation}\label{eq:Ambrosio}
\int_{J_{\pi(u)}} d_W\left(\pi(u^+(x)), \pi(u^-(x))\right)\dhno(x).
\end{equation}
More recently, Lin, Pan, and Wang (see \cite{LinXinWan}) characterized the asymptotic behavior of sequences of minimizers satisfying a Dirichlet boundary condition in the specific case where $Z$ is the union of two smooth disjoint manifolds $N^+, N^- \subset \R^M$, under the assumption that
\begin{equation}\label{eq:potWhigh}
W(p) \coloneqq f\left(d^2(p, N^+\cup N^-)\right),
\end{equation}
where $d$ denotes the distance function, and $f$ is smooth and behaves linearly in a neighborhood of the origin. Their proofs rely on the fact that geodesics for $d_W$ are shown to be segments joining two points of minimal distance between $N^+$ and $N^-$.

There are physically relevant cases where the phase variable $u$ is not expected to possess derivatives; thus a different singular perturbation is needed.
An example is the continuum limit of the Ising spin system on a lattice, where $u$ represents the magnetization density. In this case the appropriate energy to consider is
\[
\frac{1}{\e}\int_\o W(u(x))\,dx + \frac{1}{\e} \int_\o \int_\o \frac{1}{\e^N} K\left(\frac{y-x}{\e}\right)|u(y)-u(x)|^2 \,dx\,dy,
\]
where $K$ is a ferromagnetic Kac potential, which is assumed to be nonnegative and integrable. This energy was studied by Alberti and Bellettini in \cite{AlbBel2, AlbBel1}, where they proved that the discrete nature of the problem does not affect the form of the limiting energy, which in turn was shown to still be an anisotropic perimeter functional as in \eqref{eq:limenergyanis}.
Notice that the integrability assumption on the potential $K$ excludes the classical seminorm for fractional Sobolev spaces $W^{s,p}(\o)$ for $0 < s < 1$ and $p > 1$.
The one dimensional case for $s=\frac{1}{2}$ and $p = 2$ was considered by Alberti, Bouchitt\'{e}, and Seppecher in \cite{AlbBouSep}, where the authors identified the $\Gamma$-limit of the family of functionals
\[
\lambda_\varepsilon\int_I W(u(x))\,dx + \int_I\int_I \left| \frac{u(y)-u(x)}{y-x} \right|^p \,dx\,dy.
\]
Here $I \subset \R$ is a bounded interval and $\varepsilon \lambda_\varepsilon \to l \in (0,\infty)$. The one dimensional case for $p > 2$ was considered by Garroni and Palatucci in \cite{GarPal}. The $\Gamma$-convergence for the nonlocal perimeter functional
\[
\int_\o W(u(x))\,dx + \frac{\varepsilon^{2s}}{2} \int_\o\int_\o \frac{|u(y)-u(x)|^2}{|y-x|^{N+2s}}\,dx\,dy
	+\varepsilon^{2s}\int_{\R^N\setminus\o} \int_\o \left| \frac{u(y)-u(x)}{y-x} \right|^2\,dx\,dy
\]
in the case $N \ge 2$ and $s \in (0,1)$ was studied by Savin and Valdinoci in \cite{SavVal}.

The Euler--Lagrange equation for minimizers of the functional \eqref{eq:rescaled} subject to the mass constraint \eqref{eq:mcW} was investigated by Luckhaus and Modica in \cite{LucMod} (see \cite{CicNagPis} for the anisotropic case). The authors considered the equation
\[
2\e \Delta u + \frac{1}{\e}W'(u) = \lambda_\e
\]
where $\lambda_\e \in \R$ is the Lagrange multiplier associated to the mass constraint, and proved that
\begin{equation}\label{eq:GT}
\lambda_\e \to \sigma H,
\end{equation}
as $\e \to 0$. Here $H$ denotes the mean curvature of the limiting interface. Formula \eqref{eq:GT} is known in the physics literature as the \emph{Gibbs-Thomson relation}.

The minimal area interface principle serves as a first selection criterion to choose which of the (infinitely many) minimizers of \eqref{eq:enW} is physically relevant. More refined information can be obtained by considering the $\Gamma$-convergence expansion (see \cite{AnzBal_Dev, BraTru}) of the energy \eqref{eq:Wgrad}:
\[
\mathcal{E}_\e = \mathcal{E}^{(0)} + \e \mathcal{E}^{(1)} + \e^2 \mathcal{E}^{(2)} + \dots,
\]
where each $\mathcal{E}^{(i)}$ is the $\Gamma$-limit of the family of functionals
\[
\mathcal{E}^{(i)}_\varepsilon=\frac{\mathcal{E}^{(i-1)}_\varepsilon - \inf \mathcal{E}^{(i-1)}}{\varepsilon}
\]
and $\mathcal{E}^{(0)}_\varepsilon\coloneqq\mathcal{E}_\varepsilon$.
The characterization of the functional $\mathcal{E}^{(1)}$ was carried out by \cite{AnzBal_Dev, AnzBalOrl, BelNayNov, CarGurSle, DMFonLeo, LeoMur, LeoMur_AC} in several cases of interest.

Variants of phase transitions models of the form \eqref{eq:rescaled} could also be used to investigate more intricate situations. For instance, the interaction between phase transitions and homogenization phenomena is described by the functional
\[
\int_\o \left[ \frac{1}{\e}W\left(\frac{x}{\e}, u(x)\right) + \e |\nabla u(x)|^2 \right]\,dx.
\]
Here the periodic structure of the material is modeled by the periodicity of the function $W$ in the first variable. The sharp interface model was derived by Braides and Zeppieri in \cite{BraZep} for the one dimensional case $N = 1$ and by Fonseca, Hagerty, Popovici, and the first author in \cite{CriFonHagPop} (see also \cite{ChoFonLinVen}) for $N > 1$. When the homogenization takes place at the level of the singular perturbation we refer to \cite{AnsBraChi1, AnsBraChi2, DirMarNov1, DirMarNov2}.\\

All the previous works are based on (variants of) model \eqref{eq:rescaled}, which describes phenomena where the system is assumed to be under \emph{isothermal} conditions. There are physically relevant situations however, where this is not the case. For instance, consider a homogeneous mixture of a binary system in thermal equilibrium. If we quench the system below a critical temperature, then we would expect phase separation. Since the quenching takes place over a finite amount of time, the assumption of isothermal conditions is not plausible.
In addition, there are situations where the phase separation process can be directed by using
an external thermal activation (see \cite{AltPav} and the references therein).
In all of these cases, the model \eqref{eq:rescaled} is not adequate to describe the physics of the phenomenon.
A system of evolution equations aimed at modelling phase transitions under nonisothermal conditions was proposed by Penrose and Fife in \cite{PenFif} and by Alt and Pawlow in \cite{AltPav}.
The free energy they considered reads as
\[
\int_\o \left[\frac{1}{\e} W(T(x), u(x)) + \e K(T(x))|\nabla u(x)|^2\right]\,dx,
\]
where $T \colon \o \to \R$ represents the temperature of the material (or any external field), and $K$ is a given positive function. Here the unknowns of the problem are both the temperature distribution $T$ and the phase parameter $u$. In particular, it could be the case where the wells of $W$ depend themselves on the temperature, and thus are not necessarily the same for all points $x \in \o$.
The dependence of $W$ on both of the unknowns poses analytical challenges.

In order to get some insight we assume the distribution of the temperature $T$ to be given a priori and $K$ to be constant. These simplifications allow us to consider a free energy of the form
\begin{equation}\label{eq:ourenergy}
\f_\e(u) \coloneqq \int_\o \left[\frac{1}{\e} W(x, u(x)) + \e|\nabla u(x)|^2 \right]\,dx,
\end{equation}
where the potential $W \colon \o \times \R^M \to [0, \infty)$ is such that $W(x, p) = 0$ if and only if $p \in \{z_1(x), \dots, z_k(x)\}$,
and the $z_i \colon \o \to \R^M$ are given functions representing the stable phases of the material at each point $x \in \o$.

The functional \eqref{eq:ourenergy} was considered by Ishige in \cite{Ish_W} (see also \cite{Ish_vec}) in the vectorial case, \emph{i.e.} $M > 1$, when $k = 2$ and $z_1$, $z_2$ are constants.
To the best of our knowledge, there are only two papers that considered the case where the functions $z_i$ are nonconstant: \cite{Ste_Sing} by Sternberg and \cite{Bou} by Bouchitt\'{e}.
They both treated the scalar case, \emph{i.e.\@} $M = 1$, with \emph{two moving} wells.
A specific kind of potential in two dimensions is considered in the former work, while fully coupled singular perturbations in general dimension are treated in the latter.
More precisely, in \cite{Bou} the author considered an energy of the form
\[
\mathcal{G}_\e(u) \coloneqq \frac{1}{\e} \int_\o f(x, u(x),\e Du(x))\,dx,
\]
where $f(x, u, 0) = 0$ if and only if $u \in \{z_1(x), z_2(x)\}$. The wells $z_1$ and $z_2$ are allowed to coincide in a subset of $\o$. The limiting functional was shown to be
\[
\mathcal{G}_0(u)\coloneqq\int_{J_u} h(x,\nu_u(x)) \dhno(x),
\]
for $u \in BV_{loc}(\o_0)$, where $\o_0 \coloneqq \{ x \in \o : \alpha(x) \neq \beta(x) \}$, and $\infty$ otherwise in $L^1(\o)$. Here, for $x \in \o$ and $\nu \in \S^{N-1}$, we define
\[
h(x, \nu_u) \coloneqq \lim_{r \to \infty} \inf_\gamma \left\{ \int_0^r f(x, \gamma(t), \gamma'(t)\nu)\,dt \right\},
\]
where the infimum is taken over all Lipschitz curves $\gamma$ connecting $z_1(x)$ and $z_2(x)$. The scalar nature of the problem allows to implement techniques that are purely one dimensional and that cannot be adapted to the vectorial case.\\

In this paper we consider for the first time the energy \eqref{eq:ourenergy} in the vectorial case, with $k \ge 2$, and for functions $z_i$ which are possibly nonconstant. In particular, we prove that any sequence $\{u_{\e_n}\}_{n \in \N} \subset H^1(\o; \R^M)$ such that
\[
\sup \left\{ \f_{\e_n}(u_{\e_n}) : n \in \N \right\} < \infty,
\]
where $\e_n \to 0$, converges (eventually extracting a subsequence) to a function $u \in BV(\o; \R^M)$ of the form
\[
u = \sum_{i=1}^k z_i \ca_{\o_i}.
\]
Here $\{\o_1, \dots, \o_k\}$ is a Caccioppoli partition of $\o$. Moreover, the limiting sharp interface energy is
\[
\sum_{i=1}^{k-1} \sum_{j = i+1}^k  \int_{\partial^* \o_i \cap \partial^* \o_j}d_W(x,z_i(x), z_j(x)) \dhno(x),
\]
where, for $p, q \in \R^M$, $d_W(x, p, q)$ is the geodesic distance induced by the degenerate conformal factor $2 \sqrt{W(x, \cdot)}$. Notice that if the wells $z_i$ are independent of $x$ we recover \eqref{eq:limenmmulti}. We refer to the next section for the precise statement of the results and for all the assumptions we require.


\subsection{Statement of the main results}\label{statementMR}
Let $\o \subset \R^N$, $N \ge 2$, be a bounded open set with Lipschitz continuous boundary. Throughout the paper we make the following assumptions on the potential $W$.
\begin{enumerate}[label=(H.\arabic*), ref=H.\arabic*]
\item \label{H1} $W \colon \overline{\o} \times \R^M \to [0, \infty)$ is locally Lipschitz continuous, \emph{i.e.\@}, Lipschitz continuous on every compact subset of $\overline{\o} \times \R^M$. Moreover, for every $x \in \overline{\o}$, $W(x, p) = 0$ if and only if $p \in \{z_1(x), \dots, z_k(x)\}$, where the functions $z_i \colon \overline{\o} \to \R^M$ are Lipschitz continuous;
\item \label{H2} There exists $\delta > 0$ such that
\[
\min\left\{|z_i(x) - z_j(x)| : x \in \overline{\o} \text{ and } i \neq j \right\} \ge \delta;
\]
\item \label{H3} There exists $r > 0$ such that if $p \in B(z_i(x), r)$ then
\[
W(x, p) = \alpha_i |p - z_i(x)|^2,
\]
where $\alpha_i > 0$, for all $i = 1, \dots, k$ and $x\in\o$;
\item \label{H4} There exist $R, S > 0$ such that $W(x, p) \ge S|p|$, for all $x \in \o$ and all $p$ with $|p| > R$.
\end{enumerate}

\begin{definition}
For $\e > 0$, let $\f_\e \colon L^1(\o; \R^M) \to [0, \infty]$ be the functional defined via
\[
\f_\e(u) \coloneqq
\left\{
\begin{array}{ll}
\displaystyle \int_\o \left[ \frac{1}{\e}W\left(x,u(x)\right) + \e|\nabla u(x)|^2 \right]\,dx & \text{ if } u \in H^1(\o; \R^M), \\
&\\
\infty & \text{ otherwise in } L^1(\o;\R^M).
\end{array}
\right.
\]
\end{definition}

In order to define the limiting functional, we need to introduce some notation.

\begin{definition}
\label{defdW}
For $p, q \in \R^M$ consider the class
\begin{equation}
\label{Apq}
\mathcal{A}(p, q) \coloneqq \left\{ \gamma \in W^{1,1}((-1,1); \R^M) : \gamma(-1) = p, \gamma(1) = q \right\}
\end{equation}
and let $d_W \colon \overline{\o} \times \R^M \times \R^M \to[0, \infty)$ be the function defined via
\begin{equation}\label{eq:infpbdw}
d_W(x,p,q) \coloneqq \inf \left\{\int_{-1}^1 2\sqrt{W(x, \gamma(t))}|\gamma'(t)|\,dt : \gamma \in \mathcal{A}(p,q) \right\}.
\end{equation}
\end{definition}

It is immediate to verify that for all $x \in\overline{\o}$, the function $d_W(x, \cdot, \cdot) \colon \R^M \times \R^M \to [0, \infty)$ defines a distance on $\R^M$. The existence of solutions to the minimization problem \eqref{eq:infpbdw}, referred to as \emph{minimizing geodesics} throughout the paper, is a classical problem which has been the subject of investigation of several studies. Since our proofs rely on a precise understanding of the dependence of minimizing geodesics on the variable $x$, our approach (see \Cref{dFprop}) requires more stringent assumptions on $W$ than the ones required by Zuniga and Sternberg in \cite[Theorem 2.5]{ZunSter} (see also \cite{MS18}), and is in spirit closer to the work of Sternberg \cite{Ste_Vect}.

We can now define our limiting functional. For all the relevant definitions we refer the reader to \Cref{sec:preliminaries} (see, in particular, \Cref{def:jump}).

\begin{definition}
\label{f0}
Set
\[
BV(\o; z_1, \dots, z_k) \coloneqq \left\{ u \in BV(\o; \R^M) : u(x) \in \{z_1(x), \dots, z_k(x)\} \text{ for } \mathcal{L}^N\text{-a.e.\@ } x \in \o \right\},
\]
and let $\f_0 \colon L^1(\o; \R^M) \to [0, \infty]$ be the functional defined via
\[
\f_0(u) \coloneqq
\left\{
\begin{array}{ll}
\displaystyle \int_{J_u} d_W(x, u^+(x), u^-(x))\,d\hno(x) & \text{ if } u \in BV(\o; z_1, \dots, z_k), \\
&\\
\infty & \text{ otherwise in } L^1(\o;\R^M).
\end{array}
\right.
\]
\end{definition}

Throughout the rest of the paper we fix $\{\e_n\}_{n \in \N}\subset (0,\infty)$ with $\e_n \to 0$ as $n \to \infty$, and we write $\f_n$ for $\f_{\e_n}$.
We are now in position to state our main result.

\begin{theorem}\label{thm:main}
Let $W$ be given as in \emph{(\ref{H1})}--\emph{(\ref{H4})} and let $\{u_n\}_{n \in \N} \subset H^1(\o; \R^M)$ be such that
\[
\sup \left\{ \f_n(u_n) : n \in \N \right\} < \infty.
\]
Then, eventually extracting a subsequence (not relabeled), we have that $u_n \to u$ in $L^1(\o; \R^M)$, where $u \in BV(\o; z_1, \dots, z_k)$ is such that $\f_0(u) < \infty$. Moreover, the sequence of functionals $\f_n$ $\Gamma$-converges with respect to the $L^1$-topology to $\f_0$. 
\end{theorem}

The proofs we present are robust and can be adapted to work also for several variants of the problem. In this paper we focus on two of these: the mass constrained problem and the case of Dirichlet boundary conditions. 

Fix $\mathcal{M} = (m_1,\dots,m_M)\in\R^M$ in such a way that
\[
\min_{1 \le i \le k} \int_\o z_i^j(x) \,dx \le m_j \le \max_{1 \le i \le k} \int_\o z_i^j(x) \,dx
\]
for every $j = 1, \dots, M$, where $z_i^j(x)$ denotes the $j^{th}$ component of $z_i(x)$.

\begin{theorem}\label{thm:massconstr}
Let $W$ be given as in \emph{(\ref{H1})}--\emph{(\ref{H4})} and let $\mathcal{M} \in \R^M$ be as above.
For $n\in\N$, let
\[
\f_n^{\mathcal{M}}(u)\coloneqq
\left\{
\begin{array}{ll}
\f_n(u) & \text{ if }u\in H^1(\o;\R^M) \text{ with } \int_\o u(x)\, dx=\mathcal{M},\\
&\\
\infty & \text{ otherwise in } L^1(\o;\R^M).
\end{array}
\right.
\]
Then the followings hold:
\begin{itemize}
\item[$(i)$] if $\{u_n\}_{n \in \N} \subset H^1(\o; \R^M)$ is such that
\[
\sup \{ \mathcal{F}^{\mathcal{M}}_n(u_n) : n \in \N \} < \infty,
\]
then, eventually extracting a subsequence (not relabeled), we have that $u_n \to u$ in $L^1(\o; \R^M)$, where $u \in BV(\o;z_1,\dots,z_k)$ is such that $\mathcal{F}^{\mathcal{M}}_0(u) < \infty$. Here the functional $\mathcal{F}^{\mathcal{M}}_0$ is defined via
\[
\f_0^{\mathcal{M}}(u) \coloneqq
\left\{
\begin{array}{ll}
\f_0(u) & \text{ if }u \in BV(\o; z_1, \dots, z_k) \text{ with } \int_\o u(x)\,dx=\mathcal{M},\\
&\\
\infty & \text{ otherwise in } L^1(\o; \R^M).
\end{array}
\right.
\]
\item[$(ii)$] The sequence of functionals $\mathcal{F}^{\mathcal{M}}_n$ $\Gamma$-converges with respect to the $L^1$-topology to $\mathcal{F}^{\mathcal{M}}_0$.
\end{itemize}
\end{theorem}

Using the results in \cite{KohSte}, we deduce the following corollary.

\begin{corollary}\label{cor:mass}
Under the assumptions of \Cref{thm:massconstr}, let $u_0 \in BV(\o;z_1,\dots,z_k)$ be an $L^1$-isolated local minimizer for $\f_0^{\mathcal{M}}$, namely there exists $\lambda>0$ such that
\[
\f_0^{\mathcal{M}}(u_0) < \f_0^{\mathcal{M}}(v)
\]
for all $v \in L^1(\o;\R^M)$ with $0< \| v - u_0 \|_{L^1(\o; \R^M)} < \lambda$.
Then there exists $\{u_n\}_{n\in\N}\subset H^1(\o;\R^M)$ where each $u_n$ is an $L^1$-local minimizer for $\f_n^{\mathcal{M}}$, such that $u_n\to u_0$ in $L^1(\o;\R^M)$.
\end{corollary}

Next, we consider the case where a Dirichlet condition is imposed on the boundary of $\o$.

\begin{theorem}\label{thm:Dirichlet}
Let $W$ be given as in \emph{(\ref{H1})}--\emph{(\ref{H4})} and fix $g \in \Lip(\partial\o; \R^M)$ with
\[
\min\{ |z_i(x) - g(x)| : x\in \partial\o, i\in\{1,\dots,k\} \}\geq \delta.
\]
For $n \in \N$ define 
\[
\mathcal{F}^{D}_n(u)\coloneqq
\left\{
\begin{array}{ll}
\f_n(u) & \text{ if } u \in H^1(\o;\R^M)  \text{ with }  \mathrm{Tr}\,u = g  \text{ on } \partial\o,\\
& \\
\infty & \text{ otherwise in  }L^1(\o; \R^M).
\end{array}
\right.
\]
Then the followings hold:
\begin{itemize}
\item[$(i)$] if $\{u_n\}_{n\in\N} \subset H^1(\o; \R^M)$ is such that
\[
\sup \left\{ \mathcal{F}^{D}_n(u_n) : n \in \N \right\} < \infty,
\]
then, eventually extracting a subsequence (not relabeled), we have that $u_n \to u$ in $L^1(\o; \R^M)$, where $u \in BV(\o;z_1,\dots,z_k)$ is such that $\mathcal{F}^{D}_0(u) < \infty$. Here the functional $\mathcal{F}^{D}_0$ is defined via
\[
\mathcal{F}^{D}_0(u) \coloneqq \f_0(u) + \int_{\partial\o} d_W(x,\mathrm{Tr}\,u(x),g(x))\dhno(x).
\]
\item[$(ii)$] The sequence of functionals $\mathcal{F}^{D}_n$ $\Gamma$-converges with respect to the $L^1$-topology to $\f^D_0$.
\end{itemize}
\end{theorem}


\subsection{Sketch of the strategy}

Despite the fact that the strategy we have to follow is clear, the path to the proof of the main result (\Cref{thm:main}) is studded with technical difficulties.

First of all, we comment on the compactness result.
For clarity of exposition, assume that $k = 2$, \emph{i.e.\@}, there are only two wells, namely $z_1, z_2$.
From the energy bound and Young's inequality we get
\[
\sup_{n \in \N} \int_\o 2\sqrt{W(x,u_n(x))} |\nabla u_n(x)| \,dx \leq \sup_{n\in\N} \f_n(u_n)  < \infty.
\]
In the case where $W$ is independent of $x$, and thus $z_1, z_2$ are constant, the proof originally proposed by Modica in \cite{Mod_Gra} (see also \cite{FonTar}) proceeds as follows: it can be checked that
\begin{equation}\label{eq:cmpt}
\sup_{n\in\N}|D(w\circ u_n)|(\o) \leq \sup_{n\in\N} \int_\o 2\sqrt{W(x,u_n(x))} |\nabla u_n(x)| \,dx,
\end{equation}
where $w(p)\coloneqq d_W(p,z_1)$, and therefore the $BV$-compactness implies that $w \circ u_n \to w \circ u$ in $L^1(\o)$, where $w\circ u \in BV(\o)$. From this, one can then deduce that $u \in BV(\o)$ and that it only takes the values $z_1, z_2$.
In our case, since $W$ depends on $x$, instead of \eqref{eq:cmpt} we get
\[
\sup_{n\in\N} \int_\o \left| \nabla_y g_n(x,x)  \right|\,dx \le \sup_{n\in\N} \int_\o 2\sqrt{W(x,u_n(x))} |\nabla u_n(x)| \,dx,
\]
where
\[
g_n(x,y) \coloneqq d_W \left(x,z_1(x),u_n(y)\right).
\]
Therefore, in order to apply $BV$-compactness for the sequence of functions $\{g_n\}_{n\in\N}$, we need a control on the other derivatives as well. Notice that this does \emph{not} come from the energy bound, and is achieved by showing that the function $x \mapsto d_{W}(x,p,q)$ is Lipschitz continuous for every $p,q \in \R^M$ (see \Cref{dWisLipsc}).
We prove this by first deriving a uniform upper bound on the Euclidean length of minimizing geodesics for our degenerate metric (see \Cref{dFprop}); we discuss this at the end of this section.

Here a note is a must. To the best of our knowledge, the strategy that we summarized above is \emph{the} way to get compactness for this kind of problems. Indeed, in every papers that treated the issue (see, for example, \cite{BarFon, FonPop, FonTar}), suitable assumptions are required in order to use the argument described above. 

We remark that in \cite{Bal} it was assumed that (for a potential $W$ independent of $x$)
\begin{equation}
\label{BaldoW}
W(p) \ge \sup_K W
\end{equation}
for every $p \not \in K$, where $K \coloneqq [k_1, k_2]^M$. Since $W$ is continuous, (\ref{BaldoW}) implies that $W$ is constant on $\partial K$, which is a rather restrictive assumption on $W$.
Moreover, since for $M > 1$ the space $H^1(\o; \R^M)$ is not closed under truncation, we instead replace (\ref{BaldoW}) with (\ref{H4}) and consider projections rather than truncations in order to reduce to a sequence $\{u_n\}_{n\in\N}$ bounded in $L^\infty(\o;\R^M)$ (see Step 2 in \Cref{prop:cpt}).

The strategy we use to prove the liminf inequality is the blow-up method introduced by Fonseca and M\"{u}ller in \cite{FonMul_BV}.
To summarize the argument it is not restrictive to assume that $\o = Q \subset \R^N$ is the unit cube with faces orthogonal to the coordinate axes and that $u \in BV(\o; \R^M)$ is defined via
\[
u(x) \coloneqq 
\left\{
\begin{array}{ll}
z_1(x) & \text{ if } x_N < 0,\\
z_2(x) & \text{ if } x_N \ge 0.
\end{array}
\right.
\]
Let $\{\rho_m\}_{m\in\N} \subset (0, 1)$ be such that $\rho_m \to 0$, and consider the rescaled cubes $Q_{\rho_m} \coloneqq \rho_m \overline{Q}$.
Let $\{u_n\}_{n \in \N}$ be a sequence of functions in $H^1(\o;\R^M)$ such that $u_n \to u$ in $L^1(\o;\R^M)$.
For the sake of the argument, assume in addition that $u_n(x) = z_1(x)$ if $x_N = -\rho_m/2$, and that $u_n(x) = z_2(x)$ if $x_N = \rho_m/2$. Our aim is to show that
\begin{equation}
\label{liminfsketch}
\lim_{m\to\infty} \liminf_{n\to\infty} \frac{1}{\rho_m^{N-1}} \int_{Q_{\rho_m}} \left[ \frac{1}{\e_n} W(x,u_n(x)) + \e_n|\nabla u_n(x)|^2
		\right]\,dx \ge d_W(0, z_1(0), z_2(0)).
\end{equation}
Since the map $x \mapsto W(x,p)$ is continuous, by an application of Young's inequality and Tonelli's theorem one expects to obtain
\begin{align*}
&\frac{1}{\rho_m^{N-1}} \int_{Q_{\rho_m}} \left[ \frac{1}{\e_n} W(x,u_n(x))
 	+ \e_n|\nabla u_n(x)|^2 \right]\,dx \\
&\hspace{3cm} \ge \frac{1}{\rho_m^{N-1}} \int_{Q_{\rho_m}} 2\sqrt{W(x,u_n(x))}|\nabla u_n(x)\cdot e_N| \,dx \\
&\hspace{3cm} \sim \int_{-1}^1 2
	\sqrt{W(0,u_n(0',t))}|\nabla u_n(0',t)\cdot e_N| \,dt \\
&\hspace{3cm} \ge d_W(0,z_1(0),z_2(0)),
\end{align*}
where in the previous to last line we used the notation $(0',t) = (0,\dots,0,t)\in\R^N$, for $t\in\R$. One possible way make this heuristics rigorous is the following:
\begin{align*}
&\frac{1}{\rho_m^{N-1}} \int_{Q_{\rho_m}} \left[ \frac{1}{\e_n} W(x,u_n(x))
 	+ \e_n|\nabla u_n(x)|^2 \right]\,dx \\
&\hspace{3cm}\geq\frac{2}{\rho_m^{N-1}}\int_{Q'_m}\int_{-1}^{1}
	\sqrt{W((x',\rho_m s),\widetilde{u}_n(x',s))}
    |\nabla \widetilde{u}_n(x', s)\cdot e_N| \,d s \,dx'
\end{align*}
where $Q'_{\rho_m} \coloneqq \{x' : (x',0) \in Q_{\rho_m}\}$ and $\widetilde{u}_n(x',s) \coloneqq u_n(x', \rho_m s)$.
To conclude, one would need to prove that the function
\begin{equation}
\label{nongeometric}
x'\mapsto \inf \left\{ \int_{-1}^1 2\sqrt{W((x',\rho_m s), \gamma(s))}|\gamma'(s)|\,d s : \gamma \in \mathcal{A}(z_1(0),z_2(0)) \right\}
\end{equation}
is continuous.
Notice that the minimization problem on the right-hand side of (\ref{nongeometric}) is significantly different from the geodesic problem (\ref{eq:infpbdw}), in that in (\ref{nongeometric}) the conformal factor depends also on the variable of integration $s$.
One way to prove the continuity is to show that curves solving that minimization problem have uniformly finite Euclidean length (or at least that there exists one such curve enjoying this property) and then exploit the Lipschitz continuity of $W$.
However, in the present work we choose to reason as follows. Define
\[
F_m(p) \coloneqq \min\left\{2\sqrt{W(x,p)} : x \in Q_{\rho_m}\right\}.
\]
Then one can show that
\[
\frac{1}{\rho_m^{N-1}} \int_{Q_{\rho_m}} \left[ \frac{1}{\e_n} W(x,u_n(x))
 	+ \e_n|\nabla u_n(x)|^2 \right]\,dx 
\ge d_{F_m}(z_1(0),z_2(0)),
\]
where
\begin{equation}\label{eq:dFm}
d_{F_m}(p,q) \coloneqq \inf \left\{\int_{-1}^1F_m(\gamma(t))|\gamma'(t)|\,dt : \gamma \in \mathcal{A}(p,q) \right\}.
\end{equation}
With this in hand, to conclude (see (\ref{liminfsketch})) it is enough to show that
\[
\lim_{m\to\infty}d_{F_m}(z_1(0),z_2(0))= d_W(0,z_1(0),z_2(0)).
\]
Notice that the function $F_m$ vanishes on the set
\[
Z= \bigcup_{i=1}^k z_i(Q_{r_m}).
\]
In view of (\ref{H2}), we can assume $m$ large enough so that the sets $z_i(Q_{r_m})$ are pairwise disjoint.
Let us remark that the advantage to work with (\ref{eq:dFm}) instead of (\ref{nongeometric}) is that the latter is a purely geometric problem, \emph{i.e.\@}, the functional that we aim at minimizing does not depend on the specific choice of the parametrization.
We are able to prove an \emph{explicit} upper bound on the Euclidean length of certain solutions to the minimization problem in \eqref{eq:dFm} (see \Cref{dFprop}).
Furthermore, since the only property of $Z$ that needed in the proof is that it is the union of the images of a compact convex set through the $z_i$'s, the argument also works for the case where $Z = \{ z_1(x),\dots, z_k(x)\}$ for some $x \in \o$.
The strategy we use is the following.
First of all, we show that the specific behavior of the potential in a neighborhood of the wells yields
\begin{equation}\label{eq:Fm}
F_m(z) = 2\sqrt{\alpha_i}d_i(z)
\end{equation}
if $z\in\R^M$ is sufficiently close to $Z$, where $d_i(z)$ denotes the distance between $z$ and $z_i(Q_{\rho_m})$.
Given $p, q \in \R^M$ we want to show that solutions to \eqref{eq:dFm} have uniformly finite Euclidean length.
We only discuss the case where  $p$ and $q$ belong to a neighborhood of $z_i(Q_{\rho_m})$ for some $i\in\{1,\dots,k\}$ since the general case will be obtained by using the upper bound for each $i\in\{1,\dots, k\}$ and (\ref{H4}) to get an upper bound of the length of geodesics outside of these neighborhoods.
We consider three cases:
\begin{itemize}
\item[$(i)$] If $p, q \in z_i(Q_{\rho_m})$, then a minimizing geodesic is simply given by the image through $z_i$ of a segment connecting two points in $z_i^{-1}(\{p\})$ and $z_i^{-1}(\{q\})$ respectively;
\item[$(ii)$] If $p \in z_i(Q_{\rho_m})$ and $q\not \in z_i(Q_{\rho_m})$, let us denote by $q'$ one projection of $q$ on $z_i(Q_{\rho_m})$.
Then the curves obtained by first connecting $q$ and $q'$ with a segment and then $q'$ and $p$ with a curve in $z_i(Q_{\rho_m})$ is a solution to the minimization problem in \eqref{eq:dFm}. The proof of this uses the co-area formula and that each curve connecting $p$ and $q$ must traverse every level set of $d_i$ lower than $d_i(q)$.
This latter fact follows by using \eqref{eq:Fm};
\item[$(iii)$] If $p, q \not\in z_i(Q_{\rho_m})$ we are able to prove the existence of a minimizing geodesic $\gamma \in \mathcal{A}(p,q)$ with the property that
\[
L(\gamma) \leq |p-p'|+|q-q'|+\Lip(z_i) \mathrm{diam}(Q_{\rho_m}).
\]
Here $L(\gamma)$ is the Euclidean length of $\gamma$, and $p'$ and $q'$ denote projections on $z_i(Q_{\rho_m})$ of $p$ and $q$, respectively.
\end{itemize}

Next, we would like to comment on a hypothesis used to get the liminf inequality in the work \cite{Ish_W} by Ishige.
There the author considered potentials $W \colon \o \times \R^M \to [0, \infty)$ such that, for all $x \in \o$, $W(x,p) = 0$ if and only if $p \in\{ \alpha,\beta \}$, for \emph{fixed} $\alpha, \beta \in \R^M$ and satisfy the following property: for each $\lambda_1 > 0$ there exists $\lambda_2 > 0$ such that for all $p\in\R^M$ it holds
\begin{equation}\label{Ishige}
\left|\sqrt{W(x,p)} - \sqrt{W(y,p)}\right| \le \lambda_1 \sqrt{W(x,p)}
\end{equation}
whenever $|x-y| \le \lambda_2$.
As remarked in \cite{Ish_W}, \eqref{Ishige} is satisfied if, for example, $W(x,p) = h(x) U(p)$, with $h > 0$.
We notice that assumption \eqref{Ishige} does not hold even in the simple case of a single moving well. For this reason one cannot immediately adapt the proof in \cite{Ish_W} to our case. \\

The construction of the recovery sequence is carried out as follows. Thanks to Lemma \ref{lem:approxregular} we can assume
\[
u = \sum_{i=1}^k z_i\ca_{\o^i}
\]
where $\{\o_1,\dots,\o_k\}$ is a Caccioppoli partition of $\o$ and $\partial\o^i\cap\o$ is contained in a finite union of hyperplanes, for each $i=1,\dots,k$.
For the sake of exposition, we just discuss how to build the recovery sequence in a neighborhood of a connected component $\Sigma$ of $\partial \o^i \cap \o$ contained in an hyperplane.
Without loss of generality, we can assume that $\Sigma\subset \{x_N=0\}$.
The quantity we want to approximate is
\[
\int_\Sigma d_W(x,u^-(x),u^+(x)) \dhno(x).
\]
To fix the ideas, let us assume that $u^-(x)=z_1(x)$, and $u^+(x)=z_2(x)$ for all $x\in\Sigma$.
Consider a grid of $(N-1)$-dimensional cubes $Q'(\overline{y}_i,r_n)\subset\{ x_N=0 \}\sim\R^{N-1}$ of side length $r_n>0$ and centre
$\overline{y}_i\in\{ x_N=0 \}$.
Identify a point $x\in\R^N$ with the pair $(y,t)$ where $y\in\R^{N-1}$ and $t\in\R$.
Since the map $x\mapsto d_W(x, z_1(x), z_2(x))$ is continuous (see \Cref{dWisLipsc}), it is enough to approximate
\[
\sum_i d_W(\overline{y}_i,z_1(\overline{y}_i),z_2(\overline{y}_i))
	\hno(Q'(\overline{y}_i,r_n)).
\]
The advantage of considering this discretization is the following:
for each $(y,t)\in Q'(\overline{y}_i,r_n) \times (0,\tau_n)$, for some $\tau_n>0$ with $\tau_n\to0$ as $n\to\infty$, we can simply consider a suitable reparametrization of a geodesic $\gamma_i\in\mathcal{A}( z_1(\overline{y}_i),z_2(\overline{y}_i) )$ for $d_W$, instead of taking a different geodesic for each $x\in Q'(\overline{y}_i,r_n)$.
This comes at the cost of having to perform two gluing constructions in order for the function we define to have the required regularity $H^1(\o;\R^M)$. The first one is to use cut-off functions to transition between the geodesics considered in each adjacent cube. The second one is to match the value $z_i(\overline{y}_i)$ with $z_i(y,\tau_n)$. This will be done by using a linear interpolation. The technical difficulty is to show that the energy contribution of these gluing constructions is asymptotically negligible.


\subsection{A discussion on the assumptions}

It is immediate to verify that in view of (\ref{H1}), (\ref{H3}), and (\ref{H4}), there exists a positive number $\eta$ such that
\begin{equation}
\label{Hold4}
W(x,p) \ge \eta \quad \text{ for all } (x, p) \in \o \times \left( \R^M \setminus \bigcup_{i = 1}^kB(z_i(x),r/2)\right).
\end{equation}
Let us mention here that while assumption (\ref{H4}) is required in order to obtain compactness of sequences with uniformly bounded energies, it is possible to prove the $\Gamma$-convergence results of \Cref{statementMR} by assuming  (\ref{H1})--(\ref{H3}), and \eqref{Hold4} in place of  (\ref{H4}).
We refer the reader to \Cref{remH4} for more details.\\

Finally, we notice that \eqref{H2}, \eqref{H3}, and \eqref{Hold4} are only needed in oder to obtain the results of \Cref{sec:tech} (see \Cref{dFprop} and \Cref{dWisLipsc}). If the results of \Cref{sec:tech} could be obtained with weaker assumptions than \eqref{H2} and \eqref{H3}, then the statements and the proofs of the main results would require a few adjustments, as we explain below.
First of all, we notice that \eqref{H2}, \eqref{H3}, and \eqref{Hold4} imply that
\[
d_W(x, z_i(x), z_j(x)) \geq C > 0
\]
for all $i \neq j \in {1, \dots ,k}$ and $x \in \o$.
Therefore, the functional $\widetilde{\f_0} \colon L^1(\o;\R^M)\to[0,+\infty]$ defined as 
\[
\widetilde{\f_0}(u) \coloneqq
\left\{
\begin{array}{ll}
\displaystyle \int_{J_u} d_W(x, u^+(x), u^-(x))\,d\hno(x) & \text{ if } u(x) \in \{z_1(x), \dots, z_k(x)\} \text{ for } \mathcal{L}^N\text{-a.e.} x \in \o, \\
&\\
\infty & \text{ otherwise in } L^1(\o;\R^M),
\end{array}
\right.
\]
is finite only if $u\in BV(\o;z_1,\dots,z_k)$ with $\hno(J_u)<\infty$. In particular, it coincides with $\f_0$ (see \Cref{f0}).
Notice that $\widetilde{\f_0}$ is well defined since $J_u$ is countably $\hno$-rectifiable for all $u \in L^1_{\mathrm{loc}}(\o;\R^M)$ (see \cite{Del}).
On the other hand, if (\ref{H2}) does not hold, \emph{i.e.\@}, if
\[
\min \left\{ |z_i(x) - z_j(x)| : x \in \overline{\o}, i \neq j  \right\} = 0,
\]
then there could exist $u \in L^1(\o; \R^M)$ such that $\widetilde{\f_0}(u) < \infty$, but $u$ is \emph{not} of bounded variation, as the following remark shows.

\begin{remark}\label{rem:counterexample}
Take $N = M = k = 2$, $\o = (-1, 1)^2$, $W(x, p) \coloneqq |p - z_1(x)|^2 |p - z_2(x)|^2$, where
\[
z_1(x_1,x_2) \coloneqq (x_1,0),\quad\quad\quad
z_2(x_1,x_2) \coloneqq
\left\{
\begin{array}{ll}
(x_1,x_1^2) & \text{ if } x_1 \ge 0,\\
(x_1,0) & \text{ if } x_1 < 0.
\end{array}
\right.
\]
Notice that $W$ satisfies (\ref{H1}), (\ref{H3}), and (\ref{H4}), but not (\ref{H2}). Let $f(t) \coloneqq \sin \left(t^{-2}\right)$, and consider the function
\[
u(x_1,x_2)\coloneqq
\left\{
\begin{array}{ll}
z_1(x) & \text{ if } x_2<f(x_1),\\
z_2(x) & \text{ if } x_2\ge f(x_1).\\
\end{array}
\right.
\]
Fix $x = (x_1,x_2) \in \o$ with $x_1 > 0$ and let $\gamma \colon [-1, 1] \to \R^2$ be the curve given by 
\[
\gamma(t) \coloneqq \left(x_1,\frac{t+1}{2} x_1^2\right).
\]
As one can readily check, $\gamma \in \mathcal{A}(z_1(x),z_2(x))$ and
\[
d_W(x, z_1(x), z_2(x)) \le \int_{-1}^12\sqrt{W(x,\gamma(t))}|\gamma'(t)|\,dt \le x_1^6.
\]
Therefore, by means of a direct computation we see that 
\[
\widetilde{\f_0}(u) \le \int_0^1 x_1^6 \sqrt{1 + |f'(x_1)|^2}\,dx_1 < \infty,
\]
while on the other hand we have
\begin{align*}
\int_{J_u} |z_1(x) - z_2(x)|\,d\mathcal{H}^1(x) & = \int_0^1 x_1^2 \sqrt{1+ |f'(x_1)|^2} \,dx_1 \\
& \ge \int_0^1 x_1^2 |f'(x_1)| \,dx_1 = \int_1^{\infty}\frac{|\cos(t)|}{t}\,dt  \ge \sum_{n = 1}^{\infty}\int_{\frac{\pi}{4} + 2\pi n}^{{\frac{3\pi}{4} + 2\pi n}}\frac{1}{2t}\,dt = \infty.
\end{align*}
Consequently, $\widetilde{\f_0}(u) < \infty$, but $u$ is not of bounded variation. Moreover, notice that the jump set of the function $u$ is not the boundary of a partition of $\o$.
\end{remark}

\begin{theorem}\label{noH2H3}
Let $W$ be given as in \emph{(\ref{H1})} and \emph{(\ref{H4})}, and assume that the conclusions of \Cref{dFprop} and \Cref{dWisLipsc} hold true. Then the following hold:
\begin{itemize}
\item[$(i)$] if $\{u_n\}_{n \in \N} \subset H^1(\o; \R^M)$ is such that
\[
\sup \left \{\f_n(u_n) : n \in \N \right\} < \infty,
\]
then, eventually extracting a subsequence (which we do not relabel), we have that $u_n \to u$ in $L^1(\o;\R^M)$, where $u$ is such that $u(x) \in \{z_1(x), \dots, z_k(x)\}$ for $\mathcal{L}^N$-a.e.\@ $x \in \o$, and $\widetilde{\f_0}(u) < \infty$.
\item[$(ii)$] The sequence of functionals $\f_n$ $\Gamma$-converges with respect to the $L^1$-topology to $\widetilde{\f_0}$.
\end{itemize}
\end{theorem}


\subsection{Outline of the paper}

The paper is organized as follows.
In \Cref{sec:preliminaries} we introduce the notation and we recall the definitions of the mathematical objects we will need for our analysis.
The Lipschitz continuity of the function $x\mapsto d_W(x,z_i(x), z_j(x))$ is shown in \Cref{dWisLipsc}. The proof makes use of a result obtained in the first part of \Cref{sec:tech}, namely the fact that geodesics for $d_W$ (and also for more degenerate conformal factors) joining two points in a compact subset of $\R^M$ have uniformly bounded Euclidean lenght (see \Cref{dFprop}).
 The proof of \Cref{thm:main} is divided in three parts: in \Cref{sec:cmpt} we prove the compactness result, while \Cref{sec:liminf} and \Cref{sec:limsup} are devoted at obtaining the liminf and the limsup inequality, respectively.
Finally, in \Cref{sec:others} we discuss how to suitably modify the arguments we used to prove \Cref{thm:main} in order to obtain \Cref{thm:massconstr}, \Cref{thm:Dirichlet}, and \Cref{noH2H3}.


\section{Preliminaries}\label{sec:preliminaries}
For the convenience of the reader, in this section we collect a few definitions and tools used throughout the paper.

\subsection{Radon measures} 
Let $\mathcal{M}(\o)$ be the space of finite Radon measures on $\o$. We recall that in view of the Riesz representation theorem (see, for example, \cite[Theorem 1.200]{FL}), if we denote by $C_0(\o)$ the completion with respect to the $L^\infty$ norm of the space of continuous functions with compact support in $\o$, then the dual of $C_0(\o)$ can be identified with $\mathcal{M}(\o)$. The subset of $\mathcal{M}(\o)$ consisting of all finite nonnegative Radon measures on $\o$ will be denoted by $\mathcal{M}^+(\o)$. For the sake of brevity, the results of this section are stated in the form that will be used in the paper; for this reason we refer the reader to the monographs \cite{EG} and \cite{FL} for a more detailed treatment of these topics. 

\begin{definition}
We say that a sequence $\{\mu_n\}_{n\in\N}\subset\mathcal{M}^+(\o)$ \emph{weakly-$*$ converges} to
$\mu\in \mathcal{M}^+(\o)$, and we write $\mu_n\stackrel{*}{\rightharpoonup}\mu$, if
\[
\lim_{n\to\infty}\int_\o \varphi \,d\mu_n\to\int_\o \varphi \,d\mu
\]
for all $\varphi\in C_0(\o)$. 
\end{definition}

The first result of this section gives a simple criterion for weak-$*$ compactness of measures. For a proof see \cite[Proposition 1.202]{FL}.

\begin{theorem}\label{thm:wscmpt}
Let $\{\mu_n\}_{n \in \N} \subset \mathcal{M}^+(\o)$ be a sequence of finite nonnegative Radon measures such that
\[
\sup\left\{\mu_n(\o) : n \in \N \right\} < \infty.
\]
Then there exist a subsequence (not relabeled) and a measure $\mu \in \mathcal{M}^+(\o)$ such that $\mu_n \stackrel{*}{\rightharpoonup}\mu$.
\end{theorem}

The following lemma is a key ingredient in the proof of the liminf inequality (see \Cref{prop:liminf}). For a proof see \cite[Theorem 1.203(iii)]{FL}.

\begin{lemma}\label{lem:convmeas}
Let $\{\mu_n\}_{n \in \N} \subset \mathcal{M}^+(\o)$ be a sequence of finite nonnegative Radon measures such that $\mu_n\stackrel{*}{\rightharpoonup}\mu$.
Then
\[
\lim_{n \to \infty} \mu_n(A) = \mu(A)
\]
for every Borel set $A \subset \o$ with $\mu(\partial A) = 0$.
\end{lemma}

\begin{remark}
\label{mubdry=0}
For our purposes, the condition $\mu(\partial A) = 0$ in \Cref{lem:convmeas} is not very restrictive. Indeed, fix $x \in \o$ and let $E$ be an open convex set that contains $x$. Consider the family $\{E_\rho\}_{\rho > 0}$ of rescaled copies of $E$, \emph{i.e.\@}, let $E^\rho \coloneqq x + \rho (E - x)$.
Since by assumption $\mu$ is a \emph{finite} Radon measure, it is immediate to verify that the set
\[
\{ \rho > 0 : \mu(\partial E_\rho) > 0 \}
\]
is at most countable.
Indeed, take $\overline{\rho} > 0$ such that $E_{\overline{\rho}} \subset \o$.
Then, for each $m \in \N$, consider the set
\[
A_m \coloneqq \left\{ \rho\in(0,\overline{\rho}) : \frac{1}{m+1} < \mu(\partial E_\rho) \leq \frac{1}{m} \right\}.
\]
Then
\[
\sum_{m = 1}^{\infty} \frac{1}{m+1}\mathcal{L}^0(A_m) \leq \mu(E_{\overline{\rho}}) < \infty,
\]
yielding that each $A_m$ is at most finite (and not all of them can be non-empty). In particular, if $\mu_n\stackrel{*}{\rightharpoonup}\mu$, the argument above shows that
\[
\lim_{n \to \infty}\mu_n(E_{\rho}) = \mu(E_{\rho})
\]
for all but at most countably many values of $\rho < \overline{\rho}$.
\end{remark}

In an approximation result needed for the limsup inequality (see \Cref{prop:limsup}) we will also make use of another notion of convergence for measures.

\begin{definition}\label{def:Cb}
Let $\{\mu_n\}_{n \in \N} \subset \mathcal{M}^+(\o)$ be a sequence of finite nonnegative Radon measures. We say that $\mu_n$ converges in $(C_b(\o))'$ to $\mu \in \mathcal{M}^+(\o)$ if
\[
\lim_{n\to\infty}\int_\o \varphi \,d\mu_n \to \int_\o \varphi \,d\mu
\]
for all $\varphi\in C_b(\o)$. 
Here $C_b(\o)$ denotes the space of continuous bounded functions on $\o$.
\end{definition}

Since $C_0(\o)\subset C_b(\o)$, if $\mu_n$ converges in $(C_b(\o))'$ to $\mu$, then $\mu_n\stackrel{*}{\rightharpoonup}\mu$.
The opposite is true if, in addition, we know that the limit measure does not charge the boundary of $\o$, as shown in the next result (for a proof see \cite[Proposition 1.206]{FL}).

\begin{lemma}
Let $\{\mu_n\}_{n \in \N} \subset \mathcal{M}^+(\o)$ be a sequence of finite nonnegative Radon measures such that $\mu_n\stackrel{*}{\rightharpoonup}\mu$ for some $\mu \in \mathcal{M}^+(\o)$, and assume that
\[
\lim_{n\to\infty}\mu_n(\o)=\mu(\o).
\]
Then $\mu_n$ converges in $(C_b(\o))'$ to $\mu$.
\end{lemma}

We conclude this list of results on Radon measures with a well-known theorem from measure theory. The result presented below allows to recover the absolutely continuous part of a measure with respect to another via a differentiation process. For a proof we refer the reader to \cite[Theorem 1.153 and Remark 1.154]{FL}.

\begin{theorem}[Besicovitch derivation theorem]\label{thm:Bes}
Let $\mu,\nu \in \mathcal{M}^+(\o)$, and write $\nu=\nu^{ac}+\nu^s$, where $\nu^{ac}\ll\mu$, and $\nu^s\perp\mu$.
Let $C\subset\R^N$ be an open convex set that contains the origin.
Then there exists a Borel set $S\subset\o$ with $\mu(S)=0$ such that
\[
\frac{d\nu^{ac}}{d\mu}(x) = \lim_{\rho\to0} \frac{\nu(x+\rho C)}{\mu(x+\rho C)} \in[0,\infty),
\]
and 
\[
\lim_{\rho\to0} \frac{\nu^s(x+\rho C)}{\mu(x +\rho C)}=0,
\]
for all $x\in\o\setminus S$.
\end{theorem}

\subsection{Functions of bounded variation} 
We start by recalling basic definitions and well known properties of functions of bounded variation and sets of finite perimeter. We refer the reader to \cite{AFP} for more details.

\begin{definition}
Let $u \in L^1(\o;\R^M)$.
We say that $u$ is a function \emph{of bounded variation} if its distributional derivative $Du$ is a finite matrix-valued Radon measure on $\o$.
In particular,
\[
|Du|(\o) = \sup\left\{ \sum_{i = 1}^M \int_\o u^i(x) \mathrm{div}\varphi^i(x)\, dx :
	\varphi \in C^\infty_c(\o;\R^{M\times N}), \|\varphi\|_{L^\infty}\leq 1 \right\}.
\]
In this case we write $u \in BV(\o; \R^M)$.
\end{definition}

\begin{definition}\label{def:jump}
Let $u \in L^1(\o;\R^M)$. We define $J_u \subset \o$, the \emph{jump set} of $u$, as the set of points $x \in \o$ such that there exist distinct vectors $a, b \in \R^M$ and a direction $\nu \in \S^{N - 1}$ for which
\[
\lim_{\rho \to 0^+}\frac{1}{\rho^N}\int_{B^+(x,\nu, \rho)} |u(y) - a|\,dy
= \lim_{\rho \to 0^+}\frac{1}{\rho^N}\int_{B^-(x,\nu, \rho)} |u(y) - b|\,dy = 0,
\]
where
\[
B^+(x,\nu, \rho) \coloneqq \nu_+(x)\cap B(x,\rho),
\quad\quad\quad
B^-(x,\nu, \rho)\coloneqq\nu_-(x)\cap B(x,\rho),
\]
and
\[
\nu_+(x) \coloneqq \{y \in \R^N: (y-x) \cdot \nu \ge 0 \}, \quad\quad\quad
\nu_-(x) \coloneqq \{y \in \R^N : (y-x) \cdot \nu \le 0 \}.
\]
If $x\in J_u$, we denote the triple $(a, b,\nu)$ as $(u^+(x), u^-(x), \nu_u(x))$. Notice that $(u^+(x), u^-(x),\nu_u(x))$ is unique up to replacing $\nu_u(x)$ with $-\nu_u(x)$ and interchanging $u^+(x)$ and $u^-(x)$.
\end{definition}

\begin{remark}
Notice that using cubes instead of balls in \Cref{def:jump} yields an analogous characterization of the jump set. In order to keep the notation as simple as possible, in the proof of the liminf inequality (see \Cref{prop:liminf}) it will be convenient to consider cubes with two faces orthogonal to the vector $\nu_u$.
\end{remark}

The next result concerns the structure of the jump set and the decomposition of the distributional derivative of a function of bounded variation. For a proof see \cite[Theorem 3.78]{AFP}.

\begin{theorem}[Federer--Vol'pert]
The jump set $J_u$ of a function $u\in BV(\o;\R^M)$ is countably $\hno$-rectifiable, \emph{i.e.\@}, there exist Lipschitz continuous functions $f_i \colon \R^{N-1} \to \R^N$ such that
\[
\hno \left(J_u \setminus \bigcup_{i=1}^{\infty} f_i(K_i)\right) = 0,
\]
where each $K_i$ is a compact subset of $\R^{N-1}$. Moreover, 
\[
Du = \nabla u \mathcal{L}^N + (u^+ - u^-) \otimes \nu_u \hno\restr J_u + D^c u,
\]
where $D^c u$ denotes the so called \emph{Cantor part} of the distributional derivative.
\end{theorem}


We now focus on the special class of functions of bounded variations which consists of characteristic functions of sets.

\begin{definition}
Let $E\subset\o$.
We say that $E$ has \emph{finite perimeter} in $\o$ if its characteristic function $\ca_E \colon \o \to \{0,1\}$, defined as
\[
\ca_E(x)\coloneqq
\left\{
\begin{array}{ll}
1 & \text{ if } x\in E,\\
0 & \text{ otherwise},
\end{array}
\right.
\]
 is of bounded variation in $\o$.
\end{definition}

For sets of finite perimeter, we have two notions of boundary coming from measure theory.

\begin{definition}
Let $E\subset\R^N$ be a set of finite perimeter in $\o$.
We call \emph{reduced boundary} of $E$, denoted with $\mathcal{F} E$, the set of points 
$x \in \mathrm{supp}|D \ca_E|\cap\o$ for which the limit
\[
\nu_E(x) \coloneqq \lim_{\rho\to0}\frac{D\ca_E(B(x,\rho))}{|D\ca_E|(B(x,\rho))}
\]
exists in $\R^N$ and is such that $|\nu_E(x)|=1$.
\end{definition}

\begin{definition}\label{def:partialstar}
Let $E\subset \R^N$ be an $\mathcal{L}^N$-measurable set.
For $t\in[0,1]$ we define
\[
E^t \coloneqq \left\{   x\in \R^N : \lim_{\rho\to0}\frac{\mathcal{L}^N(E\cap B(x,\rho))}{\mathcal{L}^N(B(x,\rho))} = t \right\},
\]
the set of points of density $t$ for $E$.
The set $\partial^* E\coloneqq\R^N\setminus (E^1\cup E^0)$ is called the \emph{essential boundary} of $E$.
\end{definition}

The relation between these two notions of measure theoretic boundary is specified in the following theorem (see \cite[Theorem 3.61]{AFP}).

\begin{theorem}[Federer]\label{thm:Fed}
Let $E\subset\R^N$ be a set of finite perimeter in $\o$.
Then
\[
\mathcal{F} E \subset \partial^{1/2} E \subset \partial^* E,
\]
and
\[
\hno\left(\o\setminus (E^0\cup\mathcal{F}E\cup E^1) \right)=0.
\]
\end{theorem}

In a similar fashion as the Federer-Vol'pert theorem, the reduced boundary enjoys some structure properties (see \cite[Theorem 3.59]{AFP}).

\begin{theorem}[De Giorgi]\label{thm:DG}
Let $E \subset \R^N$ be a set of finite perimeter in $\o$ and for all $x_0 \in \mathcal{F}E$ and $\rho > 0$ let $E_\rho \coloneqq \frac{E-x_0}{\rho}$. Then $\mathcal{F}E$ is countably $\hno$-rectifiable and 
\[
\ca_{E_\rho} \to \ca_{H}
\]
locally in $L^1(\o)$ as $\rho\to0$, where $H \coloneqq \{ x\in\R^N : x\cdot \nu_E(x_0)\geq 0 \}$. Moreover, 
\[
\lim_{\rho\to0}\frac{\hno(\mathcal{F}E\cap Q(x_0,\rho))}{\rho^{N-1}} = 1.
\]
\end{theorem}

Finally, we define the notion of Caccioppoli partitions. This will be useful in the approximation results in order to get the limsup inequality (see \Cref{sec:limsup}).

\begin{definition}
\label{CP}
A partition $\{\o^i\}_{i \in \N}$ of $\o$ is called a \emph{Caccioppoli partition} if each $\o^i$ is a set of finite perimeter in $\o$.
\end{definition}

\begin{remark}
\label{CPR}
Notice that for every $u \in BV(\o; z_1, \dots, z_k)$ (see \Cref{f0}) there exists a Caccioppoli partition $\{\o^1,\dots,\o^k\}$ of $\o$ such that $u(x) = z_i(x)$ for $\mathcal{L}^N$-a.e.\@ $x \in \o^i$, for all $i \in \{1,\dots,k\}$, and
\[
\hno\restr J_u = \sum_{i=1}^{k-1} \sum_{j=i+1}^k \hno\restr(\partial^*\o^i\cap\partial^*\o^j ).
\]
Moreover, using \cite[Theorem 3.84]{AFP}, it is possible to show that the distributional derivative of each $u\in BV(\o;z_1,\dots,z_k)$ has no Cantor part.
\end{remark}


\subsection{\texorpdfstring{$\Gamma$}{Gamma}-convergence}
We now recall the basic definition and some properties of $\Gamma$-convergence that will be used throughout the paper (for a reference see \cite{B,DM}).

\begin{definition}
Let $(X, d)$ be a metric space. We say that a sequence of functions $F_n \colon X \to \R\cup\{\infty\}$ $\Gamma$-converges to $F \colon X \to \R\cup\{\infty\}$, and we write $F_n \stackrel{\Gamma-d}{\longrightarrow} F$, if the following hold:
\begin{itemize}
\item[$(i)$] for every $x \in X$ and every sequence $\{x_n\}_{n \in \N}$ of elements of $X$ such that $x_n \to x$ we have
\[
F(x) \le \liminf_{n \to \infty} F_n(x_n);
\]
\item[$(ii)$] for every $x \in X$ there exists a sequence $\{x_n\}_{n\in\N}$ of elements of $X$ such that $x_n \to x$ and
\[
\limsup_{n \to \infty} F_n(x_n)\le F(x).
\]
\end{itemize} 
A sequence $\{x_n\}_{n\in\N}$ as in $(ii)$ is called a \emph{recovery sequence} for $x$.
\end{definition}

We recall that the definition of $\Gamma$-convergence is primarily motivated by seeking minimal conditions which guarantee the convergence of minima and minimizers for a family of functionals (see, for example, \cite[Corollary 7.20]{DM}). This is specified in the following theorem.

\begin{theorem}\label{thm:convmin}
Let $(X,d)$ be a metric space, let $F_n, F \colon X \to \R\cup\{\infty\}$ and assume that $F_n \stackrel{\Gamma-d}{\longrightarrow} F$. For each $n \in \N$, let $x_n \in X$ be a minimizer of $F_n$ on $X$. Then every cluster point of $\{x_n\}_{n \in \N}$ is a minimizer of $F$ and
\[
\lim_{n \to \infty} F_n(x_n) = \min \{F(x) : x \in X\}.
\]
\end{theorem}


\section{Existence of minimizing geodesics}\label{sec:tech}
The purpose of this section is to collect some preliminary results concerning the existence of minimizing geodesics for possibly degenerate metrics with conformal factor $F$. To be precise, given a continuous nonnegative function $F$ and $p, q \in \R^M$, we study the minimization problem
\begin{equation}
\label{dF}
d_F(p,q) \coloneqq \inf \left\{\int_{-1}^1F(\gamma(t))|\gamma'(t)|\,dt : \gamma \in \mathcal{A}(p,q)\right\},
\end{equation}
where the class of admissible parametrizations $\mathcal{A}(p,q)$ is given as in \Cref{defdW}. Notice that the value of the integral on the right-hand side of (\ref{dF}) is a purely geometric quantity, \emph{i.e.}, it is independent of the choice of the parametrization. Throughout the rest of the paper, we refer to any function $\gamma \in \mathcal{A}(p,q)$ for which the infimum on the right-hand side of (\ref{dF}) is achieved as a \emph{minimizing geodesic}. Moreover, we use the phrase \emph{sequence of almost minimizing geodesics} to denote a sequence in $\mathcal{A}(p,q)$ for which the infimum is achieved in the limit.

Let us remark that the existence of minimizing geodesics for (\ref{dF}) has been previously investigated by many authors. We mention here the work of Zuniga and Sternberg \cite{ZunSter}, where existence of solutions to the minimization problem is shown under very general assumptions on the conformal factor $F$. Of particular interest for our analysis is the special case where the conformal factor is given by
\begin{equation}
\label{F=W}
F(z) \coloneqq 2 \sqrt{W(x,z)}.
\end{equation}
Indeed, we observe that for a fixed value of $x \in \o$, if $F$ is given as above, the distance function $d_F$ defined in (\ref{dF}) is identically equal to the function $d_W(x, \cdot, \cdot)$, introduced in \Cref{defdW}. 

As the proofs of our main results rely on a precise understanding of minimizing geodesics for (\ref{dF}), and in particular on their dependence on the variable $x$ when $F$ is chosen as in (\ref{F=W}), compared to \cite{ZunSter} we require more stringent assumptions on the behavior of the potential $W$ near the wells (see (\ref{H3})). In turn, our approach is in spirit closer to that of Sternberg \cite{Ste_Vect}, where the author considered a singular perturbation of the conformal factor which renders the associated Riemannian metric conformal to the Euclidean metric and proceeded to prove a uniform bound with respect to the perturbation parameter. Our method, on the other hand, consists of proving a uniform bound on the Euclidean length of a sequence of almost minimizing geodesics.
For technical reasons, we will need to consider conformal factors of the form
\[
F(z) \coloneqq \inf \left\{2\sqrt{W(x,z)} : x \in \mathcal{R} \right\},
\]
where $\mathcal{R} \subset \o$.

In the following, given a function $\gamma \in W^{1,1}(\mathcal{I};\R^M)$, where $\mathcal{I} \subset \R$ is an open interval, we work with its representative in $AC(\overline{\mathcal{I}};\R^M)$ and denote its Euclidean length by $L(\gamma)$, \emph{i.e.}, 
\begin{equation}
\label{lengthfunctional}
L(\gamma) \coloneqq \int_{\mathcal{I}}|\gamma'(t)|\,dt.
\end{equation}
For any two points $p,q \in \R^M$, we let $\ell_{p,q}$ be a parametrization in $\mathcal{A}(p,q)$ of the line segment that joins $p$ to $q$. To be precise, for $t \in [-1,1]$ we let
\begin{equation}
\label{lpq}
\ell_{p,q}(t) \coloneqq \frac{1 - t}{2}p + \frac{1 + t}{2}q.
\end{equation}

We begin by presenting a compactness criterion for almost minimizing geodesics. The result states that the existence of a sequence of almost minimizing geodesic for (\ref{dF}) with a uniform bound on the Euclidean length of each element in the sequence implies the existence of a minimizing geodesic which enjoys the same bound. The proof is adapted from the classical result on the existence of shortest paths, \emph{i.e.\@}, minimizers of the length functional (\ref{lengthfunctional}) (see, for example, \cite[Theorem 5.38]{L}).

\begin{lemma}
\label{L<k}
Given a continuous function $F \colon \R^M \to [0,\infty)$ and $p,q \in \R^M$, let $d_F(p,q)$ be given as in $(\ref{dF})$, $\{\gamma_n\}_{n \in \N} \subset \mathcal{A}(p,q)$ be a sequence of almost minimizing geodesics, \emph{i.e.}, 
\[
d_F(p,q) = \lim_{n \to \infty}\int_{-1}^1F(\gamma_n(t))|\gamma_n'(t)|\,dt,
\] 
and furthermore assume that $L(\gamma_n) \le \Lambda$ for some positive constant $\Lambda$ independent of $n$. Then there exists $\gamma \in \mathcal{A}(p,q)$ with $L(\gamma) \le \Lambda$ such that 
\[
d_F(p,q) = \int_{-1}^1F(\gamma(t))|\gamma'(t)|\,dt.
\]
\end{lemma}
\begin{proof}
Notice that if $p = q$ then there is nothing to do. Thus, we can assume without loss of generality that $\gamma_n \colon [-1,1] \to \R^M$ is a parametric representation of a continuous simple rectifiable curve. In turn, it can be parametrized by arclength, \emph{i.e.}, there exists a function $\varphi_n \colon [0,L(\gamma_n)] \to [-1,1]$ with the property that 
\[
v_n(s) \coloneqq \gamma_n(\varphi_n(s))
\]
is Lipschitz continuous, and in particular $|v_n'(s)| = 1$ for $\mathcal{L}^1$-a.e.\@ $s \in (0,L(\gamma_n))$. Eventually extracting a subsequence (which we do not relabel), we can assume that $L(\gamma_n) \to \lambda$ for some $\lambda > 0$. Let $\psi_n \colon [-1,1] \to [0, L(\gamma_n)]$ be defined via 
\[
\psi_n(t) \coloneqq \frac{(t + 1)L(\gamma_n)}{2},
\]
and set $w_n(t) \coloneqq v_n(\psi_n(t))$. Notice that the functions $w_n \in \mathcal{A}(p,q)$ and satisfy 
\begin{equation}
\label{|wn'|}
|w_n'(t)| = \frac{L(\gamma_n)}{2} \le \frac{\lambda + 1}{2}
\end{equation}
for $\mathcal{L}^1$-a.e.\@ $t \in (-1,1)$ and every $n$ sufficiently large. Consequently, we are in a position to apply the Ascoli-Arzel\'a theorem to find a function $\gamma \colon [-1,1] \to \R^M$ and a further subsequence (not relabeled) such that $w_n \to \gamma$ uniformly. Furthermore, since $L(\cdot)$ is lower semicontinuous with respect to pointwise convergence, we also get
\[
L(\gamma) \le \liminf_{n \to \infty}L(\gamma_n) = \lim_{n \to \infty}L(\gamma_n) = \lambda \le \Lambda.
\]
Finally, in view of (\ref{|wn'|}), we notice that for every $s,t \in (-1,1)$ we have
\[
|\gamma(s) - \gamma(t)| = \lim_{n \to \infty}|w_n(s) - w_n(t)|\le \lim_{n \to \infty}\frac{L(\gamma_n)}{2}|s - t| = \frac{\lambda}{2}|s - t|,
\]
and thus
\begin{align*}
\int_{-1}^1F(\gamma(t))|\gamma'(t)|\,dt \le \frac{\lambda}{2} \int_{-1}^1F(\gamma(t))\,dt & = \left(\lim_{n \to \infty}\frac{L(\gamma_n)}{2}\right)\left( \lim_{n \to \infty}\int_{-1}^1F(w_n(t))\,dt \right)\\
& = \lim_{n \to \infty} \int_{-1}^1F(w_n(t))|w_n'(t)|\,dt \\
& = \lim_{n \to \infty} \int_{-1}^1F(\gamma_n(t))|\gamma_n'(t)|\,dt = d_F(p,q).
\end{align*}
This concludes the proof.
\end{proof}

With \Cref{L<k} in hand, we can turn our attention back to the minimization problem (\ref{dF}).

\begin{proposition}
\label{dFprop}
Let $W$ be given as in \emph{(\ref{H1})--(\ref{H3})}, and assume that \eqref{Hold4} holds for some $\eta > 0$.
Let $\mathcal{R}$ be a convex compact subset of $\o$, and denote by $z_i(\mathcal{R})$ the set of points $z \in \R^M$ such that $z = z_i(x)$ for some $x \in \mathcal{R}$. Assume that 
\begin{equation}
\label{z_iRfar}
\mathcal{N}_{\delta/2}(z_i(\mathcal{R})) \cap \mathcal{N}_{\delta/2}(z_j(\mathcal{R})) = \emptyset
\end{equation}
whenever $i \neq j$, where 
\[
\mathcal{N}_{\rho}(z_i(\mathcal{R})) \coloneqq \left\{z \in \R^M : |z - z_i(x)| \le \rho \ \text{ for some } x \in \mathcal{R}\right\},
\]
define
\[
F(z) \coloneqq \min\left\{2\sqrt{W(x,z)} : x \in \mathcal{R}\right\},
\]
and let $d_F \colon \R^M \times \R^M \to [0,\infty)$ be given as in $(\ref{dF})$. Then for every $p,q \in \R^M$ there exists a minimizing geodesic $\gamma \in \mathcal{A}(p,q)$ for $d_F(p,q)$ such that 
\begin{equation}
\label{LS}
L(\gamma) \le k\sigma + \diam(\mathcal{R})\sum_{i = 1}^k\Lip(z_i) + \frac{d_F(p,q) + 1}{\Sigma_\sigma(\mathcal{R})},
\end{equation}
where, for $\eta$ as in $(\ref{Hold4})$,
\begin{equation}
\label{sigma}
\sigma \coloneqq \frac{1}{2} \min\left\{r, \delta, \frac{\min_i \alpha_i}{\max_i \alpha_i}\delta, \sqrt{\frac{\eta}{\max_i\alpha_i}}\right\}
\end{equation}
and 
\begin{equation}
\label{Sigma}
\Sigma_{\sigma}(\mathcal{R}) \coloneqq \min\left\{F(z) : z \notin \bigcup_{i = 1}^k\mathcal{N}_{\sigma/2}(z_i(\mathcal{R}))\right\}.\end{equation}
\end{proposition}

\begin{proof} In the following we let $d_i(z)$ denote the distance from a point $z \in \R^M$ to the set $z_i(\mathcal{R})$. We recall that the functions $d_i \colon \R^M \to [0,\infty)$ are Lipschitz continuous with Lipschitz constant at most 1. We divide the proof into several steps.
\newline
\textbf{Step 1:} We begin by showing that for $\sigma$ as in (\ref{sigma}), if $z \in \mathcal{N}_{\sigma}(z_i(\mathcal{R}))$ then 
\begin{equation}
\label{F=dist}
F(z) = 2\sqrt{\alpha_i}d_i(z).
\end{equation}
To this end, for $z \in \mathcal{N}_{\sigma}(z_i(\mathcal{R}))$ define $A \coloneqq \{x \in \mathcal{R} : |z - z_i(x)| \le r\}$ and $B \coloneqq \{y \in \mathcal{R} : |z - z_i(y)| > r\}$. We claim that 
\begin{equation}
\label{supAinfB}
\sup\left\{W(x,z) : x \in A\right\} \le \inf \left\{W(y,z) : y \in B\right\}.
\end{equation}
Indeed, for every $x \in A$ we have
\begin{equation}
\label{Wsigma1}
W(x,z) = \alpha_i|z - z_i(x)|^2 \le \alpha_i \sigma^2,
\end{equation}
while for $y \in B$, by (\ref{z_iRfar}) we obtain 
\begin{equation}
\label{Wsigma2}
W(y,z) = \alpha_j|z - z_j(y)| \ge \alpha_jd_j(z)^2 \ge \alpha_j\frac{\delta^2}{4},
\end{equation}
provided that $|z - z_j(y)| \le r$ for some $j \neq i$, and $W(y,z) \ge \eta$ otherwise, where $\eta > 0$ is the constant introduced in (\ref{Hold4}). In turn, inequality (\ref{supAinfB}) follows from (\ref{sigma}), (\ref{Wsigma1}), and (\ref{Wsigma2}). Notice in particular that (\ref{supAinfB}) implies that
\[
F(z) = \min\left\{2\sqrt{W(x,z)} : x \in A\right\} = \min\left\{2 \sqrt{\alpha_i}|z - z_i(x)| : x \in A\right\},
\]
and (\ref{F=dist}) readily follows.
\newline
\textbf{Step 2:} In this step we show that if $p \in \mathcal{N}_{\sigma}(z_i(\mathcal{R}))$, and $q \in z_i(\mathcal{R})$ realizes the distance, \emph{i.e.}, $d_i(p) = |p - q|$, then the line segment that joins $p$ and $q$ is a minimizing geodesic for $d_{F}(p,q)$. To see this, let $\gamma \in \mathcal{A}(p,q)$ and notice that the map $t \mapsto d_i(\gamma(t))$ is continuous, $d_i(\gamma(-1)) = |p - q|$, and $d_i(\gamma(1)) = 0$. Thus, by the mean value theorem, for every $y \in (0,|p - q|)$ there exists $t \in (-1,1)$ such that $d_i(\gamma(t)) = y$. We recall that the composite function $d_i \circ \gamma$ belongs to the space $W^{1,1}((-1,1))$ and that 
\begin{equation}
\label{CRineq}
\left|\frac{d}{dt}d_i(\gamma(t))\right| \le |\gamma'(t)|
\end{equation}
for $\mathcal{L}^1$-a.e.\@ $t \in (-1,1)$. 
Indeed, as remarked above, the function $d_i$ is Lipschitz continuous with Lipschitz constant at most 1, and thus by Rademacher's theorem it is differentiable for $\mathcal{L}^1$-a.e.\@ $t \in (-1,1)$.
Moreover, the composite function $d_i \circ \gamma \colon (-1,1) \to \R$ is absolutely continuous and therefore it is also differentiable for $\mathcal{L}^1$-a.e.\@ $t \in (-1,1)$.
If we now let $t \in (-1,1)$ be a point where both $d_i$ and $d_i \circ \gamma$ are differentiable we see that
\begin{align*}
\left|\frac{d}{dt}d_i(\gamma(t))\right| & = \left|\lim_{|h| \to 0}\frac{d_i(\gamma(t + h)) - d_i(\gamma(t))}{|h|}\right| \\
& = \lim_{|h| \to 0}\frac{|d_i(\gamma(t + h)) - d_i(\gamma(t))|}{|h|} \le \limsup_{|h| \to 0}\frac{|\gamma(t + h) - \gamma(t)|}{|h|} = |\gamma'(t)|,
\end{align*}
which gives \eqref{CRineq}.
By (\ref{F=dist}), (\ref{CRineq}), and by an application of the co-area formula (see \cite[Theorem 3.2.6]{AFP}), we get
\begin{align*}
\int_{-1}^1F(\gamma(t))|\gamma'(t)|\,dt & \ge 2\sqrt{\alpha_i} \int_{\{0 < d_i(\gamma(t)) < |p - q|\}}d_i(\gamma(t))|\gamma'(t)|\,dt \\
& \ge 2\sqrt{\alpha_i} \int_{\{0 < d_i(\gamma(t)) < |p - q|\}}d_i(\gamma(t))\left|\frac{d}{dt}d_i(\gamma(t))\right|\,dt \\
& = 2\sqrt{\alpha_i}\int_0^{|p - q|}y\mathcal{H}^0(\{t : d_i(\gamma(t)) = y\})\,d y \\
& \ge 2\sqrt{\alpha_i}\int_0^{|p - q|}y\,d y = \sqrt{\alpha_i}|p - q|^2.
\end{align*}
On the other hand, if we let $\ell_{p,q}$ be given as in (\ref{lpq}), as one can readily check, we have that $d_i(\ell_{p,q}(t)) = |\ell_{p,q}(t) - q|$ for every $t \in [-1,1]$ and $\mathcal{H}^0(\{t : |\ell_{p,q}(t) - q| = y\}) = 1$ for every $y \in (0, |p - q|)$. In turn, by the co-area formula we conclude that 
\begin{align}
\label{linetoz}
\int_{-1}^1F(\ell_{p,q}(t))|\ell_{p,q}'(t)|\,dt  & = 2\sqrt{\alpha_i}\int_{-1}^1d_i(\ell_{p,q}(t))|\ell_{p,q}'(t)|\,dt \notag \\
& = 2\sqrt{\alpha_i}\int_{-1}^1|\ell_{p,q}(t) - q||\ell_{p,q}'(t)|\,dt \notag \\
& = 2\sqrt{\alpha_i}\int_0^{|p - q|}y\mathcal{H}^0(\{t : |\ell_{p,q}(t) - q| = y\})\,d y \notag \\
& = \sqrt{\alpha_i}|p - q|^2.
\end{align}
This proves our claim.
\newline
\textbf{Step 3:} If $p,q \in z_i(\mathcal{R})$ then there are $x_1,x_2 \in \mathcal{R}$ such that $z_i(x_1) = p$ and $z_i(x_2) = q$. Let 
\[
\psi(t) \coloneqq \frac{1 - t}{2}x_1 + \frac{1 + t}{2}x_2,
\]
and notice that $\psi(t) \in \mathcal{R}$ for every $t \in [-1,1]$ since $\mathcal{R}$ is convex by assumption. Let $\gamma \colon [-1,1] \to \R^M$ be defined via $\gamma(t) \coloneqq z_i(\psi(t))$. Then $\gamma \in \mathcal{A}(p,q)$ is a minimizing geodesic for $d_F(p,q)$, and furthermore
\[
L(\gamma) \le \Lip(z_i)\diam(\mathcal{R}).
\]
Consequently, we see that if $p \in \mathcal{N}_{\sigma}(z_i(\mathcal{R}))$ and $q \in z_i(\mathcal{R})$ then any parametrization in $\mathcal{A}(p,q)$ of the line segment from $p$ to a closest point on $z_i(\mathcal{R})$, namely $p'$, together with any curve with support contained in $z_i(\mathcal{R})$ that connects $p'$ to $q$ gives a minimizing geodesic for $d_{F}(p,q)$.
\newline
\textbf{Step 4:} This step is concerned with the proof of the more delicate case where $p$ and $q$ are distinct points in $\mathcal{N}_{\sigma/2}(z_i(\mathcal{R}))$, neither of which lies on $z_i(\mathcal{R})$. Throughout the step, we assume without loss of generality that $d_i(p) \ge d_i(q)$. Let $\{\gamma_n\}_{n \in \N} \subset \mathcal{A}(p,q)$ be a sequence of almost minimizing geodesics for $d_F(p,q)$ and observe that it is possible to assume that 
\begin{equation}
\label{stay}
\gamma_n(t) \in \mathcal{N}_{\sigma}(z_i(\mathcal{R}))
\end{equation}
for every $(n,t) \in \N \times [-1,1]$. Indeed, if this is not the case then we can find two disjoint subintervals of $(-1,1)$, namely $\mathcal{I}_1 \coloneqq (s_1,t_1)$ and $\mathcal{I}_2 \coloneqq (s_2,t_2)$, such that 
\[
2d_i(\gamma_n(s_1)) = 2d_i(\gamma_n(t_2)) = d_i(\gamma_n(t_1)) = d_i(\gamma_n(s_2)) = \sigma
\]
and 
\[
\frac{\sigma}{2} \le d_i(\gamma_n(t)) \le \sigma
\] 
for all $t \in \mathcal{I}_1 \cup \mathcal{I}_2$. In turn, we have
\begin{equation}
\label{staysigma}
\int_{-1}^1F(\gamma_n(t))|\gamma_n'(t)|\,dt \ge 2\sqrt{\alpha_i} \int_{\mathcal{I}_1 \cup \mathcal{I}_2}d_i(\gamma(t))|\gamma'(t)|\,dt \ge \sqrt{\alpha_i}\sigma^2.
\end{equation}
On the other hand, let $p',q' \in z_i(\mathcal{R})$ be such that $d_i(p) = |p - p'|$ and $d_i(q) = |q - q'|$, and for $x_1, x_2 \in \mathcal{R}$ such that $p' = z_i(x_1)$ and $q' = z_i(x_2)$ define 
\begin{equation}
\label{competitorgeo}
\gamma(t) \coloneqq
\left\{
\begin{array}{ll}
\displaystyle (1 - 2t - 2)p + (2t + 2)p' & \text{ if } -1 < t < -1/2,\\
& \\
\displaystyle z_i\left(\left(\frac{1}{2} - t\right)x_1 + \left(t + \frac{1}{2}\right)x_2\right) & \text{ if } -\frac{1}{2} \le t < \frac{1}{2}, \\
& \\
\displaystyle (2 - 2t)q' + (2t - 1)q & \text{ if } \frac{1}{2} \le t < 1.
\end{array}
\right.
\end{equation}
Then $\gamma \in \mathcal{A}(p,q)$ is a parametric representation of the curve whose support constitutes of the line segments that join $p$ to $p'$ and $q'$ to $q$, together with an arc in $z_i(\mathcal{R})$ that connects $p'$ and $q'$. By means of a direct computation we see that 
\begin{equation}
\label{compG}
\int_{-1}^1F(\gamma(t))|\gamma'(t)|\,dt = \sqrt{\alpha_i}\left(d_i(p)^2 + d_i(q)^2\right).
\end{equation}
Comparing (\ref{staysigma}) and (\ref{compG}) shows that we can replace every $\gamma_n$ for which (\ref{stay}) does not hold with the function $\gamma$ defined in (\ref{competitorgeo}) and thus obtain a sequence of almost minimizing geodesics with the desired properties.

Next, we claim that there exists a minimizing geodesic for $d_F(p,q)$ with Euclidean length bounded from above by 
\begin{equation}
\label{EL4}
d_i(p) + d_i(q) + \Lip(z_i)\diam(\mathcal{R}).
\end{equation} 
In view of \Cref{L<k}, if the sequence $\{\gamma_n\}_{n \in \N}$ admits a subsequence $\{\gamma_{n_k}\}_{k \in \N}$ such that $L(\gamma_{n_k}) \le d_i(p) + d_i(q) + \Lip(z_i)\diam(\mathcal{R})$ then there is nothing to do. Notice also that if there exists a subsequence $\{\gamma_{n_j}\}_{j \in \N}$ with the property that
\[
\min\left\{d_i(\gamma_{n_j}(t) : t \in [-1,1])\right\} = d_i(\gamma_{n_j}(t_j)) = 0,
\]
then an application of the co-area formula yields
\begin{align*}
\int_{-1}^{t_j}F(\gamma_{n_j}(t))|\gamma_{n_j}'(t)|\,dt & \ge 2\sqrt{\alpha_i} \int_{-1}^{t_j}d_i(\gamma_{n_j}(t))\left|\frac{d}{dt}d(\gamma_{n_j}(t))\right|\,dt \\
& = 2\sqrt{\alpha_i}\int_0^{d_i(p)}y\mathcal{H}^0(\{t \in (-1,t_j) : d_i(\gamma_{n_j}(t) = y)\})\,d y \ge \sqrt{\alpha_i}d_i(p)^2,
\end{align*}
and with similar computations in the interval $(t_j,1)$ we arrive at
\[
\int_{-1}^1F(\gamma_{n_j}(t))|\gamma_{n_j}'(t)|\,dt \ge \sqrt{\alpha_i}\left(d_i(p)^2 + d_i(q)^2\right) = \int_{-1}^1F(\gamma(t))|\gamma'(t)|\,dt
\]
where $\gamma$ is the function defined in (\ref{competitorgeo}) and the last equality follows from (\ref{compG}). In turn, $\gamma$ is a minimizing geodesic and the claim would follow in this case. Thus, throughout the rest of the step we assume that 
\begin{equation}
\label{min>0}
\min\left\{d_i(\gamma_n(t)) : t \in [-1,1]\right\} = d_i(\gamma_n(t_n)) > 0
\end{equation}
for every $n \in \N$. In particular, eventually passing to a subsequence (which we do not relabel), we can assume without loss of generality that $\gamma_n \colon [-1,1] \to \R^M$ is a parametric representation of a continuous simple rectifiable curve. Hence, it can be parametrized by arclength, \emph{i.e.}, there exists a function $\varphi_n \colon [0,L(\gamma_n)] \to [-1,1]$ with the property that 
\begin{equation}
\label{geoarc}
v_n(s) \coloneqq \gamma_n(\varphi_n(s))
\end{equation}
is Lipschitz continuous, and in particular $|v_n'(s)| = 1$ for $\mathcal{L}^1$-a.e.\@ $s \in (0,L(\gamma_n))$, where 
\begin{equation}
\label{Lgn}
L(\gamma_n) = d_i(p) + d_i(q) + a_n, \qquad a_n \ge 0.
\end{equation}
Our aim is to show that if (\ref{min>0}) and (\ref{Lgn}) hold, then the function $\gamma$ in (\ref{competitorgeo}) is a minimizing geodesic for $d_F(p,q)$. We prove this claim in two substeps. 
\newline
\emph{Substep 1:} If $d_i(p) = d_i(\gamma_n(t_n))$, then $d_i(p) = d_i(q)$, and so
\begin{align}
\label{geosub1}
\int_{-1}^1F(\gamma_n(t))|\gamma_n'(t)|\,dt & \ge 2\sqrt{\alpha_i}d_i(p)\int_{-1}^1|\gamma_n'(t)|\,dt = 4\sqrt{\alpha_i}d_i(p)^2 + 2a_n\sqrt{\alpha_i}d_i(p),
\end{align}
where in the last equality we have used (\ref{Lgn}). In this case the claim readily follows by comparing (\ref{compG}) and (\ref{geosub1}).
\newline
\emph{Substep 2:} If $d_i(p) > d_i(\gamma_n(t_n))$, let $p' \in z_i(\mathcal{R})$ be given as above, let $p_n$ be the point on the line segment that joins $p$ to $p'$ with the property that $d_i(p_n) = d_i(\gamma_n(t_n))$, and define
\begin{equation}
\label{Qndef}
\mathcal{Q}_n \coloneqq B\left(p_n, \frac{d_i(p_n)}{2}\right) \cap \left\{z \in \R^M : \frac{d_i(p_n)}{2} \le d_i(z) \le d_i(p_n)\right\}.
\end{equation}
Let $w_n \colon [0, a_n + 2d_i(p_n)] \to \R^M$ be a parametrization by arclength of a simple closed arc of length $a_n + 2d_i(p_n)$ with the following properties: 
\begin{equation}
\label{di/2}
w_n(0) = w_n(a_n + 2d_i(p_n)) = p_n, \qquad w_n(t) \in \mathcal{Q}_n \text{ for all } t \in (0,a_n + 2d_i(p_n)).
\end{equation}
Let $q_n$ be the point on the line segment that joins $p$ and $p'$ with the property that $d_i(q_n) = d_i(q)$ and define $f_n \colon [0, d_i(p) - d_i(p_n)] \to \R^M$ and $g_n \colon [d_i(p) + d_i(p_n) + a_n, L(\gamma_n)] \to \R^M$ via
\begin{align*}
f_n(t) & \coloneqq \left(1 - \frac{t}{d_i(p) - d_i(p_n)}\right)p + \frac{tp_n}{d_i(p) - d_i(p_n)}, \\
g_n(t) & \coloneqq \left(1 + \frac{d_i(p) + d_i(p_n) + a_n - t}{d_i(q) - d_i(p_n)}\right)p_n + \frac{t - (d_i(p) + d_i(p_n) + a_n)}{d_i(q) - d_i(p_n)}q_n.
\end{align*}
Notice that if $d_i(q) = d_i(\gamma_n(t_n))$ then $q_n = p_n$ and the interval of definition of $g_n$ trivializes to a single point. Finally (see \Cref{compimg}), we set
\begin{equation}
\label{not-a-comp}
W_n(t) \coloneqq
\left\{
\begin{array}{ll}
f_n(t) & \text{ if } 0 \le t < d_i(p) - d_i(p_n),\\
& \\
w_n(t + d_i(p_n) - d_i(p)) & \text{ if } d_i(p) - d_i(p_n) \le t < d_i(p) + d_i(p_n) + a_n, \\
& \\
g_n(t) & \text{ if } d_i(p) + d_i(p_n) + a_n \le t \le L(\gamma_n).
\end{array}
\right.
\end{equation}

\begin{figure}[t]
\centering
\begin{tikzpicture}[blend group=normal, scale=1]
\draw [thick] plot [smooth, tension=0.5] coordinates {(0.0,0.5) (0.3,0.7) (1.1,5.5) (1.6,6.3)};
\node [left] at (1.3, 6) {$z_i(\mathcal{R})$};
\draw [fill] (2.5, 1.7) circle [radius=0.03];
\node [right] at (2.5, 1.7) {$p$};
\draw [fill] (2.5, 4.4) circle [radius=0.03];
\node [right] at (2.5,4.4) {$q$};
\draw plot [smooth, tension=0.9] coordinates {(2.5, 1.7) (2.1, 1.2) (2.3,1.9) (1.7, 2.1) (2.6, 2.9) (2.5, 3.4) (2.1, 3.8) (3.1, 4) (2.2,4.2) (2.5, 4.4)};
\draw [fill] (1.69, 2.105) circle [radius=0.03];
\node [below] at (1.49, 2.005) {$\gamma_n(t_n)$};

\draw [thick] plot [smooth, tension=0.5] coordinates {(4.0,0.5) (4.3,0.7) (5.1,5.5) (5.6,6.3)};
\node [left] at (5.3, 6) {$z_i(\mathcal{R})$};
\draw [fill] (6.5, 1.7) circle [radius=0.03];
\node [right] at (6.5, 1.7) {$p$};
\draw [fill] (6.5, 4.4) circle [radius=0.03];
\node [right] at (6.5,4.4) {$q$};
\draw [very thin, dashed] plot [smooth, tension=0.9] coordinates {(6.5, 1.7) (6.1, 1.2) (6.3,1.9) (5.7, 2.1) (6.6, 2.9) (6.5, 3.4) (6.1, 3.8) (7.1, 4) (6.2,4.2) (6.5, 4.4)};
\draw[thin, dotted, variable=\t,domain=0.14:0.76,samples=5] plot ({4.3 + 1.062+\t)},{6.428*\t});
\draw[thin, dotted, variable=\t,domain=-1.0:0.95,samples=5] plot ({4.5 + 1.062+\t)},{1.85 -0.1555*\t});
\draw[thin, variable=\t,domain=0.1:0.95,samples=5] plot ({4.5 + 1.062+\t)},{1.85 -0.1555*\t});
\draw [thin, red] plot [smooth, tension=0.2] coordinates {(5.648, 1.837) (5.565, 1.31) (5.59, 1.7) (5.37, 1.38) (5.548, 1.737) (5.22, 1.52) (5.52, 1.77) (5.13, 1.71) (5.5, 1.81) (5.12, 1.92) (5.5, 1.864) (5.18, 2.12) (5.525, 1.91) (5.32, 2.27) (5.56, 1.955) (5.52, 2.36) (5.62, 1.97) (5.732, 2.374) (5.648, 1.837)};
\draw [thin, blue] (5.648, 1.837) -- (6.088, 1.769); 
\draw[thin, dotted, variable=\t,domain=0.14:0.76,samples=5] plot ({4.75 + 1.062+\t)},{6.428*\t});
\draw [fill, blue] (6.088, 1.769) circle [radius=0.03];
\node [above, blue] at (6.383, 1.869) {$q_n$};
\draw [fill] (5.648, 1.837) circle [radius=0.03];
\node [below] at (5.82, 1.738) {$p_n$};

\draw [thick] plot [smooth, tension=0.5] coordinates {(8.0,0.5) (8.3,0.7) (9.1,5.5) (9.6,6.3)};
\node [left] at (9.3, 6) {$z_i(\mathcal{R})$};
\draw[thin, variable=\t,domain=-1.05:0.95,samples=5] plot ({8.5 + 1.062+\t)},{1.85 -0.1555*\t});
\draw[thin, variable=\t,domain=-0.624:0.95,samples=5] plot ({8.5 + 1.062+\t)},{4.55 -0.1555*\t});
\draw [fill] (10.5, 1.7) circle [radius=0.03];
\node [right] at (10.5, 1.7) {$p$};
\draw [fill] (10.5, 4.4) circle [radius=0.03];
\node [right] at (10.5,4.4) {$q$};
\draw [thin, dotted] plot [smooth, tension=0.9] coordinates {(10.5, 1.7) (10.1, 1.2) (10.3,1.9) (9.7, 2.1) (10.6, 2.9) (10.5, 3.4) (10.1, 3.8) (11.1, 4) (10.2,4.2) (10.5, 4.4)};

\node at (2,0) {$\gamma_n$};
\node at (6,0) {$W_n$};
\node at (10,0) {$\gamma$};

\end{tikzpicture}
\caption{From left to right, the figure depicts the curves parametrized by: an element of the sequence of almost minimizing geodesics, the function $W_n$ constructed to estimate the energy of $\gamma_n$, and the competitor $\gamma$.} \label{compimg}
\end{figure}
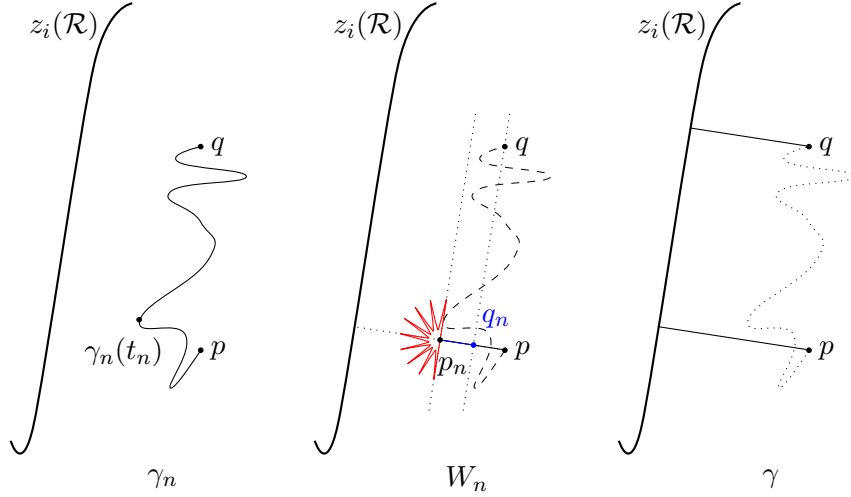

As one can readily check, when restricted to the interval $(0, d_i(p) + d_i(p_n) + a_n)$ the function $W_n$ gives a parametrization by arclength of a simple arc of length $d_i(p) + d_i(p_n) + a_n$. Moreover, if $d_i(p) + d_i(p_n) + a_n < L(\gamma_n)$, then $W_n$ restricted to the interval $(d_i(p) + d_i(p_n) + a_n, L(\gamma_n))$ gives a parametrization by arclength of a segment of length 
\[
L(\gamma_n) - d_i(p) - d_i(p_n) - a_n = d_i(q) - d_i(p_n).
\]
We claim that 
\begin{equation}
\label{diW<}
d_i(W_n(t)) \le d_i(v_n(t))
\end{equation}
for every $t \in [0,L(\gamma_n)]$, where $v_n$ is the reparametrization of $\gamma_n$ introduced in (\ref{geoarc}).
To prove claim, we argue by contradiction by assuming first that there exists $\overline{t} \in (0,d_i(p) - d_i(p_n))$ for which (\ref{diW<}) does not hold, so that by (\ref{not-a-comp})
\begin{equation}
\label{fntbar}
d_i(p) - d_i(v_n(\bar{t})) > d_i(p) - d_i(W_n(\bar{t})) = |p - f_n(\bar{t})| = \frac{\bar{t}|p - p_n|}{d_i(p) - d_i(p_n)} = \bar{t}.
\end{equation}
Recalling that $d_i$ is Lipschitz continuous with Lipschitz constant at most 1, and in view of (\ref{geoarc}), we see that
\begin{equation}
\label{1lip1lip}
d_i(p) - d_i(v_n(\bar{t})) \le |v_n(0) - v_n(\bar{t})| \le \bar{t}.
\end{equation}
Combining (\ref{fntbar}) and (\ref{1lip1lip}) we arrive at a contradiction. Notice that for $t \in (d_i(p) - d_i(p_n), d_i(p) + d_i(p_n) + a_n)$ inequality (\ref{diW<}) is satisfied in view of (\ref{Qndef}), (\ref{di/2}), and (\ref{not-a-comp}), while for all the remaining values of $t$ we can argue as above. Hence the claim is proved and therefore
\begin{equation}
\label{F<F}
F(W_n(t)) = 2\sqrt{\alpha_i}d_i(W_n(t)) \le 2\sqrt{\alpha_i}d_i(v_n(t)) = F(v_n(t))
\end{equation}
for every $t \in (0, L(\gamma_n))$.  Moreover, in view of (\ref{di/2}) and by means of a direct computation which uses the co-area formula, we see that
\begin{align*}
& \int_0^{d_i(p) - d_i(p_n)}F(W_n(t))\,dt = \sqrt{\alpha_i}\left(d_i(p)^2 - d_i(p_n)^2\right), \\
& \int_{d_i(p) - d_i(p_n)}^{d_i(p) + d_i(p_n) + a_n}F(W_n(t))\,dt \ge \sqrt{\alpha_i}d_i(p_n)(a_n + 2d_i(p_n)), \\
& \int_{d_i(p) + d_i(p_n) + a_n}^{L(\gamma_n)}F(W_n(t))\,dt = \sqrt{\alpha_i}\left(d_i(q)^2 - d_i(p_n)^2\right).
\end{align*}
In particular, combining the inequalities above with (\ref{compG}) and (\ref{F<F}) we obtain
\begin{align*}
\int_{-1}^1F(\gamma_n(t))|\gamma_n'(t)|\,dt = \int_0^{L(\gamma_n)}F(v_n(t))\,dt & \ge \int_0^{L(\gamma_n)}F(W_n(t))\,dt \\
& \ge \sqrt{\alpha_i}\left(d_i(p)^2 + d_i(q)^2\right) + \sqrt{\alpha_i}d_i(p_n)a_n \\
& \ge \int_{-1}^1F(\gamma(t))|\gamma'(t)|\,dt + \sqrt{\alpha_i}d_i(p_n)a_n \\
& \ge \int_{-1}^1F(\gamma(t))|\gamma'(t)|\,dt.
\end{align*}
Letting $n \to \infty$ in the previous inequality shows that $\gamma$ is a minimizing geodesic, thus proving our claim.
\newline
\textbf{Step 5:} Finally, we show the existence of a minimizing geodesic for any two distinct points $p,q \in \R^M$. Given $v \in \mathcal{A}(p,q)$ such that 
\[
\int_{-1}^1F(v(t))|v'(t)|\,dt \le d_F(p,q) + 1,
\]
let
\[
s_1 \coloneqq \min\left\{\inf\left\{ t \in (-1,1) : v(t) \in \bigcup_{i = 1}^k \mathcal{N}_{\sigma/2}(z_i(\mathcal{R}))\right\}, 1\right\}.
\]
If $s_1 < 1$, let $i_1$ be such that $v(s_1) \in \overline{\mathcal{N}_{\sigma/2}(z_{i_1}(\mathcal{R}))}$, and define
\[
t_1 \coloneqq \max\left\{\sup \left\{t \in (-1,1) : v(t) \in \mathcal{N}_{\sigma/2}(z_{i_1}(\mathcal{R}))\right\},s_1\right\}.
\]
Notice that if $s_1 < 1$, then $s_1$ and $t_1$ denote the first and last instance for which the support of the curve parametrized by $v$ can be found in the $\sigma/2$-neighborhood of $z_{i_1}(\mathcal{R})$, respectively. Similarly, for $j > 1$, if $t_{j - 1} < 1$, we define $s_j$, $i_j$, and $t_j$ inductively as follows:
\[
s_j \coloneqq \min\left\{\inf \left\{t \in (t_{j - 1}, 1) : v(t) \in \bigcup_{i \neq i_1, \dots, i_{j-1}}\mathcal{N}_{\sigma/2}(z_i(\mathcal{R}))\right\}, 1\right\},
\]
if $s_j < 1$ then the index $i_j$ is such that $v(s_j) \in \overline{\mathcal{N}_{\sigma/2}(z_{i_j}(\mathcal{R}))}$, and
\[
t_j \coloneqq \max\left\{\sup \left\{ t \in (t_{j - 1}, 1) : v(t) \in \mathcal{N}_{\sigma/2}(z_{i_j}(\mathcal{R}))\right\}, s_j\right\}.
\]
For every $j \in \{1, \dots, k\}$, let $v_j \colon [s_j, t_j] \to \R^M$ be the reparametrization of the minimizing geodesic which connects $v(s_j)$ to $v(t_j)$ found as in the previous steps, and let $V \colon [-1,1] \to \R^M$ be defined via 
\[
V(t) \coloneqq 
\left\{
\begin{array}{ll}
v_j(t) & \text{ if } t \in [s_j, t_j],\\
& \\
V(t) & \text{ otherwise}.
\end{array}
\right.
\] 
Then $V \in \mathcal{A}(p,q)$,
\[
\int_{-1}^1F(V(t))|V'(t)|\,dt \le \int_{-1}^1F(v(t))|v'(t)|\,dt,
\]
and furthermore, we see from (\ref{EL4}) that
\begin{align*}
L(V) & \le \sum_{j = 1}^k \int_{s_j}^{t_j}|v_j'(t)|\,dt + \int_{(-1,1) \setminus \cup_j(s_j,t_j)}|v'(t)|\,dt \\
& \le \sum_{j \in \mathcal{J}}\left(d_j(v(s_j)) + d_j(v(t_j) + \Lip(z_i)\diam(\mathcal{R})\right) + \frac{1}{\Sigma_\sigma(\mathcal{R})}\int_{-1}^1F(v(t))|v'(t)|\,dt \\
& \le  k\sigma + \diam(\mathcal{R})\sum_{i = 1}^k\Lip(z_i) + \frac{d_F(p,q) + 1}{\Sigma_\sigma(\mathcal{R})},
\end{align*}
where $\mathcal{J}$ denotes the set of indices for which $s_j \neq t_j$ and $\Sigma_\sigma(\mathcal{R})$ is defined as in (\ref{Sigma}).

Let $\{\gamma_n\}_{n \in \N} \subset \mathcal{A}(p,q)$ be a sequence of almost minimizing geodesics for $d_F(p,q)$. Then, for every $n$ sufficiently large we can find a function $V_n \in \mathcal{A}(p,q)$ such that
\[
L(V_n) \le k\sigma + \diam(\mathcal{R})\sum_{i = 1}^k\Lip(z_i) + \frac{d_F(p,q) + 1}{\Sigma_\sigma(\mathcal{R})}.
\] 
Thus, we are in a position to apply \Cref{L<k}. This concludes the the proof.
\end{proof}

\begin{remark}
\label{RemR}
In view of (\ref{H2}), condition (\ref{z_iRfar}) is satisfied if, for example, $\diam(\mathcal{R})$ is sufficiently small. Moreover, the convexity assumption on $\mathcal{R}$ can be easily relaxed by requiring that for any two points $x_1,x_2 \in \mathcal{R}$ there exists a path in $\mathcal{R}$ with finite Euclidean length from one to the other. One must then change the constant on the right-hand side of (\ref{LS}) accordingly.
\end{remark}

\begin{remark}
\label{UB}
Notice that the right-hand side of (\ref{LS}) depends continuously on $p$ and $q$. In particular, given $\lambda > 0$, set 
\begin{equation}
\label{S'}
\mathcal{W}(\lambda) \coloneqq \max\left\{2\sqrt{W(x,z)} : (x, z) \in \overline{\o} \times \overline{B(0,\lambda)}\right\}
\end{equation}
and observe that for every $p,q \in B(0,\lambda)$, if $\ell_{p,q}$ is defined as in (\ref{lpq}), we have
\[
d_F(p,q) \le \int_{-1}^1F(\ell_{p,q})|\ell_{p,q}'(t)|\,dt \le 2\lambda \mathcal{W}(\lambda).
\]
Consequently, if we let
\[
\Lambda(\lambda, \mathcal{R}) \coloneqq k\sigma + \diam(\o)\sum_{i = 1}^k\Lip(z_i) + \frac{2\lambda \mathcal{W}(\lambda) + 1}{\Sigma_\sigma(\mathcal{R})}, 
\]
then for every $p,q \in B(0,\lambda)$, \Cref{dFprop} yields the existence of a minimizing geodesic $\gamma \in \mathcal{A}(p,q)$ for $d_F(p,q)$ such that 
\begin{align*}
L(\gamma) & \le \Lambda(\lambda, \mathcal{R}), \\
\|\gamma\|_{L^{\infty}((-1,1);\R^M)} & \le \Lambda(\lambda, \mathcal{R}) + \lambda.
\end{align*}
Finally, observe that if $\mathcal{S} \subset \mathcal{R}$ then $\Sigma_{\sigma}(\mathcal{R}) \le \Sigma_{\sigma}(\mathcal{S})$ (see (\ref{Sigma})) and that in view of assumptions (\ref{H1})--(\ref{H3}) and (\ref{Hold4}),  
\[
\inf\left\{\Sigma_\sigma(\{x\}) : x \in \o \right\} > 0.
\]
Therefore, the following hold:
\begin{itemize}
\item[$(i)$] $\Lambda(\lambda, \mathcal{R}) \ge \Lambda(\lambda, \mathcal{S})$, 
\item[$(ii)$] $\sup\left\{\Lambda(\lambda, \{x\}) : x \in \o \right\} \eqqcolon \Lambda_{W}(\lambda) < \infty$.
\end{itemize}
\end{remark}

\begin{corollary}\label{dWisLipsc}
Under the assumptions of \Cref{dFprop}, the function $d_W$ introduced in \Cref{defdW} is Lipschitz continuous in $x$ and locally Lipschitz continuous with respect to the variables $p$ and $q$. In particular, $d_W$ is locally Lipschitz continuous, \emph{i.e.}, Lipschitz continuous on every compact subset of $\overline{\o} \times \R^M \times \R^M$. 
\end{corollary}

\begin{proof}
Fix $p, q \in \R^M$ and let $x_1,x_2$ be any two points in $\o$. Let $\lambda$ be such that $p,q \in B(0,\lambda)$ and notice that we can assume without loss of generality that $d_W(x_1, p, q) \ge d_W(x_2, p, q)$ since in the other case the result follows from similar computations. It follows from an application of \Cref{dFprop} with $\mathcal{R} = \{x_2\}$, together with \Cref{UB}, that there exists a minimizing geodesic for $d_W(x_2,p,q)$, namely $\gamma$, such that $L(\gamma) \le \Lambda_{W}(\lambda)$.

Since by assumption $W$ is locally Lipschitz continuous, behaves quadratically near the wells, and is bounded away from zero away from the wells (see (\ref{H1})--(\ref{H3}), (\ref{Hold4})), we have that $\sqrt{W}$ is also locally Lipschitz continuous. Thus there exists a constant $\Lip(\sqrt{W}; \lambda)$, which also depends on $\Lambda_W(\lambda)$, such that
\begin{equation}\label{eq:Lipsqrt}
\left|\sqrt{W(x_1,\gamma(t))} - \sqrt{W(x_2,\gamma(t))}\right| \le \Lip(\sqrt{W}; \lambda)|x_1-x_2|,
\end{equation}
for all $t \in (-1,1)$. Consequently, using \eqref{eq:Lipsqrt}, we can estimate
\begin{align*}
d_W(x_1, p, q) - d_W(x_2, p, q) & \le  2\int_{-1}^1 \,\left|\sqrt{W\left(x_1,\gamma(t)\right)} - \sqrt{W\left(x_2,\gamma(t)\right)}\right| |\gamma'(t)|\,dt \\
& \le 2\, \Lip(\sqrt{W}; \lambda) |x_1-x_2|\int_{-1}^1 |\gamma'(t)| \,dt \\
& \le 2\, \Lip(\sqrt{W}; \lambda) \Lambda_W(\lambda) |x_1-x_2|.
\end{align*}
On the other hand, for fixed $x \in \o$ and $p \in \R^M$ and for every $q,q' \in B(0,\lambda)$ with $q \neq q'$, if we let $\ell_{q,q'} \in \mathcal{A}(q,q')$ be defined as in (\ref{lpq}), then we have
\begin{align*}
\frac{|d_W(x,p,q) - d_W(x,p,q')|}{|q - q'|} \le \frac{d_W(x,q,q')}{|q - q'|} & \le \frac{1}{|q - q'|}\int_{-1}^12\sqrt{W(x,\ell_{q,q'}(t))}|\ell_{q,q'}'(t)|\,dt \\
& = \int_{-1}^1\sqrt{W(x,\ell_{q,q'}(t))}\,dt \le \mathcal{W}(\lambda),
\end{align*}
where $\mathcal{W}(\lambda)$ is given as in (\ref{S'}). The rest of the proof follows from similar considerations; we omit the details.
\end{proof}


\section{Proof of \Cref{thm:main}}

\subsection{Compactness}
\label{sec:cmpt}
In this section we show that any sequence with bounded energy is precompact in $L^1(\o;\R^M)$.
\begin{proposition}
\label{prop:cpt}
Let $W$ be given as in \emph{(\ref{H1})}--\emph{(\ref{H4})} and let $\{u_n\}_{n \in \N} \subset H^1(\o;\R^M)$ be such that
\begin{equation}
\label{eq:boundenergy}
\sup \left\{ \f_n(u_n) : n \in \N \right\} \eqqcolon C < \infty.
\end{equation}
Then, eventually extracting a subsequence (not relabeled), we have that $u_n \to u$ in $L^1(\o;\R^M)$, where $u \in BV(\o;z_1,\dots,z_k)$ is such that $\f_0(u) <\infty$.
\end{proposition}

\begin{proof} We divide the proof into several steps.
\newline
\textbf{Step 1:} In this first step we prove that the sequence $\{u_n\}_{n \in \N}$ is bounded in $L^1(\o;\R^M)$ and equi-integrable. The proof is standard, but we report it here for the reader's convenience. The proof we present is adapted from \cite[Proposition 4.1]{Bal} (see also \cite[Theorem 1.6 and Theorem 2.4]{LeoniCNA}). To this end, notice that in view of (\ref{H4}) and (\ref{eq:boundenergy}) we have that
\begin{equation}
\label{eq:equiint}
S \int_{\{ |u_n| > R \}} |u_n(x)|\,dx \le \int_{\o} W(x, u_n(x))\,dx \le C \e_n.
\end{equation}
Consequently, if we let $E \subset \o$ be a measurable set, then
\begin{equation}
\label{equi1}
\int_E |u_n(x)| \,dx = \int_{E \cap \{ |u_n| \le R \}} |u_n(x)| \,dx + \int_{E \cap \{|u_n| > R\}} |u_n(x)| \,dx \le R \mathcal{L}^N(E) + \frac{C}{S}\e_n,
\end{equation}
where in the last step we used \eqref{eq:equiint}. In particular, by taking $E = \o$ we obtain that the sequence $\{u_n\}_{n \in \N}$ is bounded in $L^1(\o;\R^M)$. Moreveor, for every fixed $s > 0$, setting 
\[
\overline{n} \coloneqq \inf\left\{n \in \N : \frac{C}{S}\e_m \le \frac{s}{2}\text{ for all } m \ge n\right\} 
\]
and
\[
t \coloneqq \frac{1}{R}\left(s - \frac{C}{S}\e_{\overline{n}} \right),
\]
as a consequence of (\ref{equi1}) we obtain that for every $n \ge \overline{n}$ and every measurable set $E \subset \o$
\[
\int_E |u_n(x)| \,dx \le s
\]
provided that $\mathcal{L}^N(E) \le t$. This shows that the sequence $\{u_n\}_{n \in \N}$ is equi-integrable.
\newline
\textbf{Step 2:} For $R$ as in (\ref{H4}) and using the notation introduced in \Cref{UB}, set
\begin{equation}
\label{R'=}
R' \coloneqq R + \Lambda_{W}(R)
\end{equation}
and let $\varphi \colon [0,\infty) \to [0,1]$ be a smooth cut-off function such that $\varphi(\rho) = 1$ for $\rho \le R'$ and $\varphi(\rho) = 0$ for $\rho \ge 2R'$. Let
\[
W_1(x,z) \coloneqq \varphi(|z|)W(x,z) + (1 - \varphi(|z|))S|z|.
\]
Then $W_1 \colon \o \times \R^M \to [0,\infty)$ satisfies (\ref{H1})--(\ref{H4}) and moreover 
\begin{equation}
\label{W1}
W_1(x,z) \le W(x,z)
\end{equation} 
for every $x \in \o$ and $z \in \R^M$. For every $n \in \N$, define the function $v_n \colon \o \to \R^M$ via
\begin{equation}
\label{truncate}
v_n(x) \coloneqq 
\left\{
\begin{array}{ll}
\displaystyle u_n(x) & \text{ if } u_n(x) \in B(0,2R'),\\
& \\
\displaystyle 2R'\frac{u_n(x)}{|u_n(x)|} & \text{ otherwise}.
\end{array}
\right.
\end{equation}
Notice that $v_n \in H^1(\o;\R^M) \cap L^{\infty}(\o;\R^M)$ with $\|v_n\|_{L^{\infty}(\o;\R^M)} \le 2R'$, and that $|\nabla v_n(x)| \le |\nabla u_n(x)|$ for $\mathcal{L}^N$-a.e.\@ $x \in \o$. We claim that
\begin{equation}
\label{W1un}
W_1(x,v_n(x)) \le W_1(x,u_n(x))
\end{equation}
for $\mathcal{L}^N$-a.e.\@ $x \in \o$. Indeed, equality holds for $\mathcal{L}^N$-a.e.\@ $x$ such that $u_n(x) \in B(0,2R')$, while if this is not the case then 
\[
W_1(x,v_n(x)) = S|v_n(x)| \le S|u_n(x)| = W_1(x,u_n(x)).
\]
Thus, by (\ref{eq:boundenergy}), (\ref{W1}), and (\ref{W1un}) we see that 
\begin{align}
\label{W1C}
\int_\o \left[\frac{1}{\e_n}W_1\left(x,v_n(x)\right)+\e_n|\nabla v_n(x)|^2\right]\,dx & \le \int_\o \left[\frac{1}{\e_n}W_1\left(x,u_n(x)\right)+\e_n|\nabla u_n(x)|^2\right]\,dx \notag \\
& \le \int_\o \left[\frac{1}{\e_n}W\left(x,u_n(x)\right)+\e_n|\nabla u_n(x)|^2\right]\,dx \le C.
\end{align}
We conclude this step by remarking that $R'$ (see (\ref{R'=})) is chosen in such a way that 
\begin{equation}
\label{dW=dW1}
d_{W_1}(x,z_i(x), z_j(x)) = d_W(x,z_i(x), z_j(x))
\end{equation}
for every $x \in \o$ and every $i,j \in \{1, \dots, k\}$.
\newline
\textbf{Step 3:} For $i \in \{1, \dots, k\}$ and $n \in \N$, let $f_n^i \colon \o \to \R$ be the function defined via 
\[
f^{i}_n(x) \coloneqq d_{W_1}(x, z_i(x), v_n(x)).
\]
The purpose of this step is to show that, up to the extraction of a subsequence (which we do not relabel), $f^{i}_n \to f^{i}$ in $L^1(\o)$ as $n \to \infty$, for some $f^i \in BV(\o)$. To prove the claim, it is enough to show that 
\begin{equation}
\label{wtsBaldo}
\int_{\o} |\nabla f_n^i(x)|\,dx \le \Lip(d_{W_1}; 3R')\left(1 + \Lip(z_i)\right)\mathcal{L}^N(\o) + \int_{\o}2 \sqrt{W_1(x,v_n(x))}|\nabla v_n(x)|\,dx, 
\end{equation}
where by $\Lip(d_{W_1}; \lambda)$ we denote the Lipschitz constant of $d_{W_1}$ on $\o \times B(0,\lambda) \times B(0,\lambda)$ (see \Cref{dWisLipsc}). Indeed, (\ref{W1C}) implies that
\begin{equation}
\label{youngvn}
\int_{\o} 2\sqrt{W_1(x,v_n(x))}|\nabla v_n(x)|\,dx \le \int_\o \left[\frac{1}{\e_n}W_1\left(x,v_n(x)\right)+\e_n|\nabla v_n(x)|^2\right]\,dx \le C,
\end{equation}
and so it follows from (\ref{wtsBaldo}) and (\ref{youngvn}) that
\begin{equation}
\label{fniGrad}
\int_{\o}|\nabla f_n^i(x)|\,dx \le \Lip(d_{W_1}; 3R')\left(1 + \Lip(z_i)\right)\mathcal{L}^N(\o) + C.
\end{equation}
Since, as one can readily check, the sequence $\{f_n^i\}_{n \in \N}$ is bounded in $L^1(\o)$, (\ref{fniGrad}) yields that it is also bounded in $W^{1,1}(\o)$. In turn, the claim follows by the Rellich--Kondrachov compactness theorem (see, for example, \cite[Theorem 14.36]{L}). The rest of the step is devoted to the proof of (\ref{wtsBaldo}), which is adapted from \cite[Proposition 2.1]{Bal}.

By the Meyer--Serrin approximation theorem, for every $n \in \N$ there exists a sequence $\{v_{n,k}\}_{k \in \N}$ of functions in $H^1(\o;\R^M) \cap C^1(\overline{\o};\R^M)$ such that 
\begin{equation}
\label{vnkconv}
\arraycolsep=1.4pt\def\arraystretch{1.6}
\begin{array}{rll}
v_{n,k} \to & v_n & \text{ in } H^1(\o,\R^M), \\
v_{n,k} \to & v_n & \text{ pointwise a.e. in } \o, \\
\nabla v_{n,k} \to & \nabla v_n & \text{ pointwise a.e. in } \o, \\
|v_{n,k}(x)| \le & 3R' & \text{ in } \o.
\end{array}
\end{equation}
Moreover, eventually passing to a subsequence (which we do not relabel), for every $n \in \N$ we can find a function $h_n \in L^1(\o)$ such that for $\mathcal{L}^N$-a.e.\@ $x \in \o$
\begin{equation}
\label{Lebnk}
|v_{n,k}(x)|^2 + |\nabla v_{n,k}(x)|^2 \le h_n(x).
\end{equation}
Let $f_{n,k}^i \colon \o \to \R$ be defined as 
\[
f_{n,k}^i(x) \coloneqq d_{W_1}(x, z_i(x), v_{n,k}(x)).
\]
Then $f_{n,k}^i$ is Lipschitz continuous in $\o$ and therefore differentiable almost everywhere. Observe that for $x, y \in \o$
\begin{multline*}
|f_{n,k}^i(y) - f_{n,k}^i(x)| \le |d_{W_1}(y, z_i(y), v_{n,k}(y)) - d_{W_1}(x, z_i(y), v_{n,k}(y))| + |d_{W_1}(x, z_i(y), v_{n,k}(y)) \\
- d_{W_1}(x, z_i(x), v_{n,k}(y))| + |d_{W_1}(x, z_i(x), v_{n,k}(y)) - d_{W_1}(x, z_i(x), v_{n,k}(x))|,
\end{multline*}
so that
\begin{equation}
\label{fnyx}
|f_{n,k}^i(y) - f_{n,k}^i(x)| \le \Lip(d_{W_1}; 3R')\left(1 + \Lip(z_i)\right)|x - y| + d_{W_1}(x, v_{n,k}(x), v_{n,k}(y)).
\end{equation}
Fix $\tau > 0$, and set $\o_{\tau} \coloneqq \{x \in \o : \dist(x,\partial \o) > \tau\}$. Then, if $x \in \o_{\tau}$, $h \in \R \setminus \{0\}$ is such that $|h| \le \tau$, and $\nu \in \mathbb{S}^{N - 1}$, by setting $y = x + h \nu$ in (\ref{fnyx}), we obtain
\begin{equation}
\label{fnyx2}
\frac{|f_n^i(x + h \nu) - f_n^i(x)|}{|h|} \le C_1 + \frac{d_{W_1}(x, v_n(x), v_n(x + h \nu))}{|h|},
\end{equation}
where the constant $C_1$ is defined as
\[
C_1 \coloneqq \Lip(d_{W_1}; 3R')\left(1 + \Lip(z_i)\right).
\]
Let $\ell_{n,k}^h \in \mathcal{A}(v_{n,k}(x),v_{n,k}(x + h \nu))$ be a parametrization of the line segment that joins $v_{n,k}(x)$ and $v_{n,k}(x + h \nu)$ and notice that 
\begin{align}
\label{baldo2.1}
\frac{d_{W_1}(x,v_{n,k}(x),v_{n,k}(x + h \nu))}{|h|} & \le \frac{1}{|h|}\int_{-1}^1 2\sqrt{W_1(x,\ell_{n,k}^h(t))}|(\ell_{n,k}^h)'(t)| \,dt \notag \\
& = \frac{|v_{n,k}(x + h \nu) - v_{n,k}(x)|}{|h|} \int_{-1}^1 \sqrt{W_1(x,\ell_{n,k}^h(t))}\,dt \notag \\
& \le \frac{1}{|h|}\int_0^{|h|} |\nabla v_{n,k}(x + s \nu)|\,ds \int_{-1}^1 \sqrt{W_1(x,\ell_{n,k}^h(t))}\,dt.
\end{align}
If $f_{n,k}^i$ is differentiable at $x \in \o_{\tau}$, by (\ref{fnyx2}) and (\ref{baldo2.1}), and in view of the continuity of $W_1$, it follows that by letting $h \to 0$ we get
\begin{equation}
\label{gnu}
|\nabla f_{n,k}^i(x) \cdot \nu| \le C_1 + 2 \sqrt{W_1(x,v_{n,k}(x))}|\nabla v_{n,k}(x)|.
\end{equation}
Taking the supremum over all $\nu \in \mathbb{S}^{N - 1}$ in (\ref{gnu}), we obtain that for $\mathcal{L}^N$-a.e.\@ $x \in \o$
\begin{equation}
\label{gradfni}
|\nabla f_{n,k}^i(x)| \le C_1 + 2 \sqrt{W_1(x,v_{n,k}(x))} |\nabla v_{n,k}(x)|.
\end{equation}
In turn, by (\ref{vnkconv}), (\ref{Lebnk}), (\ref{gradfni}), and Lebesgue's dominated convergence theorem we see that 
\begin{align}
\label{BaldoC1}
\int_{\o_{\tau}}|\nabla f_{n,k}^i(x)|\,dx & \le C_1 \mathcal{L}^N(\o_{\tau}) + \int_{\o_{\tau}}2 \sqrt{W_1(x,v_{n,k}(x))} |\nabla v_{n,k}(x)|\,dx \notag \\
& \le C_1 \mathcal{L}^N(\o) + \int_{\o}2 \sqrt{W_1(x,v_{n,k}(x))} |\nabla v_{n,k}(x)|\,dx \notag \\
& \to C_1 \mathcal{L}^N(\o) + \int_{\o}2 \sqrt{W_1(x,v_n(x))} |\nabla v_n(x)|\,dx
\end{align}
as $k \to \infty$. Next, using the notation introduced in \Cref{UB} (see (\ref{S'})), we observe that by (\ref{Lebnk}) and (\ref{gradfni}) we have that 
\begin{align*}
\int_{\o_{\tau}}|\nabla f_{n,k}^i(x)|^2\,dx & \le 2C_1^2\mathcal{L}^N(\o) + 8\mathcal{W}(3R') \int_{\o}|\nabla v_{n,k}(x)|^2\,dx \\
& \le 2C_1^2\mathcal{L}^N(\o) + 8\mathcal{W}(3R') \int_{\o}h_n(x)\,dx.
\end{align*}
By the monotone convergence theorem, letting $\tau \to 0$ in the previous inequality, we conclude that $\{f_{n,k}^i\}_{k \in \N}$ is bounded in $H^1(\o;\R^M)$. Therefore, eventually extracting a subsequence (which we do not relabel), there exists $g_n^i \in H^1(\o;\R^M)$ such that $f_{n,k}^i \rightharpoonup g_n^i$ in $H^1(\o;\R^M)$ and $f_{n,k}^i \to g_n^i$ in $L^2(\o;\R^M)$ as $k \to 0$. Since $f_{n,k}^i \to f_n^i$ in $L^2(\o;\R^M)$, we obtain that $f_n^i$ must coincide with $g_n^i$ almost everywhere in $\o$, and therefore
\begin{equation}
\label{weaknk}
\int_{\o_{\tau}}|\nabla f_{n,k}^i(x)|\,dx \to \int_{\o_{\tau}}|\nabla f_n^i(x)|\,dx.
\end{equation}
Combining (\ref{BaldoC1}) and (\ref{weaknk}), we arrive at (\ref{wtsBaldo}) with an application of the monotone convergence theorem, letting $\tau \to 0$.
\newline
\textbf{Step 4:} Without loss of generality, we can assume that $f_n^i \to f^i$ pointwise $\mathcal{L}^N$-a.e.\@ in $\o$ for every $i$. Next, let 
\[
\o^i \coloneqq \{x \in \o : f^i(x) = 0\}
\] 
and set 
\begin{equation}
\label{udef}
u(x) \coloneqq \sum_{i = 1}^kz_i(x)\ca_{\o^i}(x).
\end{equation}
We claim that eventually extracting a subsequence, $v_n \to u$ in $L^1(\o;\R^M)$. Notice that by Vitali's convergence theorem, together with the results of Step 1, it is enough to show the existence of a subsequence which converges pointwise almost everywhere to $u$. To this end, we observe that (\ref{W1C}) implies that
\[
\mathcal{L}^N\left(\left\{x \in \o : \limsup_{n \to \infty}W_1(x,v_n(x)) >0 \right\}\right) = 0.
\]
Thus, for $\mathcal{L}^N$-a.e.\@ $x \in \o$ we have that cluster points of the sequence $\{v_n(x)\}_{n \in \N}$ belong to the set $\{z_1(x), \dots, z_k(x)\}$. For $i \neq j$, let $B_{ij}$ be the subset of $\o^i$ consisting of all points $x$ for which there exists a subsequence $\{v_{n_m}(x)\}_{m \in \N}$ such that $v_{n_m}(x) \to z_j(x)$. Arguing by contradiction, assume that $\mathcal{L}^N(B_{ij}) > 0$ for some $j$. Then we have
\begin{align*}
0 < \int_{B_{ij}}d_{W_1}(x, z_i(x), z_j(x))\,dx & = \lim_{m \to \infty} \int_{B_{ij}}d_{W_1}(x,z_i(x),v_{n_m}(x))\,dx \\
& = \lim_{m \to \infty}\int_{B_{ij}}f_{n_m}^i(x)\,dx \\
& = \int_{B_{ij}}f^i(x)\,dx = 0.
\end{align*}
Thus we have arrived at a contradiction. In particular, since $v_n \to u$ in $L^1(\o;\R^M)$, we deduce that 
\begin{equation}
\label{fi=}
f^i(x) = d_{W_1}(x,z_i(x),u(x)) = d_W(x,z_i(x),u(x))
\end{equation}
for $\mathcal{L}^N$-a.e.\@ $x \in \o$, where in the last equality we have used (\ref{dW=dW1}) and (\ref{udef}). Finally, notice that by (\ref{eq:equiint}) and (\ref{truncate}) we have that
\[
\int_{\o}|u_n(x) - v_n(x)|\,d x = \int_{\{|u_n| > 2R\}}|u_n(x)|\,d x \le C\e_n,
\]
and therefore, eventually extracting a subsequence, $u_n \to u$ in $L^1(\o;\R^M)$ as it was claimed.
\newline
\textbf{Step 5:} This step is concerned with the proof of additional regularity properties of the function $u$, defined in (\ref{udef}). Let us remark that in view of (\ref{fi=}), throughout the rest proof we can return to working with the potential $W$ instead of $W_1$. We begin by showing the following characterization of the jump set of $u$ (see \Cref{def:jump}):
\begin{equation}
\label{JUij}
J_u = \bigcup_{i = 1}^{k - 1}\bigcup_{j = i + 1}^k U_{ij}, 
\end{equation}
where the sets $U_{ij}$ are defined via
\begin{equation}
\label{Uij}
U_{ij} \coloneqq \left\{ x \in J_{f^i} : (f^i(x)^+, f^i(x)^-, \nu_{f^i}(x)) = (0, d_W(x,z_i(x), z_j(x)), \nu_{f^i}(x))\right\}.
\end{equation}
Notice that 
\[
J_u = \bigcup_{i = 1}^k\bigcup_{j = i + 1}^k V_{ij}, 
\]
where
\begin{equation}
\label{Vij}
V_{ij} \coloneqq \left\{ x \in J_u : (u^+(x), u^-(x), \nu_u(x)) = (z_i(x), z_j(x), \nu_u(x)) \right\},
\end{equation}
so that to prove (\ref{JUij}) it is enough to show that $U_{ij} = V_{ij}$. To this end, fix $x \in V_{ij}$ and observe that 
\begin{align}
\label{Vij+}
d_W(y, z_i(y), u(y)) & = d_W(y, z_i(y), u(y)) - d_W(x, z_i(x), u(y)) + d_W(x, z_i(x), u(y)) \notag \\
& \le \Lip(d_W; 3R')\left((1 + \Lip(z_i))|x - y| + |u(y) - z_i(x)|\right)
\end{align}
holds for $\mathcal{L}^N$-a.e.\@ $y \in \o$.
Consequently, from (\ref{fi=}), (\ref{Vij}), and (\ref{Vij+}) we see that 
\begin{align*}
0 & \le \frac{1}{\rho^N}\int_{B^+(x,\nu_u(x), \rho)} f^i(y)\,dy \\
& \le \Lip(d_W; 2R')\left(\frac{(1 + \Lip(z_i))\rho}{
2} + \frac{1}{\rho^N}\int_{B^+(x, \nu_u(x), \rho)} |u(y) - z_i(x)|\,dy\right) \to 0
\end{align*}
as $\rho \to 0^+$. Similarly, one can show that 
\[
\frac{1}{\rho^N}\int_{B^-(x, \nu_u(x), \rho)} |f^i(y) - d_W(x,z_i(x),z_j(x))|\,dy \to 0,
\]
and thus we conclude that $V_{ij} \subset U_{ij}$. To prove the reverse inequality, we begin by noticing that in view of (\ref{H2}) there exists a constant $\omega$ such that 
\begin{equation}
\label{littleomega}
\inf\left\{d_W(x, z_i(x), z_j(x)) : x \in \o \text{ and } i \neq j\right\} \ge \omega.
\end{equation}
In turn, if $x \in U_{ij}$ (see (\ref{Uij})) and $\e > 0$ is given, there exists $\overline{\rho} > 0$ such that if $\rho \le \overline{\rho}$ then 
\[
\rho_N\e \ge \int_{B^+(x, \nu_{f^i}(x), \rho)} f^i(y)\,dy \ge \omega \sum_{j \neq i}\mathcal{L}^N\left(B^+(x, \nu_{f^i}(x), \rho) \cap \o^j\right).
\]
Consequently, recalling that $\|z_i\|_{L^{\infty}(\o;\R^M)} \le R$ by (\ref{H4}), notice that for every such $\rho$ we have
\begin{align*}
\int_{B^+(x, \nu_{f^i}(x), \rho)}|u(y) - z_i(x)|\,dy & = \sum_{j = 1}^k\int_{B^+(x, \nu_{f^i}(x), \rho) \cap \o^j}|z_j(y) - z_i(x)|\,dy \\
& \le \Lip(z_i)\rho^{N + 1} + 2R \sum_{j \neq i} \mathcal{L}^N\left(B^+(x, \nu_{f^i}(x), \rho) \cap \o^j\right) \\
& \le \Lip(z_i)\rho^{N + 1} +  \frac{2R \rho^N \e}{\omega},
\end{align*}
and therefore  
\[
\limsup_{\rho \to 0^+} \frac{1}{\rho^N}\int_{B^+(x, \nu_{f^i}(x), \rho)}|u(y) - z_i(x)|\,dy \le \frac{2R \e}{\omega}.
\]
Finally, letting $\e \to 0^+$ yields the desired result. Notice that with similar computations one can show that 
\[
\lim_{\rho \to 0^+} \frac{1}{\rho^N}\int_{B^-(x, \nu_{f^i}(x), \rho)}|u(y) - z_j(x)|\,dy = 0.
\]
Hence $x \in V_{ij}$, and therefore we have shown that also $U_{ij} \subset V_{ij}$. The characterization of the jump set in (\ref{JUij}) readily follows.

Next, observe that for $U_{ij}$ defined as in (\ref{Uij}), if $\omega$ is the constant given in (\ref{littleomega}) then 
\begin{align}
\label{measUij}
\mathcal{H}^{N - 1}(U_{ij}) = \int_{U_{ij}}\,d\mathcal{H}^{N - 1}(x) & = \int_{U_{ij}}\frac{d_W(x, z_i(x), z_j(x))}{d_W(x, z_i(x), z_j(x))}\,d\mathcal{H}^{N - 1}(x) \notag \\
& \le \frac{1}{\omega} \int_{U_{ij}}d_W(x, z_i(x), z_j(x))\,d\mathcal{H}^{N - 1}(x) \notag \\
& = \frac{1}{\omega} \int_{U_{ij}}\left|f^i(x)^+ - f^i(x)^-\right|\,d\mathcal{H}^{N - 1}(x) < \infty,
\end{align}
where in the last equality we have used the fact that $f^i \in BV(\o)$. Combining (\ref{udef}), (\ref{JUij}), and (\ref{measUij}) we see that $u \in BV(\o; z_1, \dots, z_k)$. Finally, notice that
\begin{align*}
\f_0(u) = \int_{J_u}d_{W}(x, u^+(x), u^-(x)))\,d\mathcal{H}^{N - 1}(x) & = \sum_{i = 1}^{k - 1}\sum_{j = i + 1}^k\int_{U_{ij}} d_{W}(x,z_i(x), z_j(x))\,d\mathcal{H}^{N - 1}(x) \\
& \le \sum_{i = 1}^{k - 1}\sum_{j = i + 1}^k \int_{U_{ij}}\left|f^i(x)^+ - f^i(x)^-\right|\,d\mathcal{H}^{N - 1}(x) \\
& < \infty.
\end{align*}
This concludes the proof.
\end{proof}


\subsection{Liminf inequality}\label{sec:liminf}
The goal of this section is to prove the following result.

\begin{proposition}\label{prop:liminf}
Let $W$ be given as in \emph{(\ref{H1})}--\emph{(\ref{H4})}. For $u \in L^1(\o;\R^M)$, let $\{u_n\}_{n \in \N}$ be a sequence of functions in $H^1(\o;\R^M)$ such that $u_n \to u$ in $L^1(\o;\R^M)$. Then
\[
\f_0(u) \le \liminf_{n \to \infty} \f_n(u_n).
\]
\end{proposition}
\begin{proof}
We divide the proof into several steps.
\newline
\textbf{Step 1:} Eventually extracting a subsequence (which we do not relabel), we can assume without loss of generality that
\begin{equation}
\label{liminf=lim}
\liminf_{n\to\infty}\f_n(u_n) = \lim_{n\to\infty}\f_n(u_n) < \infty.
\end{equation}
Consequently we are in a position to apply \Cref{prop:cpt} for $n$ large enough and conclude that $u \in BV(\o;\{z_1,\dots,z_k\})$ (see \Cref{f0}).
In order to prove the liminf inequality we will use the blow-up method of Fonseca and M\"uller (see \cite{FonMul_BV}). Let us consider the finite positive Radon measures $\mu_n\in\mathcal{M}^+(\o)$ given by
\[
\mu_n\coloneqq\left(\, \frac{1}{\e_n}W(\cdot,u_n(\cdot)) + \e_n|\nabla u_n(\cdot)|^2  \,\right) \mathcal{L}^N\restr\o.
\]
In view of (\ref{liminf=lim})
we can further assume that $\sup \left\{ \mu_n(\o) : n \in \N \right\} < \infty$, and therefore, up to the extraction of a subsequence (which again we do not relabel), there exists a measure $\mu \in \mathcal{M}^+(\o)$ such that
$\mu_n\stackrel{*}{\rightharpoonup}\mu$ (see \Cref{thm:wscmpt}). Let $\lambda \coloneqq \mathcal{H}^{N - 1}\restr J_u$ and let $x_0 \in J_u$ be such that 
\begin{equation}
\label{goodpoint}
\frac{d\mu}{d\lambda}(x_0) < \infty \quad \text{ and } \quad \frac{d\mathcal{H}^{N - 1}}{d\lambda}(x_0) = 1.
\end{equation}
Recall that \eqref{goodpoint} holds for $\hno$-a.e.\@ $x_0\in J_u$.
Let $Q_\nu \subset \R^N$ be a unit cube centered at the origin with two faces orthogonal to $\nu \coloneqq \nu_u(x_0)$, where the direction $\nu_u(x_0)$ is given as in \Cref{def:jump}, and for $\rho > 0$ write
$Q(x_0,\nu, \rho) \coloneqq x_0 + \rho Q_\nu$. Let $b_1, \dots, b_{N - 1} \in \R^N$ be such that $\{b_1, \dots, b_{N - 1}, \nu\}$ is an orthonormal bases for $\R^N$. Then, for every point $x \in \R^N$ there are constants $y_1, \dots, y_{N - 1}, t \in \R$ such that 
\[
x = \sum_{i = 1}^{N - 1}y_ib_i + t\nu.
\] 
In the following we identify $x$ with $(x', t)$, where $x' \in \R^{N - 1}$ denotes the vector $(y_1, \dots, y_{N - 1})$. Let
\begin{equation}\label{eq:qprime}
Q'(x_0,\nu,\rho) \coloneqq \left\{ x' = (y_1, \dots, y_{N - 1}) : \sum_{i = 1}^{N - 1}y_ib_i + t \nu \in Q(x_0, \nu, \rho) \text{ for all } t \in \left(-\frac{\rho}{2}, \frac{\rho}{2}\right) \right\},
\end{equation}
and notice that with this notation at hand we can write
\[
Q(x_0, \nu, \rho) = \left\{x = (x', t) : x' \in Q'(x_0,\nu,\rho) \text{ and } t \in (-\rho/2,\rho/2) \right \}.
\]
Moreover, we define
\begin{align*}
Q^+(x_0,\nu, \rho) & \coloneqq \left\{x = (x', t) : x' \in Q'(x_0, \nu, \rho) \text{ and } t \in \left(0, \frac{\rho}{2}\right)\right\}, \\
Q^-(x_0,\nu, \rho) & \coloneqq \left\{x = (x', t) : x' \in Q'(x_0, \nu, \rho) \text{ and } t \in \left(- \frac{\rho}{2}, 0\right)\right\}.
\end{align*}
Since by assumption $x_0 \in J_u$, there are two indices $1 \le i_1 < i_2 \le k$ such that for every $\e > 0$ there exists $\overline{\rho} = \overline{\rho}(\e) > 0$ with the property that for every $\rho \le \overline{\rho}$
\begin{equation}
\label{x0jump}
\frac{1}{\rho^N}\int_{Q^+(x_0,\nu, \rho)}|u(x) - z_{i_1}(x_0)|\,dx \le \e \quad \text{ and } \quad \frac{1}{\rho^N}\int_{Q^-(x_0,\nu, \rho)}|u(x) - z_{i_2}(x_0)|\,dx \le \e.
\end{equation}
In particular, if we let $\o^j \coloneqq \{x : u(x) = z_j(x)\}$, we see that 
\begin{align}
\label{slicing1}
\e & \ge \frac{1}{\rho^N}\int_{Q^+(x_0,\nu, \rho)}|u(x) - z_{i_1}(x_0)|\,dx \notag \\
& \ge \sum_{j \neq i_1}\frac{1}{\rho^N}\int_{Q^+(x_0,\nu, \rho) \cap \o^j}|z_j(x) - z_{i_1}(x_0)|\,dx \notag \\
& \ge \sum_{j \neq i_1}\frac{1}{\rho^N}\int_{Q^+(x_0,\nu, \rho) \cap \o^j}(|z_j(x_0) - z_{i_1}(x_0)| - |z_j(x) - z_j(x_0)|)\,dx \notag \\
& \ge \sum_{j \neq i_1}\left(\frac{\delta}{\rho^N}\int_{Q^+(x_0,\nu, \rho)} \ca_{\o^j}(x)\,dx - \Lip(z_j)c(N)\rho\right).
\end{align}
Notice that we can choose $\overline{\rho}$ in such a way that it also satisfies
\begin{equation}
\label{slicing2}
c(N)\overline{\rho}\sum_{j = 1}^k\Lip(z_j) \le \e.
\end{equation}
Combining (\ref{slicing1}) and (\ref{slicing2}) we obtain that for every $j \neq i_1$
\begin{equation}
\label{slicing3}
\frac{2\e}{\delta} \ge \frac{1}{\rho^N}\int_{Q^+(x_0,\nu, \rho)} \ca_{\o^j}(x)\,dx = \frac{1}{\rho}\int_0^{\rho/2}G_j^+(\rho,t)\,dt,
\end{equation}
where 
\[
G_j^+(\rho,t) \coloneqq \frac{1}{\rho^{N - 1}}\int_{Q'(x_0,\nu, \rho)}\ca_{\o^j}(x',t)\,dx'.
\]
In view of (\ref{slicing3}), there exists a measurable set $E^+_j(\rho) \subset (0,\rho/2)$ with 
\begin{equation}
\label{E+j}
\mathcal{L}^1(E^+_j(\rho)) = \left(\frac{1}{2}-\frac{1}{4(k - 1)}\right)\rho
\end{equation}
such that 
\begin{equation}
\label{G+}
G_j^+(\rho,t) \le \frac{8(k - 1)\e}{(2k - 3)\delta} \quad \text{ for } \mathcal{L}^1\text{-a.e.\@ } t \in E^+_j(\rho).
\end{equation}
Indeed, if we assume that (\ref{G+}) does not hold, then for every $E \subset (0,\rho/2)$ with $\mathcal{L}^1(E) = (2k - 3)\rho/4(k-1)$ we would have
\[
\frac{2\e}{\delta} \ge \frac{1}{\rho}\int_EG^+_j(\rho,t)\,dt > \frac{1}{\rho}\frac{(2k - 3)\rho}{4(k - 1)}\frac{8(k - 1)\e}{(2k - 3)\delta} = \frac{2\e}{\delta}.
\] 
Let $E^+_{\rho} \coloneqq \bigcap_{j \neq i_1}E^+_j(\rho)$, and notice that by De Morgan's law and (\ref{E+j})
\[
\mathcal{L}^1(E^+_{\rho}) = \frac{\rho}{2} - \mathcal{L}^1\left(\bigcup_{j \neq i_1}(E^+_j(\rho))^c\right) \ge \frac{\rho}{2} - \sum_{j \neq i_1}\mathcal{L}^1((E^+_j(\rho))^c) = \frac{\rho}{4}.
\]
A similar argument yields the existence of a set $E^-_{\rho} \subset (-\rho/2,0)$ with the property that $\mathcal{L}^1(E^-_{\rho}) \ge \rho/4$ and such the analogous statement to (\ref{G+}) holds for every $j \neq i_2$. Thus, by letting $\e = 1/m$ we can find two decreasing sequences $\{\rho_m\}_{m \in \N}$ and $\{r_m^+\}_{m \in \N}$ and an increasing sequence $\{r_m^-\}_{m \in \N}$ such that
\[
\arraycolsep=1.4pt\def\arraystretch{1.6}
\begin{array}{rll}
\displaystyle \rho_m & \in \left(0, \overline{\rho}\left(\frac{1}{m}\right)\right), & \\
\displaystyle r_m^+ & \in \left(\frac{\rho}{8}, \frac{\rho}{2}\right),  & \quad r_m^+ \in E^+_{\rho_m}, \\
\displaystyle r_m^- & \in \left(-\frac{\rho}{2}, -\frac{\rho}{8}\right), & \quad r_m^- \in E^-_{\rho_m},
\end{array}
\]
and with the property that $\mu(\partial R_m(x_0,\nu)) = 0$ (see \Cref{mubdry=0}), where 
\[
R_m(x_0,\nu) \coloneqq Q'(x_0,\nu,\rho_m) \times (r_m^-,r_m^+).
\]
Moreover, eventually extracting a subsequence (not relabeled), we can assume that for $\mathcal{L}^{N-1}$-a.e.\@ $x' \in Q'(x_0,\nu,\rho_m)$ it holds
\begin{equation}
\label{un->u}
\lim_{n\to\infty} u_n(x',r_m^{\pm}) = u(x',r_m^{\pm}),
\end{equation}
and 
\begin{equation}
\label{goodslices}
u(x',r_m^{\pm}) \in \{z_1(x',r_m^{\pm}),\dots, z_k(x',r_m^{\pm})\}.
\end{equation}
Notice in particular that there are constants $C, c$ such that 
\[
\overline{B(x_0,c\rho_m)} \subset R_m(x_0,\nu) \subset \overline{B(x_0,C\rho_m)}.
\]
Therefore, by Besicovitch's derivation theorem (see \Cref{thm:Bes}), \Cref{lem:convmeas}, and (\ref{goodpoint}) we obtain that 
\[
\frac{d\mu}{d\lambda}(x_0) = \lim_{m \to \infty}\frac{\mu(R_m(x_0,\nu))}{\lambda(R_m(x_0,\nu))} =  \lim_{m \to \infty}\frac{\mu(R_m(x_0,\nu))}{\rho_m^{N - 1}} = \lim_{m \to \infty}\lim_{n \to \infty}\frac{\mu_n(R_m(x_0,\nu))}{\rho_m^{N - 1}}.
\]
Thus, to prove the liminf inequality it is enough to show that 
\begin{equation}
\label{wts-liminf}
\lim_{m\to\infty} \lim_{n\to\infty} \frac{\mu_n(R_m(x_0,\nu))}{\rho_m^{N-1}} \ge d_W(x_0,z_{i_1}(x_0),z_{i_2}(x_0)).
\end{equation}
\textbf{Step 2:} Using the notation introduced in the previous step, we begin by noticing that
\begin{align*}
\mu_n(R_m(x_0,\nu)) & \ge \int_{R_m(x_0,\nu)}2\sqrt{W(x,u_n(x))}|\nabla u_n(x)|\, dx \\
& \ge \int_{Q'(x_0,\nu,\rho_m)}\int_{r_m^-}^{r_m^+}2\sqrt{W((x',t),u_n(x',t))}|\nabla u_n(x',t)\cdot \nu|\, dt\,dx'\\
& = \int_{Q'(x_0,\nu,\rho_m)}\int_{-1}^12\sqrt{W(x',g_m(t), \gamma_{m,n}(x',t))}|\gamma'_{m,n}(x',t)|\,dt\,dx',
\end{align*}
where 
\[
g_m(t) \coloneqq 
\left\{
\begin{array}{ll}
tr_m^+ & \text{ if } t \in (0,1),\\
& \\
tr_m^- & \text{ if } t \in (-1,0),
\end{array}
\right. 
\] 
and $\gamma_{m,n}(x',t) \coloneqq u_n(x',g_m(t))$. Then, if we set
\begin{equation}
\label{Fm}
F_m(z) \coloneqq \min\left\{2\sqrt{W(x,z)} : x \in \overline{R_m(x_0,\nu)}\right\}
\end{equation}
we obtain
\begin{align}
\label{muRm}
\mu_n(R_m) & \ge \int_{Q'(x_0,\nu,\rho_m)}\int_{-1}^1F_m(\gamma_{m,n}(x',t))|\gamma'_{m,n}(x',t)|\,dtdx' \notag \\
& \ge \int_{Q'(x_0,\nu,\rho_m)}d_{F_m}(\gamma_{m,n}(x',1), \gamma_{m,n}(x',-1))\,dx',
\end{align}
where for $p,q \in \R^M$
\begin{equation}
\label{dFm}
d_{F_m}(p,q) \coloneqq \inf \left\{\int_{-1}^1F_m(\gamma(t))|\gamma'(t)|\,dt : \gamma \in \mathcal{A}(p,q)\right\}.
\end{equation}
By (\ref{un->u}), the continuity of $d_{F_m}$, and Fatou's lemma, we see that 
\begin{equation}
\label{wtsliminf}
\liminf_{n \to \infty}\mu_n(R_m(x_0,\nu)) \ge \int_{Q'(x_0,\nu,\rho_m)}d_{F_m}(u(x',r_m^+), u(x',r_m^-))\,dx'.
\end{equation}
Let 
\[
Q'_m(i,j) \coloneqq Q'(x_0,\nu,\rho_m) \cap \{x' : u(x',r_m^+) = z_i(x',r_m^+), u(x',r_m^-) = z_j(x',r_m^-)\},
\]
and notice that in view of (\ref{goodslices}) we can write
\[
Q'(x_0,\nu,\rho_m) = T_m \cup \bigcup_{i,j = 1}^kQ'_m(i,j), 
\]
where $\mathcal{L}^{N - 1}(T_m) = 0$. Observe that as a consequence of (\ref{Fm}) and (\ref{dFm}) we have
\begin{align}
\label{CQmij}
\int_{Q'_m(i,j)}d_{F_m}(z_i(x',r_m^+), z_j(x',r_m^-))\,dx' & \le \int_{Q'_m(i,j)}d_W(x_0,z_i(x',r_m^+), z_j(x',r_m^-))\,dx' \notag \\
& \le C\mathcal{L}^{N - 1}(Q'_m(i,j)),
\end{align}
where $C$ is a constant that only depends on $W$, $z_i$, and $z_j$. Recalling the definition of $r_m^{\pm}$, if $i \neq i_1$ we have
\begin{equation}
\label{HQm}
\mathcal{L}^{N - 1}(Q_m(i,j)) \le \rho_m^{N - 1}G_i^+(\rho_m,r_m) \le \rho_m^{N - 1}\frac{8(k - 1)}{(2k - 3)\delta}\frac{1}{m}.
\end{equation}
Similar computations hold if $j \neq i_2$. Consequently, combining (\ref{wtsliminf}), (\ref{CQmij}), and (\ref{HQm}) we arrive at
\begin{align}
\label{-k/m}
\liminf_{n \to \infty}\mu_n(R_m) & \ge \sum_{i, j = 1}^k\int_{Q'_m(i,j)}d_{F_m}(z_i(x',r_m^+), z_j(x',r_m^-))\,dx' \notag \\
& \ge \int_{Q'(x_0,\nu,\rho_m)}d_{F_m}(z_{i_1}(x',r_m^+), z_{i_2}(x',r_m^-))\,dx' - \frac{C \rho_m^{N - 1}}{m},
\end{align}
for some positive constant $C$. Using the fact that for every $p,q,z \in \R^M$
\[
|d_{F_m}(p,z) - d_{F_m}(z,q)| \le d_{F_m}(p,q) \le d_W(x_0,p,q)
\]  
and the continuity of $d_W(x_0,\cdot,\cdot)$, we obtain
\begin{equation}
\label{putx0}
d_{F_m}(z_{i_1}(x',r_m^+), z_{i_2}(x',r_m^-)) \ge d_{F_m}(z_{i_1}(x_0), z_{i_2}(x_0)) + \mathcal{O}(1).
\end{equation}
Combining (\ref{-k/m}) and (\ref{putx0}), we see that 
\begin{align*}
\lim_{m \to \infty}\liminf_{n \to \infty}\frac{\mu_n(R_m)}{\rho_m^{N - 1}} & \ge \liminf_{m \to \infty} d_{F_m}(z_{i_1}(x_0), z_{i_2}(x_0)).
\end{align*}
Therefore, in order to prove (\ref{wts-liminf}) it is enough to show that
\begin{equation}
\label{finalwtsliminf}
\lim_{m \to \infty}d_{F_m}(z_{i_1}(x_0),z_{i_2}(x_0)) = d_W(x_0, z_{i_1}(x_0),z_{i_2}(x_0)).
\end{equation}
\textbf{Step 3:} This step is dedicated to the proof of (\ref{finalwtsliminf}). We claim that the sequence $\{F_m\}_{m \in \N}$ converges uniformly to $2\sqrt{W(x_0,\cdot)}$ on every compact subset of $\R^M$. Indeed, as one can readily check, the map $z \mapsto F_m(z)$ is continuous for every $m$, while the map $m \mapsto F_m(z)$ is nondecreasing for every $z$. To conclude, notice that $F_m(z) \to 2\sqrt{W(x_0,z)}$ for every $z \in \R^M$ as $m \to \infty$. Since the map $z \mapsto 2\sqrt{W(x_0,z)}$ is continuous by assumption (see (\ref{H1})), we are in a position to apply Dini's convergence theorem. This proves the claim.

Let $p \coloneqq z_{i_1}(x_0)$, $q \coloneqq z_{i_2}(x_0)$ and, using the notation introduced in \Cref{dFprop}, notice that (see (\ref{H2}); see also \Cref{RemR}) there exists $m_1$ such that if $m \ge m_1$ then for every $i \neq j$
\[
\mathcal{N}_{\delta/2}(z_i(R_m(x_0,\nu))) \cap \mathcal{N}_{\delta/2}(z_j(R_m(x_0,\nu))) = \emptyset.
\]
Thus, if $m \ge m_1$, we are in a position to apply the results of \Cref{dFprop} and \Cref{UB} and conclude that there exists a minimizing geodesic for $d_{F_m}(p,q)$, namely $\gamma_m \in \mathcal{A}(p,q)$, such that if we set
\[
\lambda \coloneqq \max\left\{|z_i(x)| : i \in \{1, \dots, k\}, x \in \overline{\o} \right\}
\]
then
\begin{align}
L(\gamma_m) & \le \Lambda(\lambda, \mathcal{R}_{m_1}(x_0,\nu)), \label{Lgmbound} \\
\|\gamma_m\|_{L^{\infty}((-1,1);\R^M)} & \le \Lambda(\lambda, \mathcal{R}_{m_1}(x_0,\nu)) + \lambda. \notag
\end{align} 
Consequently, given $\e > 0$, we can find $m_2 \ge m_1$ with the property that for every $m \ge m_2$ 
\begin{equation}
\label{FmWuniform}
\left|F_m(z) - 2\sqrt{W(0,z)}\right| < \e
\end{equation}
for every $z \in B(0, \Lambda(\lambda, \mathcal{R}_{m_1}(x_0,\nu)) + \lambda)$. Then, from (\ref{Lgmbound}) and (\ref{FmWuniform}) we see that
\begin{align*}
d_{F_m}(p,q) & = \int_{-1}^1F_m(\gamma_m(t))|\gamma_m'(t)|\,dt \\
& = \int_{-1}^12\sqrt{W(x_0,\gamma_m(t))}|\gamma_m'(t)|\,dt - \int_{-1}^1\left[2\sqrt{W(x_0,\gamma_m(t))} - F_m(\gamma_m(t))\right]|\gamma_m'(t)|\,dt \\
& \ge d_W(x_0,p,q) - \e \Lambda(\lambda, \mathcal{R}_{m_1}(x_0,\nu)).
\end{align*}
Letting $m \to \infty$ in the previous inequality we obtain
\[
d_W(x_0,p,q) \ge \limsup_{m \to \infty}d_{F_m}(p,q) \ge \liminf_{m \to \infty}d_{F_m}(p,q) \ge d_W(x_0,p,q) - \e\Lambda(\lambda, \mathcal{R}_{m_1}(x_0,\nu)).
\]
The desired inequality (\ref{finalwtsliminf}) follows immediately by letting $\e \to 0^+$.
\end{proof}


\subsection{Limsup inequality}\label{sec:limsup}

This section is devoted to proving the following result, which concludes the proof of \Cref{thm:main}.

\begin{proposition}\label{prop:limsup}
Let $W$ be as is \eqref{H1}--\eqref{H4}. Then, for every $u \in BV(\o; z_1, \dots, z_k)$ there exists a sequence $\{u_n\}_{n \in \N}$ of functions in $H^1(\o; \R^M)$ such that $u_n \to u$ in $L^1(\o; \R^M)$ and $\f_n(u_n) \to \f_0(u)$.
\end{proposition}

We split the proof in three lemmas. In the first result presented below, we exhibit a sequence
$\{u_n\}_{n \in \N}$ of functions with polyhedral jump set which approximates the function $u$ and furthermore satisfies $\f_0(u_n) \to \f_0(u)$.
This is a straightforward consequence of a classical result due to Baldo (see \cite[Lemma 3.1]{Bal}).
In the following we say that $U \subset \o$ is a \emph{polyhedral set} if $\partial U \cap \o$ is contained in a finite union of hyperplanes.

\begin{lemma}\label{lem:approxregular}
Under the assumptions of \Cref{prop:limsup}, let $u\in BV(\o;z_1,\dots,z_k)$. Then there exists a sequence $\{u_n\}_{n \in \N}$ of functions in $BV(\o; z_1, \dots, z_k)$ of the form
\[
u_n = \sum_{i=1}^k z_i \ca_{\o^i_n},
\]
where $\o^i_n$ is a polyhedral set for each $i \in \{1,\dots,k\}$ and each $n \in \N$, such that $u_n \to u$ in $L^1(\o; \R^M)$ and $\f_0(u_n)\to \f_0(u)$.
\end{lemma}

\begin{proof}
Write $u = \sum_{i=1}^k z_i\ca_{\o^i}$, where $\{\o^i\}_{i=1}^k$ is a Caccioppoli partition of $\o$ (see \Cref{CP} and \Cref{CPR}). Thanks to \cite[Proposition A.2]{Bal} and as a consequence of the proof of \cite[Lemma A.7]{Bal}, it is possible to find a sequence $\{\{\o^i_n\}_{i=1}^k\}_{n \in \N}$ of Caccioppoli partitions of $\o$ with the following properties:
\begin{itemize}
\item[$(i)$] $\o^i_n$ is a polyhedral set;
\item[$(ii)$] $\ca_{\o^i_n}\to\ca_{\o^i}$ in $L^1(\o)$;
\item[$(iii)$] $\mu^{ij}_n$ converges in $(C_b(\o))'$ to $\mu^{ij}$(see Definition \ref{def:Cb}), where for each $i\neq j\in\{1,\dots,k\}$, and $n \in \N$ the measures $\mu^{ij}_n, \mu^{ij} \in \mathcal{M}^+(\o)$ are given by
\[
\mu^{ij}_n\coloneqq\hno\restr(\partial^*\o^i_n\cap\partial^*\o^j_n),\quad\quad\quad\quad
\mu^{ij}\coloneqq\hno\restr(\partial^*\o^i\cap\partial^*\o^j).
\]
\end{itemize}

Since $d_W\in C_b(\o)$ (see \Cref{dWisLipsc}) and recalling that the $z_i$'s are bounded we get
\begin{align*}
\lim_{n\to\infty}\f_0(u_n) & = \lim_{n\to\infty}\sum_{i=1}^{k - 1} \sum_{j=i+1}^k \int_{\partial^*\o^i_n\cap \partial^*\o^j_n}d_W(x, z_i(x), z_j(x)) \dhno(x) \\
& = \lim_{n\to\infty}\sum_{i=1}^{k - 1} \sum_{j=i+1}^k \int_\o d_W(x, z_i(x), z_j(x)) \,d \mu^{ij}_n(x) \\
& = \sum_{i=1}^{k - 1} \sum_{j=i+1}^k \int_\o d_W(x, z_i(x), z_j(x)) \,d\mu^{ij}(x) \\
& = \f_0(u),
\end{align*}
where in the previous to last step we used property $(iii)$ above.
\end{proof}

A key ingredient in the proof of the limsup inequality is a reparametrization due to Modica (see \cite[Proposition 2]{Mod_Gra}, and \cite[Lemma 3.2]{Bal}).
Being crucial to our construction of the recovery sequence, we present the proof for the reader's convenience.

\begin{lemma}\label{lem:constructionprofile}
Under the assumptions of \Cref{prop:limsup}, fix $\lambda>0$, $\e>0$, $x\in\o$, and $p,q\in\R^M$.
Let $\gamma\in \mathcal{A}(p,q)$ be the parametrization of a curve of class $C^1$ with $\gamma'(s) \neq 0$ for all $s \in (-1,1)$.
Then there exist $\tau>0$ with
\begin{equation}\label{eq:eta}
\tau \le \frac{\e}{\sqrt{\lambda}}L(\gamma) = \frac{\e}{\sqrt{\lambda}}\int_{-1}^1 |\gamma'(t)|\,dt,
\end{equation}
and $g \in C^1((0,\tau); [-1,1])$ such that
\begin{equation}\label{eq:reparametrization}
(g'(t))^2 = \frac{\lambda + W(x,\gamma(g(t)))}{\e^2 |\gamma'(g(t))|^2}
\end{equation}
for all $t\in(0,\tau)$, $g(0)=-1$, $g(\tau)=1$, and
\begin{multline}
\label{modicarepr}
\int_{0}^{\tau} \left[\frac{1}{\e} W\left(x, \gamma(g(t)) \right) +\e |\gamma'\left( g(t) \right)|^2\left(g'(t)\right)^2 \right] \,dt \\ 
\le \int_{-1}^1 2 \sqrt{W(x, \gamma(s))}|\gamma'(s)|\,ds + 2\sqrt{\lambda} \int_{-1}^1 |\gamma'(s)| \,ds.
\end{multline}
\end{lemma}

\begin{proof}
Define
\[
\Psi(t) \coloneqq \int_{-1}^t \frac{\e |\gamma'(s)|}{\sqrt{\lambda + W(x,\gamma(s))}}\,ds,
\]
and set $\tau\coloneqq\Psi(1)$.
Then $\Psi$ is strictly increasing, and its inverse $g \colon [0,\tau] \to [-1,1]$
is of class $C^1$ in $(0,\tau)$ and satisfies \eqref{eq:reparametrization} for all $t \in (0,\tau)$. Moreover,
\[
\tau = \Psi(1) = \int_{-1}^1 \frac{\e |\gamma'(s)|}{\sqrt{\lambda + W(x,\gamma(s))}} \,ds
	\le \frac{\e}{\sqrt{\lambda}}\int_{-1}^1 |\gamma'(s)| \,ds.
\]
Finally, by (\ref{eq:reparametrization}), a change of variables, and the subadditivity of the square root function we get
\begin{align*}
& \int_0^\tau \left[\frac{1}{\e} W(x, \gamma(g(t))) +\e |\gamma'( g(t))|^2 (g'(t))^2\right] \,dt \\
& \hspace{3cm} = \int_0^\tau 2 |\gamma'( y(t))| |g'(t)| \sqrt{W\left(x, \gamma(g(t)\right) + \lambda} \,dt \\
& \hspace{3cm} = \int_{-1}^1 2 \sqrt{W(x, \gamma(s)) + \lambda}|\gamma'(s)| \,ds \\
& \hspace{3cm} \le \int_{-1}^1 2 \sqrt{W(x, \gamma(s))}|\gamma'(s)|\,ds + 2\sqrt{\lambda}\int_{-1}^1 |\gamma'(s)|\,ds,
\end{align*}
and (\ref{modicarepr}) follows, thus concluding the proof.
\end{proof}

We now turn to the technical construction of the recovery sequence.

\begin{lemma}\label{lem:approxenergy}
Under the assumptions of \Cref{prop:limsup}, let $u \in BV(\o; z_1,\dots,z_k)$ be such that $u = \sum_{i=1}^k z_i \ca_{\o^i}$, where, for each $i=1,\dots,k$, $\partial\o^i \cap \o$ is a polyhedral set. Then there exists a sequence $\{u_n\}_{n \in \N}$ of functions in $H^1(\o; \R^M)$ such that $u_n \to u$ in $L^1(\o; \R^M)$ and $\f_n(u_n) \to \f_0(u)$ as $n\to\infty$.
\end{lemma}

\begin{proof}
Throughout the proof we denote by $C$ a positive constant independent of $n$, which could possibly differ from line to line.
We divide the proof into several steps.
\newline
\textbf{Step 1:} (\emph{definitions of the main objects}).
For $n\in\N$ define
\begin{equation}\label{eq:rn}
r_n\coloneqq\e_n^{\frac{1}{6}},
\quad\quad\quad\quad
r'_n\coloneqq\left(1-\e_n^{\frac{2}{3}}\right)r_n.
\end{equation}
Thanks to \Cref{dFprop} and \Cref{UB}, there exists $\overline{L}>0$ such that for every $x \in \o$ and $i\neq j\in\{1,\dots,k\}$ it is possible to find a minimizing geodesic $\gamma\in\mathcal{A}(z_i(x),z_j(x))$ for $d_W(x, z_i(x), z_j(x))$ satisfying
\begin{equation}\label{eq:L}
L(\gamma)< \overline{L}.
\end{equation}
For each $n\in\N$ set
\[
\ell_n \coloneqq \overline{L}\e_n^{\frac{5}{8}}.
\]
For each $i\in\{1,\dots,k\}$ write
\[
\partial\o^i\cap\o = \bigcup_{m=1}^{p_i} \overline{\Sigma^i_m}\cap\o,
\]
where $p_i\in\N$, the $\Sigma^i_m$'s are pairwise disjoint, and, for each $m\in\{1,\dots,p_i\}$, $\Sigma^i_m$ is contained in a hyperplane with normal $\nu^i_m \in \S^{N-1}$.
Define the \emph{singular set} of the partition $\{\o^i\}_{i = 1}^k$ as
\[
S \coloneqq \bigcup_{i=1}^k \left[ \bigcup_{m \neq s} \left( \overline{\Sigma_m^i} \cap \overline{\Sigma_s^i} \right) \cup \left( \overline{\partial\o^i} \cap \partial\o \right) \right].
\]
For $i \in \{1, \dots, k\}$ and $m \in \{1, \dots, p_i\}$ set
\[
A^n_{m,i} \coloneqq \left\{ x \in \R^N : x = y + t \nu^i_m, y \in \Sigma^i_m \setminus \mathcal{N}_{\Theta\e_n}(S),
    t \in \left(-\ell_n - \e_n, \ell_n + \e_n\right)\right\},
\]
where $\mathcal{N}_{\Theta\e_n}(S) \coloneqq \{ x \in \R^N : |x-y| < \Theta \e_n \text{ for some } y \in S \}$.
Using the fact that $\partial\o$ is Lipschitz, it is possible to find $\overline{n}\in\N$ and $\Theta>0$ such that for all $n\geq\overline{n}$, $i \in \{1, \dots, k\}$ and $m\in\{1,\dots,p_i\}$ it holds that $A^n_{m,i} \subset \o$, and furthermore that either $A^n_{m,i} = A^n_{s,j}$, or
\begin{equation}\label{eq:disjoint}
A^n_{m,i} \cap A^n_{s,j}=\emptyset
\end{equation} 
for all $n\geq\overline{n}$, $i,j\in\{1,\dots,k\}$, $m\in\{1,\dots,p_i\}$, and $s\in\{1,\dots,p_j\}$ with $m\neq s$ if $i = j$ (see \Cref{limsupfigure}).
Without loss of generality, up to increasing the value of $\overline{n}$, we can assume that for all $i\in\{1,\dots,k\}$ and $m\in\{1,\dots,p_i\}$ there exists only two different indexes $j_1,j_2\in\{1,\dots,k\}$ such that
\[
\left\{ x \in \R^N : x = y + t \nu^i_m, y \in \Sigma^i_m \setminus \mathcal{N}_{\Theta\e_n}(S),
    t \in \left(-0, \ell_n + \e_n\right)\right\} \subset \o^{j_1}
\]
and
\[
\left\{ x \in \R^N : x = y + t \nu^i_m, y \in \Sigma^i_m \setminus \mathcal{N}_{\Theta\e_n}(S),
    t \in \left(-\ell_n - \e_n, 0 \right)\right\} \subset \o^{j_2}.
\]
For $n\in\N$ set $S_n\coloneqq\mathcal{N}_{\Theta \e_n}(S)$, and notice that
\begin{equation}\label{eq:Sn}
\mathcal{L}^N(S_n)\le D \e_n^2,
\end{equation}
for some constant $D>0$ depending only on $\mathcal{H}^{N-2}(S)$.

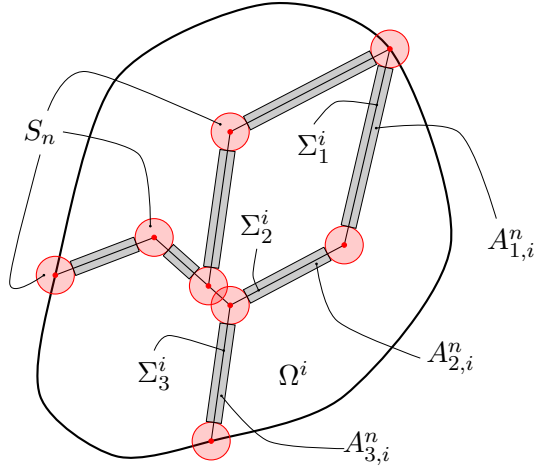
\begin{figure}
\begin{center}
\begin{tikzpicture}[blend group=normal, scale=1.0]

\draw [thick, fill = black, fill opacity = 0.0] plot [smooth cycle] coordinates {(2.0,.2) (3.25,0.4) (5,1) (6.4,3) (5.6,5.6) (4.0, 6.2) (2.0,5.2) (1.2, 2.6) (1,1)};
\draw (3.25,0.4) -- (3.5,2.2) -- (2.5, 3.1) -- (1.2,2.6);
\draw (3.214, 2.457) -- (3.5,4.5) -- (5.6,5.6) -- (5, 3) -- (3.5,2.2);

\filldraw [red] (3.25,0.4) circle (0.03); 
\filldraw [red] (3.5,2.2) circle (0.03);
\filldraw [red] (2.5, 3.1) circle (0.03);
\filldraw [red] (1.2,2.6) circle (0.03);
\filldraw [red] (3.214, 2.457) circle (0.03);
\filldraw [red] (3.5,4.5) circle (0.03);
\filldraw [red] (5.6,5.6) circle (0.03);
\filldraw [red] (5, 3) circle (0.03);

\filldraw [red, fill opacity = 0.2] (3.25,0.4) circle (0.25); 
\filldraw [red, fill opacity = 0.2] (3.5,2.2) circle (0.25);
\filldraw [red, fill opacity = 0.2] (2.5, 3.1) circle (0.25);
\filldraw [red, fill opacity = 0.2] (1.2,2.6) circle (0.25);
\filldraw [red, fill opacity = 0.2] (3.214, 2.457) circle (0.25);
\filldraw [red, fill opacity = 0.2] (3.5,4.5) circle (0.25);
\filldraw [red, fill opacity = 0.2] (5.6,5.6) circle (0.25);
\filldraw [red, fill opacity = 0.2] (5, 3) circle (0.25);

\draw [very thin, fill = black, fill opacity = 0.2] (3.675, 4.705) -- (3.768, 4.527) -- (5.425, 5.395) -- (5.332, 5.573) -- (3.675, 4.705);
\draw [very thin, fill = black, fill opacity = 0.2] (3.185, 0.661) -- (3.383, 0.634) -- (3.565, 1.939) -- (3.367, 1.966) -- (3.185, 0.661);
\draw [very thin, fill = black, fill opacity = 0.2] (4.959, 3.266) -- (5.154, 3.221) -- (5.641, 5.334) -- (5.446, 5.379) -- (4.959, 3.266);
\draw [very thin, fill = black, fill opacity = 0.2] (2.231, 3.104) -- (2.303, 2.917) -- (1.469, 2.596) -- (1.397, 2.783) -- (2.231, 3.104);
\draw [very thin, fill = black, fill opacity = 0.2] (3.150, 2.718) -- (3.348, 2.691) -- (3.564, 4.239) -- (3.366, 4.266) -- (3.150, 2.718);
\draw [very thin, fill = black, fill opacity = 0.2] (4.732, 2.971) -- (4.826, 2.794) -- (3.768, 2.229) -- (3.674, 2.406) -- (4.732, 2.971);
\draw [very thin, fill = black, fill opacity = 0.2] (2.619, 2.858) -- (2.753, 3.007) -- (3.095, 2.699) -- (2.961, 2.550) -- (2.619, 2.858);

\node at (1, 4.5) {$S_n$};
\draw [very thin] plot [smooth] coordinates {(1, 4.2) (0.6, 2.5) (1.05, 2.6)}; 
\filldraw (1.05, 2.6) circle (0.01); 
\draw [very thin] plot [smooth] coordinates {(1.3, 4.5) (2.3, 4.2) (2.45, 3.25)}; 
\filldraw (2.45, 3.25) circle (0.01); 
\draw [very thin] plot [smooth] coordinates {(1.1, 4.78) (1.6, 5.3) (3.37, 4.6)}; 
\filldraw (3.37, 4.6) circle (0.01); 

\node [right] at (4,1.3) {$\Omega^i$};

\node at (2.5, 1.3) {$\Sigma^i_3$};
\draw [very thin] plot [smooth] coordinates {(2.5, 1.6) (2.7, 1.8) (3.415,1.6)}; 
\filldraw (3.415,1.6) circle (0.01); 

\node at (3.85, 3.25) {$\Sigma^i_2$};
\draw [very thin] plot [smooth] coordinates {(3.85, 2.95) (3.8, 2.6) (3.875, 2.4)}; 
\filldraw (3.875, 2.4) circle (0.01); 

\node at (4.6, 4.3) {$\Sigma^i_1$};
\draw [very thin] plot [smooth] coordinates {(4.6, 4.6) (4.9, 4.9) (5.428, 4.857)}; 
\filldraw (5.428, 4.857) circle (0.01); 

\node at (7.2, 3) {$A_{1,i}^n$};
\draw [very thin] plot [smooth] coordinates {(7, 3.35) (6.5, 4.5) (5.42,4.55)}; 
\filldraw (5.42,4.55) circle (0.01); 

\node at (6.4, 1.5) {$A_{2,i}^n$};
\draw [very thin] plot [smooth] coordinates {(6, 1.5) (4.9, 2) (4.635, 2.75)}; 
\filldraw (4.635, 2.75) circle (0.01); 

\node at (5.3, 0.3) {$A_{3,i}^n$};
\draw [very thin] plot [smooth] coordinates {(5, 0.3) (4, 0.4) (3.385, 1)}; 
\filldraw (3.385, 1) circle (0.01); 

\end{tikzpicture}
\caption{The region shaded in red depicts the tubular neighborhood of the singular set of the partition $\{\o^i\}_{i = 1}^k$. The main construction of the recovery sequence is carried out in the region shaded in grey.} 
\label{limsupfigure}
\end{center}
\end{figure}

\textbf{Step 2:} (\emph{definition of $u_n$ close to $\Sigma^i_m$}).
Let us first consider $\o^1\setminus S_n$.
Fix $m \in \{1,\dots, p_1\}$.
Without loss of generality we can assume that $\Sigma^1_m\subset\{x_N=0\}$ and that
the normal vector $\nu^1_m = e_N$ points outside $\o^1$.
In the following we write points $x \in \R^N$ as $x = (y,t)$ where $y \in \R^{N-1}$ and $t \in \R$.
For each $n \in \N$, consider the sets
\[
U_n^j \coloneqq \{ y\in\R^{N-1} : (y,0)\in Q(\overline{y_j}, r_n) \cap (\Sigma^1_m\setminus S_n) \},
\]
where $\overline{y_j}$ ranges among the elements of $\frac{r_n}{2}\Z^{N-1}\times\{0\}$.
In the following we will consider only the indices $j$ for which $U_n^j \neq \emptyset$.
Let $\varphi_n^j\in C_c^\infty(\R^{N-1})$ be a function such that
\begin{equation}\label{eq:varphi1}
0\le\varphi_n^j\le1,
\quad\quad
\varphi_n^j\equiv 1 \text{ in } Q'(\overline{y_j}, e_N, r'_n),
\quad\quad
\varphi_n^j\equiv 0 \text{ on } \R^{N-1}\setminus Q'(\overline{y_j} , e_N, 2r_n-r'_n),
\end{equation}
\begin{equation}\label{eq:varphi2}
\sum_j \varphi_n^j=1 \text{ in } \Sigma^1_m,
\quad\quad\quad\quad
| \nabla \varphi_n^j |\le\frac{C}{r_n-r'_n},
\end{equation}
where $C$ is a positive constant independent of $n$, and we used the notation in \eqref{eq:qprime}.
For every such index $j$, let $\gamma^j\in\mathcal{A}(u^-(\overline{y_j}),  u^+(\overline{y_j}))$ be a minimizing geodesic for $d_W(\overline{y_j}, u^-(\overline{y_j}), u^+( \overline{y_j}))$
such that $L(\gamma^j) < \overline{L}$, where $\overline{L}$ is the constant given in \eqref{eq:L}.
Then it is possible to find a sequence $\{\gamma^j_n\}_{n\in\N}$ of $C^1$ curves in $\mathcal{A}(u^-(\overline{y_j}),  u^+(\overline{y_j}))$ with $(\gamma^j_n)'(s)\neq 0$ for all $s\in(-1,1)$ and
such that
\begin{equation}\label{eq:approxcurves}
L(\gamma^j_n)<\overline{L},\quad\quad\quad
\lim_{n\to\infty} \int_{-1}^1 \sqrt{W( \overline{y_j}, \gamma^j_n(s) )} |(\gamma^j_n)'(s)| \,ds = d_W( \overline{y_j},
u^+( \overline{y_j}) , u^-(\overline{y_j})   ).
\end{equation}
Let $\tau^j_n\in(0,\ell_n)$ and $g^j_n\in C^1((0,\tau^j_n))$ be given by \Cref{lem:constructionprofile} corresponding to the choice of
\[
\e=\e_n, \quad\quad\quad \lambda=\e_n^{\frac{3}{4}}, \quad\quad\quad \gamma=\gamma^j_n.
\]
Extend $g^j_n$ to the whole interval $(0,\infty)$ by setting $g^j_n(s) \coloneqq 1$ for $s \geq \tau^j_n$.
For each $j$, let $s_j,q_j\in\{1,\dots,k\}$ be such that $u^-(\overline{y_j})=z_{s_j}(\overline{y_j})$, and $u^+(\overline{y_j})=z_{q_j}(\overline{y_j})$.
We then define the function $u_n^j$ as follows: 
\begin{itemize}
\item[$(i)$] for $y \in U_n^j$ and $t \in (0, \ell_n)$ set
\[
u^j_n(y,t) \coloneqq \gamma^j_n(g^j_n(t));
\]
\item[$(ii)$] for $y \in U_n^j$ and $t \in (\ell_n, \ell_n+\e_n)$ set
\[
u^j_n(y,t) \coloneqq z_{q_j}(\overline{y_j}) + \frac{t - \ell_n}{\e_n}
	\left[ z_{q_j} \left(y, \ell_n+\e_n \right) - z_{q_j}(\overline{y_j}) \right];
\]
\item[$(iii)$] for $y \in U_n^j$ and $t \in (-\e_n,0)$ set
\[
u^j_n(y,t) \coloneqq z_{s_j}(\overline{y_j}) - \frac{t}{\e_n} \left[ z_{s_j} \left(y, -\e_n \right) - z_{s_j}(\overline{y_j}) \right].
\]
\end{itemize}
Then, for $(y,0) \in \Sigma^1_m \setminus S_n$ and $t \in (-\e_n, \e_n + \ell_n)$ define
\begin{equation}\label{eq:defun}
u_n(y,t) \coloneqq \sum_j \varphi_j(y) u^j_n(y,t).
\end{equation}
Notice that, thanks to \eqref{eq:varphi1}, $u_n(y,t) = u_n^j(y,t)$ for $y \in Q'(\overline{y_j},e_N,  r'_n)$, and the number of non-zero terms in the sum is bounded above by $3^N$.
We repeat the same construction for all $m \in \{1, \dots, p_1\}$, and notice that thanks to \eqref{eq:disjoint} the functions $u_n$ are well defined, as the constructions in this step do not overlap.

Next, for 
\[
x \in \o^1 \setminus \left(  S_n \cup \bigcup_{m=1}^{p_1} \left\{ z + t\nu^1_m \in \R^N : 
z\in \Sigma^1_m \setminus S_n,  t \in \left(-\e_n, \e_n+\ell_n\right) \right\} \right),
\]
define $u_n(x) \coloneqq z_1(x)$.

Finally, we repeat the same argument for $i = 2, \dots, k$ in the sets
\[
\o^i\setminus \left(  S_n \cup \bigcup_{j<i} \,\,\bigcup_{m=1}^{p_j} \left\{ z + t\nu^j_m \in \R^N :  z\in \Sigma^j_m \setminus S_n,  t \in \left(-\e_n, \e_n+\ell_n\right)\right\} \right),
\]
assuming that the normal $\nu^j_m$ points outside the set $\o^j$.
\newline
\textbf{Step 3:} (\emph{estimate of the Lipschitz constant of $u_n$ close to $\Sigma^i_m$}).
Reasoning as in the previous step, we assume without loss of generality that $\Sigma^i_m\subset\{x_N=0\}$.
Using the fact that the number of non-zero terms in the definition of $u_n$ (see \eqref{eq:defun}) is bounded above by $3^N$, we have that
\begin{align}\label{eq:estLipun}
\Lip(u_n) \leq 3^N\sup_j \Lip(\varphi_j u^j_n) 
\leq 3^N C \sup_j [ \Lip(\varphi_j) +  \Lip(u^j_n)]
\end{align}
where in the second step we used  \eqref{eq:varphi1} together with the fact that, by construction, and using \eqref{eq:L} it holds
\begin{equation}\label{eq:Linfty}
\max_j \| u^j_n \|_{L^\infty(\o;\R^M)}\le C\,
\end{equation}
for some constant $C>0$ independent of $n$.
In order to estimate the second term in \eqref{eq:estLipun} we reason as follows.
Let $y_1, y_2\in U^j_n$ and $t_1,t_2\in(0,\ell_n)$.
Using \eqref{eq:reparametrization}, our choice of $\e$ and $\lambda$, and the fact that minimizing geodesics are uniformly bounded, we get
\begin{align}\label{eq:Lip1}
|u^j_n(y_1,t_1) - u^j_n(y_2,t_2)| & = | \gamma^j_n(g^j_n(t_2)) - \gamma^j_n(g^j_n(t_1)) | \nonumber \\
&\leq \frac{(\e_n^{\frac{3}{4}} + C )^{\frac{1}{2}}}{\e_n}  |t_1-t_2|.
\end{align}
Let $y_1, y_2\in U^j_n$ and $t_1,t_2\in[\ell_n, \ell_n+\e_n)$.
Then we can estimate
\begin{align}\label{eq:Lip2}
|u_n^j(y_1, t_1) - u_n^j(y_2, t_2)| & \le \left| u_n^j(y_1, t_1) - \frac{t_1-\ell_n}{\e_n}
	\left[ z_{q_j}(y_2, \ell_n+\e_n)  -  z_{q_j}(\overline{y_j}) \right] \right|  \nonumber \\
&\hspace{2cm}	+ \left |\frac{t_1-\ell_n}{\e_n}\left[ z_q(y_2, \ell_n+\e_n) 
	-  z_{q_j}\overline{y_j} \right]  - u_n^j(y_2, t_2) \right| \nonumber \\
&\leq  \frac{t_1-\ell_n}{\e_n}
	\left| z_{q_j}(y_1,  \ell_n+ \e_n) - z_{q_j}(y_2, \ell_n+\e_n) \right| \nonumber\\
&\hspace{3cm}+|t_2 - t_1|\left| \frac{z_{q_j}(y_1, \ell_n+\e_n)
	- z_{q_j}(\overline{y_j})}{\e_n} \right| \nonumber \\
&\le \Lip(z_{q_j}) |y_1-y_2| + C|t_2 - t_1|\frac{(\e_n^{\frac{5}{4}} + r_n^2)^{\frac{1}{2}}}{\e_n}.
\end{align}
Similar computations show that, if $y_1,y_2 \in U^j_n$ and $t_1, t_2 \in(-\e_n,0)$, we get
\begin{equation}\label{eq:Lip3}
|u_n^j(y_1, t_1) - u_n^j(y_2, t_2)|  \le \Lip(z_{s_j}) |y_1-y_2| + C|t_2 - t_1|\frac{\sqrt{\e_n^2 + r_n^2}}{\e_n}.
\end{equation}
\textbf{Step 4: } (\emph{definition of $u_n$ in $\o$}).
We are now in position to define $u_n$ in the whole $\o$.
Using the estimates \eqref{eq:estLipun}, \eqref{eq:Lip1}, \eqref{eq:Lip2}, and \eqref{eq:Lip3}, Kirszbraun's theorem ensures that it is possible to extend $u_n$ to a Lipschitz function, still denoted by $u_n$, defined in the whole $\o$ in such a way that the Lipschitz constant of the extension is controlled by the Lipschitz constant of the original function $u_n$. It is immediate to verify that $u_n \in H^1(\o;\R^M)$ and that $u_n \to u$ in $L^1(\o;\R^M)$ as $n \to \infty$.
\newline
\textbf{Step 5: } (\emph{estimate of the energy}).
We claim that $\f_n(u_n)\to\f_0(u)$ as $n\to\infty$.
To show this, we split $\o$ into several pieces and compute the asymptotic behavior of the energy $\f_n$ in each of them.
For $i\in\{1,\dots,k\}$,  let $\widetilde{\o}^i_n$ be the set where $u_n\equiv z_i$.
Then
\begin{equation}\label{eq:estenergy1}
\lim_{n\to\infty}\int_{\widetilde{\o}^i_n} \left[\, \frac{1}{\e_n}W(x,u_n(x))
	+ \e_n |\nabla u_n(x)|^2 \,\right] \,d x = 0,
\end{equation}
since $W(x, u_n(x)) = 0$ for $x \in \widetilde{\o}^i$, and $u_n$ has uniformly bounded gradient in that region.

Fix one connected component $\Sigma$ of $(J_u \cap \o) \setminus S_n$.
Without loss of generality, we can assume $\Sigma \subset \{x_N = 0\}$.
We first consider the transition region
\[
I_n\coloneqq\bigcup_j \left\{ (y,t)\in\R^N : (y,0)\in
\Sigma\cap \left( Q(\overline{y_j}, r_n)\setminus Q(\overline{y_j}, r'_n) \right),\,
	t\in (-\e_n,  \ell_n+\e_n) \right\},
\]
which we split in three parts. Let us start with
\[
I^{1}_n\coloneqq\bigcup_j\left\{ (y,t)\in\R^N : (y,0)\in
\Sigma\cap \left( Q(\overline{y_j}, r_n)\setminus Q(\overline{y_j}, r'_n) \right),\,
	t\in (-\e_n,  0) \right\} .
\]
Notice that
\begin{equation}\label{eq:I1n}
\mathcal{L}^N(I^{1}_n) \le C \left(\frac{r_n-r'_n}{r_n}\right)^{N-1}\e_n.
\end{equation}
Indeed, the number of $(N-1)$-dimensional cubes we consider is of the order of $r_n^{1-N}$, for each of which we are integrating over a volume of the order of $\e_n (r_n-r'_n)^{N-1}$.
Since $\|u_n\|_{L^\infty(\o;\R^M)} \le C$ we have
\begin{equation}\label{eq:boundW}
W(x,u_n(x)) \le C,
\end{equation}
for all $x\in I_n$.
Using \eqref{eq:estLipun}, \eqref{eq:Lip3}, \eqref{eq:I1n}, and \eqref{eq:boundW} we get
\begin{align}\label{eq:estenergy2a}
&\int_{I^1_n}  \left[ \frac{1}{\e_n}W(x,u_n(x)) + \e_n |\nabla u_n(x)|^2 \right] \,d x
	\nonumber \\
&\hspace{2cm} \le \left[ \frac{C}{\e_n}
	+ \e_n \left( \frac{C}{r_n-r'_n} + \Lip(u^+)
	+ \frac{\sqrt{\e_n^2 + r_n^2}}{\e_n}
		\right)^2 \right] \mathcal{L}^N(I^1_n) \nonumber \\
&\hspace{2cm}\le C\left[ \frac{1}{\e_n}
	+  \e_n \left( \frac{1}{(r_n-r'_n)^2} + (\Lip(u^+))^2
	+ \frac{\e_n^2 + r_n^2}{\e_n^2} \,\right) \right] 
		\left(\frac{r_n-r'_n}{r_n}\right)^{N-1}\e_n  \nonumber \\
&\hspace{2cm}\le C\left[ 1
	+  \left( \frac{\e_n}{r_n-r'_n}\right)^2 + \e_n^2(\Lip(u^+))^2
	+ \e_n^2 + r_n^2  \right] 
		\left(\frac{r_n-r_n'}{r_n}\right)^{N-1}  \nonumber \\
&\hspace{2cm}= C\left[ 1 + \left( \frac{\e_n}{\e_n^{\frac{1}{6}+\frac{2}{3}}} \right)^2
	+\e_n^2(\Lip(u^+))^2 + \e_n^2 + r_n^2  \right] 
		\e_n^{\frac{2(N-1)}{3}} \nonumber \\
&\hspace{2cm}\to 0,
\end{align}
as $n\to\infty$, where in the previous to last step we used \eqref{eq:rn}.

Next, we prove that the energy contribution in the region
\[
I^{2}_n\coloneqq\bigcup_j \left\{ (y,t)\in\R^N : (y,0)\in
\Sigma\cap \left( Q(\overline{y_j}, r_n)\setminus Q(\overline{y_j}, r'_n) \right),\,
	t\in (0,  \ell_n) \right\},
\]
is asymptotically negligible. We do this by noticing that
\[
\mathcal{L}^N(I^{2}_n) \le C \left(\frac{r_n-r'_n}{r_n}\right)^{N-1}\e_n^{\frac{5}{8}},
\]
and that, in view of \eqref{eq:Lip1} and \eqref{eq:boundW}, we have that
\begin{align}\label{eq:estenergy2b}
\int_{I^2_n}  \left[ \frac{1}{\e_n}W(x,u_n(x)) + \e_n |\nabla u_n|^2 \right] \,d x
&\leq \left[ \frac{C}{\e_n} + \e_n\frac{\e_n^{\frac{3}{4}}+C}{\e_n^2} \right] \mathcal{L}^N(I^2_n) \nonumber \\
&\le C\left[ \e_n^{-1} + \e_n^{-\frac{1}{4}}  \right]  \left(\frac{r_n-r'_n}{r_n}\right)^{N-1}\e_n^{\frac{5}{8}}   \nonumber \\
&\leq C \e_n^{-\frac{3}{8}+\frac{2}{3}(N-1)} + C\e_n^{\frac{2}{3}(N-1)+\frac{3}{8}}  \nonumber \\
&\to 0,
\end{align}
as $n\to\infty$. Finally, we consider
\[
I^{3}_n\coloneqq \bigcup_j \left\{ (y,t)\in\R^N : (y,0)\in
\Sigma\cap \left( Q(\overline{y_j}, r_n)\setminus Q(\overline{y_j}, r'_n) \right),\,
	t\in (\ell_n,  \ell_n+\e_n) \right\},
\]
and notice that
\[
\mathcal{L}^N(I^{3}_n)\leq C \left(\frac{r_n-r'_n}{r_n}\right)^{N-1}\e_n.
\]
In turn, by \eqref{eq:Lip2} and with similar computations to the ones in \eqref{eq:estenergy2a}, we get
\begin{equation}\label{eq:estenergy2c}
\lim_{n\to\infty} \int_{I^3_n}  \left[ \frac{1}{\e_n}W(x,u_n(x)) + \e_n |\nabla u_n(x)|^2 \right] \,d x = 0.
\end{equation}
Therefore, from \eqref{eq:estenergy2a}, \eqref{eq:estenergy2b}, and \eqref{eq:estenergy2c} we deduce that
\begin{equation}\label{eq:estenergy2}
\lim_{n\to\infty} \int_{I_n}  \left[ \frac{1}{\e_n}W(x,u_n(x)) + \e_n |\nabla u_nx)|^2 \right] \,d x = 0.
\end{equation}

We now consider the region
\[
U_n \coloneqq \bigcup_j \left( Q'(\overline{y_j}, e_N, r'_n)\times \left(\ell_n, \ell_n+\e_n\right) \right)
\]
and observe that
\begin{equation}\label{eq:Un}
\mathcal{L}^N(U_n) \le C \left( \frac{r'_n}{r_n} \right)^{N-1} \e_n.
\end{equation}
Let $i \in \{1,\dots,k\}$ be such that $U_n \subset \o^i$.
Then, for $x \in Q'(\overline{y_j}, e_N, r'_n)\times \left(\ell_n, \ell_n+\e_n\right)$ it holds
\begin{align}\label{eq:estW}
W(x,u_n(x)) & = W(x,u_n(x)) - W(x,z_i(x)) \nonumber \\
& \le \Lip(W;K) |u_n(x)-z_i(x)| \nonumber \\
& \le \Lip(W;K) |z_i(\overline{y_j}) - z_i(y, \ell_n+\e_n)| \nonumber \\
& \le C \Lip(W;K) \Lip(z_i) (\e_n^{\frac{5}{4}}+r_n^2)^{\frac{1}{2}},
\end{align}
where in the previous to last step we used the triangle inequality together with fact that $u_n$ is a linear interpolation between $z_i (\overline{y_j})$ and $z_i(y,\ell_n+\e_n)$. Here $K > 0$ is such that $u_n(x), z_i(x) \in B(0,K)$ for all $x \in Q'(\overline{y_j}, e_N, r'_n)\times \left(\ell_n, \ell_n+\e_n\right)$.
Therefore, from \eqref{eq:Lip2} and \eqref{eq:estW} we get
\begin{align}\label{eq:estenergy3}
&\int_{U_n} \left[ \frac{1}{\e_n}W(x,u_n(x)) + \e_n |\nabla u_n|^2 \right] \,d x \nonumber \\
&\hspace{2cm}\le C \left[ \sup_i \Lip(z_i) \frac{(\e_n^{\frac{5}{4}}+r_n^2)^{\frac{1}{2}}}{\e_n}
	+ \e_n \left( \frac{C}{r_n-r'_n} + \Lip(u^+)
	+ \frac{(\e_n^{\frac{5}{4}} + r_n^2)^{\frac{1}{2}}}{\e_n}
		\right)^2 \right] \mathcal{L}^N(U_n) \nonumber \\
&\hspace{2cm} \le C\left[  \frac{(\e_n^{\frac{5}{4}}+r_n^2)^{\frac{1}{2}}}{\e_n} + \frac{\e_n}{(r_n-r_n')^2} + \frac{\e_n^{\frac{5}{4}} + r_n^2}{\e_n} \right] \left(\frac{r'_n}{r_n}\right)^{N-1} \e_n \nonumber \\
&\hspace{2cm} \to 0
\end{align}
as $n\to\infty$, where the last step follows from \eqref{eq:rn} with analogous computations as those we used to deduce \eqref{eq:estenergy2a}.

With a similar argument it is possible to show that
\begin{equation}\label{eq:estenergy4}
\lim_{n\to\infty}\int_{L_n} \left[ \frac{1}{\e_n}W(x,u_n(x)) + \e_n |\nabla u_n(x)|^2 \right] \,dx  = 0, 
\end{equation}
where
\[
L_n \coloneqq \bigcup_j \left( Q'(\overline{y_j}, e_N, r'_n)\times(-\e_n, 0) \right).
\]
Moreover, using \eqref{eq:Sn}, \eqref{eq:estLipun}, and \eqref{eq:Linfty} we also get
\begin{equation}\label{eq:estenergy5}
\lim_{n\to\infty}\int_{S_n} \left[ \frac{1}{\e_n}W(x,u_n(x)) + \e_n |\nabla u_n(x)|^2 \right] \,d x  = 0.
\end{equation}

Finally, we prove that
\begin{equation}\label{eq:estenergy6}
\lim_{n\to\infty}\int_{G'_n} \left[ \frac{1}{\e_n}W(x,u_n(x)) + \e_n |\nabla u_n(x)|^2 \right] \,d x  = \f_0(u),
\end{equation}
where
\[
G'_n \coloneqq \bigcup_j  \left(  Q'(\overline{y_j}, e_N, r'_n) \times \left(0, \ell_n\right) \right) \setminus S_n.
\]
Thanks to \eqref{eq:estenergy2}, we can equivalently consider the region
\[
G_n \coloneqq \bigcup_j  \left( U^j_n \times \left(0, \ell_n\right) \right).
\]
Moreover, the Lipschitz continuity of $W$ gives
\begin{align}\label{eq:approx2}
&\lim_{n\to\infty}\frac{1}{\e_n} \int_{G_n} \left| W(x,u_n(x))
	- W(\overline{y_i},u_n(x)) \right| \,d x \nonumber \\
&\hspace{3cm}\le \lim_{n\to\infty}\frac{1}{\e_n}\sum_j 
	\int_{U^j_n\times(0, \ell_n)} |x-\overline{y_j}| \,dx \nonumber \\
&\hspace{3cm}\le \lim_{n\to\infty}\frac{1}{\e_n}\sum_j 
	\int_{B(\overline{y_j},\ell_n\sqrt{N})} |x-\overline{y_j}| \,dx \nonumber \\
&\hspace{3cm}=\lim_{n\to\infty}\frac{1}{\e_n}\sum_j
	\int_0^{\ell_n\sqrt{N}} t N\omega_N t^{N-1}\,dt \nonumber \\
&\hspace{3cm}\le \lim_{n\to\infty}\frac{1}{\e_n}\frac{C}{(r_n)^{N-1}}
	\frac{N\omega_N}{N+1}\left(\ell_n\sqrt{N}\right)^{N+1}  \nonumber \\
&\hspace{3cm}\leq \lim_{n\to\infty} C \e_n^{\frac{11}{24}N-\frac{5}{24}}= 0,
\end{align}
since $N\geq2$.
Therefore, using the fact that for every $j$ and every $n$
\[
\int_{\tau^j_n}^{\ell_n}
	\left[ \frac{1}{\e_n}W(\overline{y_j},  \gamma^j_n(g^j_n(t)) )
		+ \e_n |(\gamma^j_n(g^j_n))'(t)|^2 \right] \,dt = 0,
\]
we see that
\begin{align}\label{eq:approx4}
&\lim_{n\to\infty}\int_{G'_n} \left[ \frac{1}{\e_n}W(x,u_n(x))
	+ \e_n |\nabla u_n(x)|^2 \right] \,d x  \nonumber \\
&\hspace{2cm}= \lim_{n\to\infty} \int_{G_n} \left[ \frac{1}{\e_n}W(x,u_n(x))
	+ \e_n |\nabla u_n(x)|^2 \right] \,d x  \nonumber \\
&\hspace{2cm}= \lim_{n\to\infty}  \sum_j \int_{U^j_n\times \left(0,\ell_n\right)}
	\left[ \frac{1}{\e_n}W(\overline{y_j},u_n(x))
		+ \e_n |\nabla u_n(x)|^2 \right] \,d x  \nonumber \\
&\hspace{2cm}= \lim_{n\to\infty}  \sum_j \hno(U^j_n)\int_0^{\ell_n}
	\left[ \frac{1}{\e_n}W(\overline{y_j},  \gamma^j_n(g^j_n(t)) )
		+ \e_n |(\gamma^j_n(g^j_n))'(t) |^2 \right] \,dt  \nonumber \\
&\hspace{2cm}\leq \lim_{n\to\infty}
	\sum_j  \hno(U^j_n) d_W\left( \overline{y_j},
	u^+(\overline{y_j}), u^-(\overline{y_j}) \right) \nonumber \\
&\hspace{2cm}=\int_\Sigma d_W(x, u^+(x), u^-(x))\dhno(x),
\end{align}
where in the last step we used the fact that the function $x \mapsto d_W(x, u^+(x), u^-(x))$ is continuous on $\Sigma$ (see \Cref{dWisLipsc}), while the previous to last step follows from \eqref{eq:approxcurves} together with the result of \Cref{lem:constructionprofile}.

Finally, since the number of connected components of $(J_u\cap\o)\setminus S_n$ is finite, by \eqref{eq:estenergy1}, \eqref{eq:estenergy2}, \eqref{eq:estenergy3}, \eqref{eq:estenergy4}, \eqref{eq:estenergy5}, and \eqref{eq:estenergy6} we conclude.
\end{proof}

We are now in position to prove the limsup inequality.

\begin{proof}[Proof of \Cref{prop:limsup}]
For $u\in BV(\o;z_1,\dots,z_k)$, let $\{v_n\}_{n\in\N}$ be the sequence of functions in $BV(\o;z_1,\dots,z_k)$ provided by \Cref{lem:approxregular}. In particular, recall that each $v_n$ can be written as
\[
v_n = \sum_{i=1}^k z_i \ca_{\o^i_n},
\]
where each $\o^i_n$ is a polyhedral set, that $v_n \to u$ in $L^1(\o;\R^M)$, and that
\[
\lim_{n\to\infty} \f_0(v_n) = \f_0(u).
\]
By an application of \Cref{lem:approxenergy}, for each $n \in \N$ there exists a sequence $\{w^n_m\}_{m \in \N}$ of functions in $H^1(\o;\R^M)$ such that $w^n_m \to v_n$ in $L^1(\o;\R^M)$ and
\[
\lim_{m \to \infty}\f_n(w^n_m) = \f_0(v_n).
\]
Therefore, by using a diagonalization argument, it is possible to find a sequence $\{m_n\}_{n\in\N}$ such that the sequence $\{u_n\}_{n\in\N}$ of functions in $ H^1(\o;\R^M)$ defined as $u_n \coloneqq w^n_{m_n}$ is such that $u_n \to u$ in $L^1(\o;R^M)$ and
\[
\lim_{n\to\infty} \f_n(u_n) = \f_0(u).
\]
This concludes the proof.
\end{proof}


\section{Proofs of the variants of the main results} \label{sec:others}

In this section we discuss how to suitably modify our arguments to obtain the proofs for the several variants we consider.
In the following, in order to keep the notation as simple as possible, the value of the constant $C > 0$ might change from one instance to another.

\subsection{Proof of \Cref{thm:massconstr}}
To prove the compactness result, we notice that
\[
\f_n(u) \le \f_n^{\mathcal{M}}(u)
\]
for all $u \in L^1(\o; \R^M)$.
Therefore, for any sequence $\{u_n\}_{n \in \N} \subset L^1(\o; \R^M)$, we have that
\[
\sup \left\{\f_n^{\mathcal{M}}(u_n) : n \in \N \right\} < \infty \quad \Longrightarrow \quad
\sup \left\{ \f_n(u_n) : n \in \N \right\} < \infty.
\]
By an application of Proposition \ref{prop:cpt}, we get that up to the extraction of a subsequence (which we do not relabel)
$u_n \to u$ in $L^1(\o; \R^M)$, for some $u \in BV(\o; z_1, \dots, z_k)$ with $\f_0(u) < \infty$. Since
\[
\mathcal{M} = \lim_{n \to \infty} \int_\o u_n(x)\,dx = \int_\o u(x)\,dx,
\]
we also deduce that $\f_0^{\mathcal{M}}(u) < \infty$.

Since the proof of the liminf inequality remains unchanged, we omit the details. Next, we discuss how to construct the recovery sequence. Our approach is inspired by that of \cite{Ish_vec}. To be precise, let $u \in BV(\o; z_1, \dots, z_k)$ with
\[
\int_\o u(x)\,dx = \mathcal{M},
\]
and let $\{w_n\}_{n \in \N} \subset BV(\o; z_1, \dots, z_k)$ be the sequence provided by \Cref{lem:approxregular}. Without loss of generality, we can assume that
\begin{equation}\label{eq:wnu}
\int_\o |u(x) - w_n(x)|\,dx \le \e_n.
\end{equation}
For each $n\in\N$, by applying \Cref{lem:approxenergy} to the function $w_n$, it is possible to find let $v_n \in H^1(\o; \R^M)$ such that
\begin{equation}\label{eq:energywnvn}
\left|  \f_n(v_n) - \f_0(w_n) \right| \leq \varepsilon_n,
\end{equation}
and such that (see Step 1 and 2 in the proof of \Cref{lem:approxenergy})
\begin{equation}\label{eq:sigmamn}
\mathcal{L}^N \left( E_n \right) \le C \e_n^{\frac{5}{8}},
\end{equation}
for some constant $C > 0$ independent of $n$, where $E_n \coloneqq \{ x \in \o : v_n(x) \neq w_n(x) \}$.
Moreover, for $n \in \N$, let
\[
\eta_n \coloneqq \int_\o u(x)\,dx - \int_\o v_n(x)\,dx,
\]
and define
\[
\widetilde{v}_n(x) \coloneqq v_n(x) + \frac{\eta_n}{\mathcal{L}^N(\o)}.
\]
Notice that $\widetilde{v}_n$ satisfies the mass constraint, \emph{i.e.\@},
\[
\int_\o \widetilde{v}_n(x)\,dx = \mathcal{M}.
\]
Notice also that in view of \eqref{eq:wnu} and \eqref{eq:sigmamn} we get
\begin{equation}\label{eq:etanm}
|\eta_n| \leq C \e_n^{\frac{5}{8}}
\end{equation}
and so $\| \widetilde{v}_n - w_n \|_{L^1(\o;\R^M)}\leq C \e_n^{\frac{5}{8}}$.
Moreover, recalling that $p \mapsto W(x,p)$ is smooth for $x \in B(z_i(x), r)$ for all $i\in\{1, \dots, k\}$, that $\nabla_p W(x,z_i(x)) = 0$ for all $x \in \o$, and since $\|v_n\|_{L^\infty(\o; \R^M)} \le C$ for some constant $C > 0$ independent of $n$, an application of Taylor's formula yields
\begin{align*}
\f^{\mathcal{M}}_n(\widetilde{v}_n)& = \int_\o \left[ \frac{1}{\e_n}W(x,\widetilde{v}_n(x)) + \e_n |\nabla \widetilde{v}_n(x)|^2 \right] \,dx \\ 
& = \frac{1}{\e_n} \int_\o  \left[ W(x,v_n(x)) + \nabla_p W(x, v_n(x)) \cdot \eta_n  +\nabla^2_{pp} W(x,\xi_n(x))[\eta_n, \eta_n]   \right] \,dx \\
&\hspace{2cm}		+ \int_ \o \e_n |\nabla v_n(x)|^2 \,d x \\
&\leq  \f_n(v_n)
	+ \frac{|\eta_n| }{\e_n}\int_{E_n}  | \nabla_p W(x, v_n(x)) | \, d x
	+ \frac{|\eta_n|^2}{\e_n} \int_\o |\nabla^2_{pp} W(x,\xi_n(x)) | \, d x \\
&\leq \f_n(v_n)
	+ C\frac{|\eta_n| \mathcal{L}^N(E_n)}{\e_n} + C\frac{ |\eta_n|^2}{\e_n}.
\end{align*}
Here $|\xi_n(x) - v_n(x)| \le |\eta_n|$ for all $x\in\o$.
In view of the result of  \Cref{lem:approxregular} together with \eqref{eq:energywnvn}, \eqref{eq:sigmamn}, and \eqref{eq:etanm} we get
\[
\lim_{n\to\infty} \f^{\mathcal{M}}_n(\widetilde{v}_n) = \f^{\mathcal{M}}_0(u).
\]
This concludes the proof.
\qed

\begin{remark}
It would be interesting to understand if the mass constraint can play a role in weakening the assumptions on $W$. In this case, one might conjecture that the strategy implemented in \cite{Leo} would allow one to drop the assumption (\ref{H4}). It is not immediately clear whether this is possible since (\ref{H4}) plays a crucial role in obtaining the pivotal estimate (\ref{fniGrad}).
\end{remark}

\begin{remark}
For the proof of \Cref{thm:massconstr}, we opted to use a different approach than the one proposed in \cite{Bal}. The reason for this is that, to the best of our understanding, the argument in \cite{Bal} has a flaw that we were not able to correct.
\end{remark}


\subsection{Proof of Theorem \ref{thm:Dirichlet}}

In the proof of the compactness result we only need to modify the definition of $R'$ in \eqref{R'=} as follows:
\[
R'\coloneqq R_g + \Lambda_W(R_g),
\quad\text{ where }\quad
R_g \coloneqq \max\{ R,  \|g\|_{L^\infty(\o;\R^M)} \}.
\]
This is done in order to get \eqref{dW=dW1} and
\[
d_{W_1}(x,z_i(x), g(x)) = d_W(x,z_i(x), g(x))
\]
for all $i \in \{1,\dots,k\}$ and $x \in \partial \o$.

The liminf inequality is obtained as follows. Let $u \in BV(\o; z_1, \dots, z_k)$. Let $\widetilde{\o} \supset \overline{\o}$ be an open bounded set with Lipschitz continuous boundary and define the function $\widetilde{u} \in BV(\widetilde{\o}; \R^M)$ as
\begin{equation}\label{eq:utilde}
\widetilde{u}\coloneqq
\left\{
\begin{array}{ll}
u & \text{ in } \o\\
\widetilde{g} & \text{ in } \widetilde{\o} \setminus \o
\end{array}
\right.
\end{equation}
where $\widetilde{g} \in \Lip(\widetilde{\o} \setminus \o; \R^M)$ is such that $\widetilde{g} = g$ on $\partial \o$. We then conclude by repeating the argument in the proof of \Cref{prop:liminf} to the function $\widetilde{u}$.\\

The only change we need to make to the construction of the recovery sequence is in \Cref{lem:approxenergy}. Let $u \in BV(\o; z_1, \dots, z_k)$ be as in the statement of \Cref{lem:approxenergy}, and define $\widetilde{u} \in BV(\widetilde{\o}; \R^M)$ as in \eqref{eq:utilde}.
Extend $W$ and $z_1, \dots, z_k$ to locally Lipschitz functions defined in $\widetilde{\o} \times \R^M \times \R^M$ and $\widetilde{\o}$ respectively.
Set $z_{k+1} \coloneqq \widetilde{g}$.
By applying the construction in the proof of \Cref{lem:approxenergy} to $\widetilde{u}$ in $\widetilde{\o}$, we obtain a sequence $\{\widetilde{u}_n\}_{n \in \N} \subset H^1(\widetilde{\o}; \R^M)$ such that
\[
\lim_{n\to\infty} \int_\o |\widetilde{u}_n(x)-u(x)|\,dx = 0,
\quad\quad\quad\quad
\mathrm{Tr} (\widetilde{u}_n) = g \text{ on } \partial \o.
\]
Moreover, by looking at the energy estimates \eqref{eq:estenergy1}, \eqref{eq:estenergy2}, \eqref{eq:estenergy3}, \eqref{eq:estenergy4}, \eqref{eq:estenergy5}, and \eqref{eq:estenergy6}, we get
\[
\lim_{n\to\infty} \f_n^{D}(u_n) = \f_0(u) + \int_{\partial\o} d_W(x, \mathrm{Tr}\,u(x), g(x))\,d\hno(x).
\]
This concludes the proof.
\qed


\subsection{Proof of \Cref{noH2H3}}

Since the proof of the compactness result remains essentially unchanged, we omit the details. We only notice that since in \eqref{littleomega} one can have $\omega = 0$, we do not recover \eqref{measUij}, and thus we cannot in general conclude that $u \in BV(\o; z_1, \dots, z_k)$.

Recall that in view of \Cref{rem:counterexample}, if $u \in L^1(\o,\R^M)$ is such that $u(x) \in \{z_1(x), \dots, z_k(x)\}$ for $\mathcal{L}^N$-a.e.\@ $x\in\o$, the collection of sets $\left\{\{x \in \o : u(x) = z_i(x)\}\right\}_{i = 1}^k$ does not necessarily give a partition of $\o$. On the other hand, as one can readily check, it is possible to find sets $\o_1, \dots, \o_k$ which are measurable and pairwise disjoint and with the property that $u(x) = z_i(x)$ for $\mathcal{L}^N$-a.e.\@ $x \in \o^i$. In particular, this ensures that it is possible to write
\begin{equation}
\label{oipartition}
u = \sum_{i=1}^k z_i \ca_{\o^i}
\end{equation}
almost everywhere in $\o$, where $\{\o^i\}_{i=1}^k$ is a partition of $\o$ in sets that are not necessarily of finite perimeter.

The proof of the liminf inequality follows closely that of \Cref{prop:liminf}; the only changes required are described here in detail for the reader's convenience. For $x_0 \in J_u$ as in (\ref{goodpoint}), let $i_1, i_2 \in \{1, \dots, k\}$ be such that for every $\e > 0$ there exists $\bar{\rho}$ with the property that (\ref{x0jump}) holds for every $\rho \le \bar{\rho}$. Set
\begin{align*}
\mathcal{J}_1 & \coloneqq \left\{ j \in \{1, \dots, k\} : z_j(x_0) = z_{i_1}(x_0)\right\}, \\
\mathcal{J}_2 & \coloneqq \left\{ j \in \{1, \dots, k\} : z_j(x_0) = z_{i_2}(x_0)\right\}.
\end{align*}
With these notations at hand, reasoning as in (\ref{slicing1}) we arrive at
\begin{equation}
\label{slicing1.1}
\e \ge \sum_{j \notin \mathcal{J}_1} \left(\frac{\delta_1(x_0)}{\rho^N}\int_{Q^+(x_0,\nu, \rho)} \ca_{\o^j}(x)\,dx - \Lip(z_j)c(N)\rho\right), 
\end{equation}
where the sets $\o^j$ are defined as in (\ref{oipartition}) and $\delta_1(x_0) \coloneqq \min \left\{ |z_{i_1}(x_0) - z_j(x_0)| : j \notin \mathcal{J}_1 \right\}$. Arguing as in the proof of \Cref{prop:liminf} with (\ref{slicing1.1}) in place of (\ref{slicing1}) yields the following analogue to (\ref{-k/m}): 
\begin{align*}
\liminf_{n \to \infty}\mu_n(R_m) & \ge \sum_{i, j = 1}^k\int_{Q'_m(i,j)}d_{F_m}(z_i(x',r_m^+), z_j(x',r_m^-))\,dx' \\
& \ge \sum_{i \in \mathcal{J}_1} \sum_{j \in \mathcal{J}_2}\int_{Q'_m(i,j)}d_{F_m}(z_i(x',r_m^+), z_j(x',r_m^-))\,dx' - \frac{C \rho_m^{N - 1}}{m}.
\end{align*}
Notice that in view of (\ref{oipartition}) the sets $\{Q_m'(i,j)\}_{i,j = 1}^k$ are pairwise disjoint. Moreover, since for every $i \in \mathcal{J}_1$ and every $j \in \mathcal{J}_2$ it holds
\begin{equation}
\label{putx0.1}
d_{F_m}(z_i(x',r_m^+), z_j(x',r_m^-)) \ge d_{F_m}(z_{i_1}(x_0), z_{i_2}(x_0)) + \mathcal{O}(1),
\end{equation}
the rest of the proof follows without changes (see (\ref{putx0})). 

In the proof of the limsup inequality we need an additional approximation before applying \Cref{lem:approxregular} since, as priviously discussed in Remark \ref{rem:counterexample}, functions $u \in L^1(\o;\R^M)$ with $\widetilde{\f_0}(u) < \infty$ can be such that $\hno(J_u) = \infty$. The following result is adapted from \cite[Lemma 4.3]{Bou}.

\begin{lemma}\label{lem:approxsofp}
Let $u \in L^1(\o;\R^M)$ with $\widetilde{\f_0}(u) < \infty$ and $u(x)\in \{z_1(x), \dots, z_k(x)\}$ for $\mathcal{L}^N$-a.e.\@ $x\in\o$.
Then there exits a sequence
$\{u_n\}_{n \in \N} \subset BV(\o; z_1, \dots, z_k)$ with
\[
u_n = \sum_{i=1}^k z_i \ca_{\o^i_n},
\]
where $\{\o^i_n\}_{i=1}^k$ is Caccioppoli partition of $\o$, such that
$u_n \to u$ in $L^1(\o;\R^M)$ and
\[
\lim_{n\to\infty}\widetilde{\f_0}(u_n) 
=\lim_{n\to\infty} \sum_{i=1}^{k-1} \sum_{j=i+1}^k \int_{\partial^*\o^i_n\cap\partial^*\o^j_n} d_W(x,z_i(x),z_j(x)) \dhno(x)
\leq \widetilde{\f_0}(u).
\]
\end{lemma}

\begin{proof} We divide the proof into two steps. 
\newline
\textbf{Step 1:} Let $\{\o_i\}_{i = 1}^k$ be a partition of $\o$ such that (\ref{oipartition}) holds, and notice that we can assume without loss of generality that for every $i \in \{1, \dots, k\}$
\[
\mathcal{L}^N\left(\o^i \cap \{x \in \o : u(x) = z_j(x)\}\right) > 0 \quad \Longrightarrow \quad j \le i.
\]
For $i = 1, \dots, k - 1$ and $j = i + 1, \dots, k$ consider the Lipschitz function $\gamma_{ij} \colon \o \to [0,\infty)$ given by $\gamma_{ij}(x)\coloneqq |z_i(x) - z_j(x)|$ and, for $t > 0$, set
\[
E^t_{ij} \coloneqq \left\{ x \in \o : 0 < \gamma_{ij}(x) \le t \right\}.
\]
The co-area formula (see \cite[Theorem 2.93]{AFP}) gives us
\[
\int_0^{\infty} \hno(\partial^* E^t_{ij})\,dt = \int_\o |\nabla\gamma_{ij}(x)|\,dx < \infty,
\]
from which we infer that the function $t \mapsto \hno(\partial^* E^{t}_{ij})$ belongs to the space $L^1((0,\infty))$. In turn, it is possible to find a sequence $\{t_n\}_{n \in \N} \subset (0, \infty)$ with $t_n \searrow 0$ such that, for all $i = 1, \dots, k - 1$ and $j = i + 1, \dots, k$, we have
\begin{equation}\label{eq:goodsequencetn}
\partial^* E^{t_n}_{ij} = \{ x \in\o : \gamma_{ij}(x)=t_n \},
\end{equation}
and
\begin{equation}\label{eq:energyapproxgoestozero}
\lim_{n \to \infty} t_n \hno(\partial^* E^{t_n}_{ij})=0.
\end{equation}
For $i\in\{1,\dots,k\}$ and $n\in\N$ set $E^{t_n}_{ii}\coloneqq\emptyset$.
For $n \in \N$ define
\[
\o_n^1 \coloneqq \o^1 \cup \bigcup_{j=2}^k \left[   E^{t_n}_{1j}
	\cap  \left(  (\o^1\cup\o^j)  \setminus \bigcup_{r\neq 1,j} \left(  E^{t_n}_{1r} \cup  E^{t_n}_{jr} \right)  \right)  \right]
\]
and, for $i=2,\dots,k$, 
\[
\o_n^i \coloneqq \o^i \cup \bigcup_{j=i+1}^k \left[   E^{t_n}_{ij}
	\cap  \left(  (\o^i\cup\o^j)  \setminus \bigcup_{r>i, r\neq j} \left(  E^{t_n}_{ir} \cup  E^{t_n}_{jr} \right)  \right)  \right]
	\setminus \bigcup_{j=1}^{i-1} \o^j_n.
\]
Note that $\{\o^i_n\}_{i=1}^k$ is a partition of $\o$.
We claim that the sets $\o^i_n$ are of finite perimeter in $\o$.
To prove this, for each $n \in \N$, let $\omega_n > 0$ be such that
\begin{equation}\label{omegan}
\inf \left\{ d_W(x,z_i(x),z_j(x)) : x \in \o \setminus E^{t_n}_{ij} \right\} \ge \omega_n
\end{equation}
for all $i\neq j$. For $i \in \{1, \dots, k\}$ set $J^i_u \coloneqq \{ x\in J_u : u^-(x)=z_i(x) \}$ and note that
\begin{align*}
\widetilde{\f_0}(u) & = \int_{J_u} d_W(x,u^+(x), u^-(x)) \dhno(x) \\
&\geq \int_{J^i_u} d_W(x,u^+(x), u^-(x)) \dhno(x) \\
& \geq \int_{J_u^i \setminus \bigcup_{j\neq 1}E^{t_n}_{ij}} d_W(x,u^+(x), u^-(x)) \dhno(x) \\
& \geq \omega_n \hno\left(J_u^i\setminus \bigcup_{j\neq 1}E^{t_n}_{ij}\right) \\
& = \omega_n \sum_{r\neq i} \hno\left(\partial^{1/2}\o^i_n \cap \partial^{1/2} \o^r_n \setminus \bigcup_{j\neq 1}E^{t_n}_{ij}\right) \\
& = 2 \omega_n \hno\left(\partial^{1/2}\o^i_n \setminus \bigcup_{j\neq 1}E^{t_n}_{ij}\right),
\end{align*}
where in the last step we used the fact that $u(x)\in\{ z_1(x),\dots, z_k(x) \}$ for $\mathcal{L}^N$-a.e.\@ $x\in\o$.
Using the fact that each of the sets $E^{t_n}_{ij}$ has finite perimeter, and recalling that by assumption $\widetilde{\f_0}(u) < \infty$, we conclude.
\newline
\textbf{Step 2:} Define
\[
u_n \coloneqq \sum_{i=1}^k z_i \ca_{\o^i_n}.
\]
Since the sets $\o^i_n$ are of finite perimeter in $\o$, it follows from \cite[Theorem 3.84]{AFP} that $u_n \in BV(\o; z_1, \dots, z_k)$, and that it is possible to write
\begin{equation}\label{eq:writing}
\widetilde{\f_0}(u_n) = \sum_{i=1}^{k-1} \sum_{j=i+1}^k \int_{\partial^*\o^i_n\cap\partial^*\o^j_n} d_W(x,z_i(x),z_j(x)) \dhno(x).
\end{equation}
Using the fact that $\|z_i\|_{L^\infty(\o;\R^M)} \le C < \infty$, we deduce that
\[
\lim_{n\to\infty}\|u_n-u\|_{L^1(\o; \R^M)} \le \lim_{n\to\infty}C \sum_{i=1}^{k-1} \sum_{j=i+1}^k \mathcal{L}^N(E^{t_n}_{ij})=0.
\]
Moreover, by \eqref{S'} we infer that there exists $C<\infty$ such that
\begin{equation}\label{eq:boundness}
d_W(x,z_i(x), z_j(x)) \le C|z_i(x)-z_j(x)|,
\end{equation}
for all $x \in \o$ and all $i \neq j$.
Using \eqref{eq:energyapproxgoestozero}, \eqref{eq:writing}, and \eqref{eq:boundness}, we obtain
\begin{align*}
& \lim_{n\to\infty} |\widetilde{\f_0}(u_n) - \widetilde{\f_0}(u)| \\
&\hspace{1cm}\le \lim_{n\to\infty} \left[ C \sum_{i=1}^{k-1} \sum_{j=i+1}^k t_n \hno(\partial E^{t_n}_{ij})
	    +\int_{J_u\setminus E^{t_n}} d_W(x, u^+(x), u^-(x)) \dhno(x) \right] \\
&\hspace{1cm}\leq \lim_{n\to\infty} \int_{J_u\setminus E^{t_n}} d_W(x, u^+(x), u^-(x)) \dhno(x),
\end{align*}
where $E^{t_n} \coloneqq \bigcup_{i\neq j} E^{t_n}_{ij}$.
This concludes the proof.
\end{proof}

The rest of the proof of \Cref{noH2H3} follows by a diagonal argument without additional changes; therefore we omit the details. \qed

\subsection{Final remarks}\label{remH4}
We conclude this section by remarking that \Cref{prop:liminf} and \Cref{prop:limsup} can be proved independently of \Cref{prop:cpt}. Indeed, we only invoke our compactness result in \Cref{prop:liminf} to deduce that for a given $u \in L^1(\o;\R^M)$ the following hold:
\begin{itemize}
\item[$(i)$] the jump set $J_u$ is countably $\hno$-rectifiable;
\item[$(ii)$] if $\{u_n\}_{n\in\N} \subset H^1(\o;\R^M)$ with $\sup_{n\in\N} \f_n(u_n)<\infty$ is such that $u_n \to u$ in $L^1(\o;\R^M)$, then $u(x) \in \{z_1(x),\dots,z_k(x)\}$ for $\mathcal{L}^N$-a.e.\@ $x\in\o$.
\end{itemize}
Notice the recent result \cite{Del} implies that $(i)$ holds for all $u \in L^1_{\mathrm{loc}}(\o; \R^M)$, while $(ii)$ is readily derived reasoning as follows: without loss of generality, up to the extraction of a subsequence, we can assume $u_n(x) \to u(x)$ for $\mathcal{L}^N$-a.e.\@ $x \in \o$. Let
\[
A \coloneqq \left\{ x \in \o : u(x)\not\in \{z_1(x), \dots, z_k(x)\} \right\},
\]
and, arguing by contradiction, suppose that $\mathcal{L}^N(A) > 0$. Then Fatou's lemma gives
\[
0 < \int_A W(x,u(x))\,dx \leq \liminf_{n\to\infty} \int_A W(x,u_n(x))\,dx \leq \liminf_{n\to\infty }C \e_n = 0,
\]
where $C > 0$ is independent of $n$. Consequently, we see that assumption (\ref{H4}) is used only to deduce \eqref{Hold4}, which, together with (\ref{H1})--(\ref{H3}), is needed to ensure that we are in a position to apply \Cref{dFprop} and \Cref{dWisLipsc}. In conclusion, the $\Gamma$-convergence results of \Cref{statementMR} hold for a potential $W$ satisfying (\ref{H1})--(\ref{H3}) and (\ref{Hold4}) in place of (\ref{H4}).


\subsection*{Acknowledgements}
The authors acknowledge the Center for Nonlinear Analysis at Carnegie Mellon University (NSF PIRE Grant No.\@ OISE-0967140) where part of this work was carried out. The work of the first author has been supported by the National Science Foundation under Grant No.\@ DMS-1411646 of Irene Fonseca during the period at CMU, by Grant Nos.\@ EP/R013527/1 and EP/R013527/2 ``Designer Microstructure via Optimal Transport Theory" of David Bourne during the period at Heriot-Watt University. The research of the second author was partially funded by the National Science Foundation under Grant Nos.\@ DMS-1412095 and DMS-1714098 during the period at CMU, by the research support programs of Charles University under Grant Nos.\@ PRIMUS/19/SCI/01 and UNCE/SCI/023, and by the Czech Science Foundation (GA\v{C}R) under Grant No.\@ GJ17-01694Y. The authors would also like to thank Giovanni Leoni for his helpful insights.


\bibliographystyle{siam}
\bibliography{Bibliography}

\end{document}